\documentclass[10pt]{article}
\usepackage{graphicx}
\usepackage{amssymb,amsmath}
\usepackage{mathtools}
\usepackage{hyperref}
\usepackage[margin=1in,papersize={8.5in,11in}]{geometry}

\usepackage{times}
\usepackage[T1]{fontenc}
\usepackage[utf8]{inputenc}
\usepackage[english]{babel}
\usepackage{amsthm}
\usepackage{float}
\usepackage{amsfonts,mathrsfs}
\usepackage{bm}
\usepackage{epstopdf}
\usepackage[normalem]{ulem}

\usepackage{comment}
\usepackage[table]{xcolor}

\renewcommand{\Bbb}{\mathbb}

\newtheorem{theorem}{Theorem}[section]
\newtheorem{corollary}{Corollary}[theorem]

\makeatletter
\newcommand*{\rom}[1]{\expandafter\@slowromancap\romannumeral #1@}
\makeatother

\DeclareMathOperator{\Tr}{Tr}

\DeclareMathOperator{\Div}{div}

\DeclareMathOperator{\Image}{Im}

\newcommand{\LV}[1][]{\mathcal{L}_{#1}}

\newcommand{\Prj}{\mathcal{P}}
\newcommand{\PrjC}{\mathcal{Q}}

\newcommand{\PD}[1]{{{#1}_{\ast}}}   

\def\L{\mathcal{L}} 
\def\M{\mathcal{M}} 
\def\H{\mathcal{H}} 
\def\P{\mathcal{P}}
\def\Q{\mathcal{Q}}
\def\K{\mathcal{K}} 
\def\F{\mathcal{F}} 
\def\I{\mathcal{I}}

\def\Pi{\mathbf{P}}

\newcommand{\red}[1]{\textcolor{black}{#1}}

\begin{document}

\title{\red{On the estimation of the Mori-Zwanzig memory integral}}
\author{Yuanran Zhu, Jason M. Dominy and Daniele Venturi\thanks{Corresponding Author: venturi@ucsc.edu}\vspace{0.2cm}\\
\em Department of Applied Mathematics and Statistics\\ 
\em University of California, Santa Cruz}
\date{}
\maketitle
 
\begin{abstract}
We develop rigorous estimates and provably convergent 
approximations for the memory integral in the Mori-Zwanzig (MZ) 
formulation. The new theory is built upon rigorous mathematical 
foundations and is presented for both state-space and probability 
density function space formulations of the MZ equation. In particular,
we derive errors bounds and sufficient convergence conditions for 
short-memory approximations, the $t$-model, and hierarchical 
(finite-memory) approximations. \red{In addition, we derive computable upper bounds for the MZ memory integral, which allow us to estimate (a priori) the contribution of the MZ memory to the dynamics.} 
Numerical examples demonstrating convergence of the 
proposed algorithms are presented for linear and nonlinear dynamical 
systems evolving from random initial states.
\end{abstract}


\section{Introduction}
The Mori-Zwanzig (MZ) formulation is a technique from irreversible 
statistical mechanics that allows the development of formally exact 
evolution equations for quantities of interest such as macroscopic 
observables in high-dimensional dynamical systems
\cite{Zwanzig,Chorin2000OPMZ,VenturiBook,venturi2014convolutionless}. 
One of the main advantages of developing 
such exact equations is that they provide a 
theoretical starting point to avoid integrating 
the full (possibly high-dimensional) dynamical system 
and instead solve directly for the quantities of interest, 
thus reducing the computational cost significantly.
Computing the solution to the Mori-Zwanzig equation, 
however, is a challenging task that relies on approximations 
and appropriate numerical schemes.
One of the main difficulties lies in the approximation of the 
MZ memory integral (convolution term), which 
encodes the effects of the so-called orthogonal 
dynamics in the time evolution of the quantity of interest. 
The orthogonal dynamics is essentially a 
high-dimensional flow satisfying a complex 
integro-differential equation. 
Over the years, many techniques have been proposed 
for approximating the MZ memory integral, 
the most efficient ones being problem-dependent \cite{Snook,VenturiBook}. 
For example, in applications to statistical mechanics, 
Mori's continued fraction method \cite{mori1965continued,florencio1985exact} has been 
quite successful in determining exact solutions to several prototype 
problems, such as the dynamics of the auto-correlation function of a tagged 
oscillator in an harmonic chain 
\cite{espanol1996dissipative,kim2000dynamics}.
Other effective approaches to approximate 
the MZ memory integral rely on perturbation methods 
\cite{Breuer2001,venturi2014convolutionless,choventuri2014},  
mode coupling theories, \cite{alder1970decay,Snook}, 
and functional approximation methods \cite{VenturiPhysRep,Fox,Venturi_PRS,harp1970time, Moss1}. 
In a parallel effort, the applied mathematics 
community has, in recent years, attempted to derive general 
easy-to-compute representations of the memory integral. 
In particular, various approximations such as the $t$-model 
\cite{Chorin2000OPMZ,Chorin1,Stinis}, the modified $t$-model 
\cite{Chertock2008Modified} and, more recently, 
renormalized perturbation methods 
\cite{stinis2015renormalized} were 
proposed to address approximation of 
the memory integral in situations where there 
is no clear separation of scales between 
resolved and unresolved dynamics. \par

\red{The main objective of this paper is to develop rigorous estimates of the memory integral and provide convergence analysis of different 
approximation models of the Mori-Zwanzig equation, such as the short-memory 
approximation \cite{stinis2004stochastic},
the $t$-model \cite{Chorin1}, and hierarchical methods
\cite{stinis2007higher}. In particular, we study the MZ equation 
corresponding to two broad classes of projection operators: 
i) infinite-rank projections (e.g., Chorin's  projection \cite{Chorin2000OPMZ}) and ii) finite-rank projections (e.g., Mori's projection \cite{Mori}).}
We develop our analysis for both state-space and probability 
density function space formulations of the MZ equation. These 
two descriptions are connected by the same duality principle 
that pairs the Koopman and Frobenious-Perron 
operators \cite{Dominy2017}. \par
This paper is organized as follows. 
In section \ref{sec:MZ formulation},
we outline the general procedure
to derive the MZ equation in the phase space and discuss 
common choices of projection operators.
\red{In section \ref{sec:MZphasespace} 
we derive error bounds for the MZ memory integral 
and provide convergence analysis of different memory approximation methods, including the $t$-model \cite{Stinis,Chorin2000OPMZ,Chorin1}, the short-memory approximation  \cite{stinis2004stochastic}, and the hierarchical memory approximation technique \cite{stinis2007higher}. 
Such estimates are built upon 
semigroup estimation methods we present 
in Appendix \ref{sec:semigroupBoundsDecomposition}.}
In section \ref{sec:application} we present numerical 
examples demonstrating the accuracy of the memory 
approximation/estimation methods we develop 
throughout the paper. The main findings are summarized 
in section \ref{sec:summary}. 
Convergence analysis of the MZ memory 
term in the probability density function 
space formulation is presented in 
Appendix \ref{app:MZPDF}.

\section{The Mori-Zwanzig Formulation}
\label{sec:MZ formulation}
Consider the nonlinear dynamical system 
\begin{equation}
\frac{dx}{dt} = F(x),\qquad x(0)=x_0
\label{eqn:nonautonODE}
\end{equation}
evolving on a smooth manifold $\mathcal{S}$. 
For simplicity, let us assume that $\mathcal{S}=\mathbb{R}^n$.  We will consider the dynamics of scalar-valued observables $g:\mathcal{S}\to \mathbb{C}$, \red{and for concreteness, it will be desirable to identify structured spaces of such observable functions.  In \cite{Dominy2017}, it was argued that $C^{*}$-algebras of observables such as $L^{\infty}(\mathcal{S},\mathbb{C})$ (the space of all measureable, essentially bounded functions on $\mathcal{S}$) and $C_{0}(\mathcal{S},\mathbb{C})$ (the space of all continuous functions on $\mathcal{S}$, vanishing at infinity) make natural choices.  In what follows, we do not require the observables to comprise a $C^{*}$-algebra, but we will want them to comprise a Banach space as the estimation theorems of section \ref{sec:MZphasespace} make extensive use of the norm of this space.  Having the structure of a Banach space of observables also gives greater context to the meaning of the linear operators $\L$, $\K$, $\P$, and $\Q$ to be defined hereafter.}

The dynamics of any scalar-valued observable $g(x)$ (quantity of interest) 
can be expressed in terms of a semi-group $\K(t,s)$ of operators acting on the Banach space of observables.  This is the Koopman operator \cite{Koopman1931} which acts
 on the function $g$ as 
\begin{equation}
 g(x(t)) =\left[\K (t,s) g\right](x(s)), 
\end{equation}
where
\begin{equation}
\K (t,s)=e^{(t-s)\L},\qquad \L g(x)= F(x)\cdot \nabla g(x).
 \label{Koopman}
\end{equation}
 \red{Rather than compute the Koopman operator applicable to all observables, it is often more tractable to compute the evolution only of a (closed) subspace of quantities of interest.  This subspace can be described conveniently by means of a projection operator $\P$ with the subspace as its image.}  
Both $\P$ and the complementary projection $\Q=\I-\P$
act on the space of observables.
The nature, mathematical properties and connections 
between $\P$ and the observable $g$ are discussed 
in detail in \cite{Dominy2017}, \red{and summarized in section \ref{sec:projections}}.  
For now it suffices to assume  that $\P$ is 
a bounded linear operator, and that $\P^{2} = \P$. 
The MZ formalism describes the evolution of 
observables initially in the image of $\P$. 
Because the evolution of observables 
is governed by the semi-group $\K(t,s)$, 
we seek an evolution equation for $\K(t,s)\P$. By using 
the definition of the Koopman operator \eqref{Koopman}, and the 
well-known Dyson identity 
\begin{align*}
e^{t\L}=e^{t\Q\L}+
\int_0^t e^{s\L}\P\L e^{(t-s)\Q\L}ds
\end{align*}
we obtain \red{the operator equation
\begin{align}
\frac{d}{dt} e^{t\L} = e^{t\L}\P\L + 
e^{t\Q\L}\Q\L+ \int_0^t e^{s\L}\P\L e^{(t-s)\Q\L}\Q\L ds.
\label{MZKoop}
\end{align}
By applying this equation to an observable function $u_0$}, 
we obtain the well-known MZ equation in phase space
\begin{align}
\frac{\partial}{\partial t}e^{t\mathcal{L}}u_{0}
&=e^{t\mathcal{L}}\mathcal{PL}u_{0}
+e^{t\mathcal{QL}}\mathcal{QL}u_{0}+\int_0^te^{s\mathcal{L}}\mathcal{PL}
e^{(t-s)\mathcal{QL}}\mathcal{QL}u_{0} ds.
\label{MZstatespace}
\end{align}
Acting on the left with $\P$, we obtain the evolution equation for projected dynamics\footnote{Note that 
the second term in \eqref{MZstatespace}, i.e., 
$\mathcal{P}e^{t\mathcal{QL}}\mathcal{QL}x_0=0$ vanishes since 
$\P\Q=0$.}
\begin{align}\label{reduced order equation}
\frac{\partial}{\partial t}\mathcal{P}e^{t\mathcal{L}}u_{0}
=\mathcal{P}e^{t\mathcal{L}}\mathcal{PL}u_{0}
+\int_0^t\mathcal{P}e^{s\mathcal{L}}\mathcal{PL}
e^{(t-s)\mathcal{QL}}\mathcal{QL}u_{0} ds.
\end{align}

\subsection{Projection Operators}
\label{sec:projections}
\red{In this section, we make a summary on the commonly used projection operators $\P$ in the Mori-Zwanzig framework. To make our definition mathematically sound, we begin by assuming that the Liouville operator \eqref{Koopman} acts on observable functions in a $C^*$-algebra $\mathfrak{A}$, for instance $L^{\infty}(\M,\Sigma,\mu)$, where $\M$ is a space such as $\mathbb{R}^{N}$, $\Sigma$ is a $\sigma$-algebra on $\M$, and $\mu$ is a measure on $\Sigma$. Let $\sigma\in\mathfrak{A}_{*}$ be a positive linear functional on $\mathfrak{A}$. We define the weighted pre-inner product
\begin{align*}
\langle f,g \rangle_{\sigma}:=\sigma(f^*g).
\end{align*}
This can be used to define a Hilbert space $\H=L^2(\M,\sigma)$, which is the completion of the quotient space $$\H'=\{f\in\mathfrak{A}:\sigma(f^*f)<\infty\}/\{f\in\mathfrak{A}:\sigma(f^*f)=0\}$$ endowed with the inner product $\langle\cdot,\cdot \rangle_{\sigma}$. The $L^2$ norm induced by the inner product is denoted as $\|\cdot\|_{\sigma}$. In the rest of the paper, the positive linear functional $\sigma$ is always induced by a probability distribution $\tilde \sigma$ through 
\begin{align*}
\sigma(u)=\int_{\M} \tilde \sigma(\omega) u(\omega) d\omega,
\end{align*}
where $\tilde \sigma$ is typically chosen to be the probability density of the initial condition $\rho_0$, or the equilibrium distribution $\rho_{eq}$ in statistical physics. To conform to the literature, we also use notation $\langle\cdot,\cdot\rangle_{\rho_0}$, $\langle\cdot,\cdot\rangle_{\rho_{eq}} ,\langle\cdot,\cdot\rangle_{eq}$ to represent the weighted inner product corresponding to different probability measures $\tilde\sigma(\omega) d\omega$. With the Hilbert space determined, we now focus on the following two broad class of orthogonal projections on $\H$. }

\red{
\subsubsection{Infinite-Rank Projections} The first class of projection operators to consider in this setting are the conditional expectations $\P$ such that $\P_*\sigma=\sigma$. In this case, the properties of conditional expectations (in particular that $\P[\P(f)g\P(h)]=\P(f)\P(g)\P(h)$ \cite{Umegaki1954}) and the fact that $\P_{*}\sigma=\sigma$ imply that 
\begin{align*}
\langle\P f,g\rangle_{\sigma}&
=\sigma[(\P f)^*g]
=\P_*(\sigma)[(\P f)^*g]=\sigma[\P((\P f)^*g)]=\sigma[(\P f)^*(\P g)]\\
\langle f,\P g\rangle_{\sigma}&
=\sigma[f^*\P g]=\P_*(\sigma)[f^*\P g]=\sigma[\P(f^*\P g)]=\sigma[(\P f^*)(\P g)]=\sigma[(\P f)^*(\P g)]
\end{align*}
so that
\begin{align*}
\langle\P f,g\rangle_{\sigma}=\langle f,\P g\rangle_{\sigma}
\end{align*}
for all $f,g\in\H$.  It follows that
\begin{align*}
\langle \Q f,g\rangle_{\sigma}=\langle f,g\rangle_{\sigma}-\langle \P f,g\rangle_{\sigma}
=\langle f,g\rangle_{\sigma}-\langle  f,\P g\rangle_{\sigma}=\langle f,\Q g\rangle_{\sigma}.
\end{align*}
Therefore both $\P$ and $\Q$ are self-adjoint (i.e. orthogonal) projections onto closed subspaces of $\H$, hence contractions $\|\P\|_{\sigma}\leq 1$, $\|\Q\|_{\sigma}\leq 1$. Chorin's projection \cite{Chorin1,Chorin2000OPMZ} is one of this class, and is defined as 
\begin{align}
\big(\mathcal{P}g\big)(\hat {x}_{0})=\frac{\displaystyle \int_{-\infty}^{+\infty} g(\hat {x}(t;\hat{x}_0,\tilde{x}_0),\tilde {x}(t;\hat{x}_0,\tilde{x}_0))\rho_0(\hat {x}_0,\tilde {x}_0)d\tilde{x}_0}{\displaystyle \int_{-\infty}^{+\infty}\rho_0(\hat {x}_0,\tilde {x}_0)d\tilde{x}_0}
=
\Bbb{E}_{\rho_{0}}[g|\hat{x}_0].
\label{Chorin_projection}
\end{align} 
Here $x(t;x_0)=(\hat {x}(t;\hat{x}_0,\tilde{x}_0),\tilde {x}(t;\hat{x}_0,
\tilde{x}_0))$ is the flow map generated by \eqref{eqn:nonautonODE} 
split into resolved ($\hat 
{x}$) and unresoved ($\tilde{x}$) variables, and $g(x)=g(\hat {x},\tilde {x})$ is 
the quantity of interest. For Chorin's projection, the positive functional $\sigma$ defining the Hilbert space $\H$ may be taken to be integration with respect to the probability measure $\rho_0(\hat x_0,\tilde x_0)$. 
}
Clearly, if $x_0$ is deterministic then $\rho_0(\hat x_0,\tilde x_0)$ is a 
product of Dirac delta functions. 
On the other hand, if $\hat{x}(0)$ and 
$\tilde{x}(0)$ are statistically independent, i.e. 
$\rho_{0}(\hat{x}_0,\tilde{x}_0) = \hat{\rho}_{0}(\hat{x}_0)\tilde{\rho}_{0}(\tilde{x}_0)$, then the conditional expectation $\P$ simplifies to
\begin{align}
\big(\mathcal{P}u\big)(\hat {x}_0)=\int_{-\infty}^{+\infty}u(\hat {x}(t;\hat {x}_0,\tilde{x}_0),\tilde {x}(t;\hat {x}_0,\tilde{x}_0))\tilde{\rho}_0(\tilde {x}_0)d\tilde{x}_0.
\label{8}
\end{align} 
In the special case where $u(\hat {x},\tilde {x})= \hat {x}(t;\hat {x}_0,\tilde{x}_0)$ we have 
 \begin{align}
\big(\mathcal{P} \hat{x}\big)(\hat {x}_0)=\int_{-\infty}^{+\infty}\hat {x}(t;\hat {x}_0,\tilde{x}_0)\tilde{\rho}_0(\tilde {x}_0)d\tilde{x}_0,
\label{9}
\end{align} 
i.e., the conditional expectation of the resolved 
variables $\hat{x}(t)$ given the initial condition $\hat{x}_0$.  
This means that an integration of \eqref{9} 
with respect to $\hat{\rho}_0(\hat{x}_0)$ 
yields the mean of the resolved variables, i.e., 
\begin{equation}
\mathbb{E}_{\rho_0}[\hat{x}(t)] = \int_{-\infty }^\infty \big(\mathcal{P} \hat{x}\big)(\hat {x}_0)\hat{\rho}_0(\hat{x}_0)d\hat{x}_0 = \int_{-\infty }^\infty \hat{x}(t,x_{0})\rho_{0}(x_{0})dx_{0}.
\end{equation}
Obviously, if the resolved variables $\hat{x}(t)$ 
evolve from a deterministic initial state $\hat{x}_0$ then 
the conditional expectation \eqref{9} represents 
the the average of the reduced-order flow 
map $\hat{x}(t;\hat{x}_0,\tilde{x}_0)$ with respect to the 
PDF of $\tilde{x}_0$, i.e., the flow map
\begin{equation}
\P e^{t\L}\hat x(0)=X_0(t;\hat{x}_0)=\int_{-\infty}^{+\infty}\hat {x}(t;\hat {x}_0,\tilde{x}_0)\tilde{\rho}_0(\tilde {x}_0)d\tilde{x}_0.
\label{conditonal mean path}
\end{equation} 
In this case, the MZ equation \eqref{reduced order equation} is an exact (unclosed) evolution equation (PDE) for the multivariate 
field $X_0(t,\hat{x}_0)$. 
In order to close such an equation, a mean field approximation 
of the type $\P f(\hat{x})=f(\P\hat{x})$ was introduced by Chorin 
{\em et al.} in \cite{Chorin2000OPMZ,Chorin1,Chorin2000Hamiltonian}, together with the assumption that the probability 
distribution of $x_0$ is invariant under the flow 
generated by \eqref{eqn:nonautonODE}.

\red{
\subsubsection{Finite-Rank Projections} Another class of projections is 
defined by choosing a closed (typically finite-dimensional) linear 
subspace $V\subset \H=L^2(\M,\sigma)$ and letting $\P$ be the orthogonal projection onto $V$ in the $\sigma$ inner product. An example of such 
projection is Mori's projection \cite{zwanzig2001nonequilibrium}, widely used in statistical physics. For finite-dimensional $V$, given a linearly independent set $\{u_1,...,u_M\}\subset V$ that spans $V$, $\P$ can be defined by first constructing the 
positive definite Gram matrix $G_{ij}=\langle u_i,u_j\rangle_{\sigma}$. Then 
\begin{align}
\P f=\sum_{i,j=1}^M(G^{-1})_{ij}\langle u_i,f\rangle_{\sigma}u_j.
\label{MoriProjection}
\end{align}
This projection, is orthogonal with respect to the $L^2_{\sigma}$ inner 
product. In statistical physics, a common choice for the positive functional $\sigma$ that generates $\H$ is integration with respect to the Gibbs canonical distribution  $\rho_{eq}=e^{-\beta\H}/Z$, for the Hamiltonian $\H=\H(p, q)$ and the associated partition function $Z$. Here $q$ are generalized coordinates while $p$ are kinetic momenta.
}
 
\section{Analysis of the Memory Integral}
\label{sec:MZphasespace}

In this section, we develop a thorough mathematical 
analysis of the MZ memory integral 
\begin{align} \label{MemoryPhaseSpace}
\int_0^t\mathcal{P}e^{s\mathcal{L}}\mathcal{PL}
e^{(t-s)\mathcal{QL}}\mathcal{QL}u_{0}ds=\int_0^t\mathcal{P}e^{s\mathcal{L}}\mathcal{P}e^{(t-s)\mathcal{LQ}}\L\mathcal{QL}u_{0}ds.
 \end{align}
and its approximation. 
We begin by describing the behavior of the semigroup norms $\|e^{t\LV}\|$, $\|e^{t\Q\L\Q}\|$, and $\|e^{t\LV\PrjC}\|$ as functions of time, for different choices of projection $\Prj$ and different norms.  As we will see, the analysis will give clear computable bounds only in some circumstances, illustrating the difficulty of this problem and the need for further development and insight.

\red{
\subsection{Semigroup Estimates}
\label{sec:semigroup_estimation}
For any Liouville operator $\LV$ of the form \eqref{Koopman} acting on $\mathfrak{A} = L^{\infty}(\mathcal{M},\Sigma,\mu)$ and for any $\sigma$ identified with an element of $L^{1}(\mathcal{M},\Sigma,\mu)$, the functional $\PD{\LV}\sigma$ (assuming $\sigma$ lies in the domain of $\PD{\LV}$) is absolutely continuous with respect to $\sigma$ (essentially because $\LV$ acts locally) and the Radon-Nikodym derivative \cite{Gudder1979} may be identified with the negative of the divergence of the vector field ${F}$ with respect to the measure induced by $\sigma$, i.e., 
  \begin{align}
 	 \frac{d\PD{\LV}\sigma}{d\sigma} = -\Div_{\sigma}{F}, 
 	 \qquad \text{i.e.,} \qquad (\PD{\LV}\sigma)(u)  = -\sigma(u\Div_{\sigma}{F}), 
  \end{align}
where  
\begin{equation}
\Div_{\sigma}{F} = \frac{\nabla\cdot(\tilde{\sigma}(x)F(x))}{\tilde{\sigma}(x)}.
\end{equation}
When $\mathcal{M} = \mathbb{R}^{N}$ and $\sigma$ has the form $\sigma(u) = \int_{\mathbb{R}^{N}}\tilde{\sigma}({x})u({x})d{x}$, this can be shown more directly using integration by parts. By assuming that $\tilde{\sigma}(x)$ or $F_i$ decays to $0$ at $\infty$, we have 
 \begin{subequations}
 \begin{align}
 	\PD{\LV}\sigma(u) = \sigma(\LV u)&= \int_{\mathbb{R}^{N}}\tilde{\sigma}\sum_{i=1}^NF_{i}\frac{\partial u}{\partial x_{i}}d{x} = -\int_{\mathbb{R}^{N}}u\sum_{i=1}^N\frac{\partial}{\partial x_{i}}(\tilde{\sigma}F_{i})d{x}
 	 = -\int_{\mathbb{R}^{N}}u\left[\frac{1}{\tilde{\sigma}}\sum_{i=1}^N\frac{\partial}{\partial x_{i}}(\tilde{\sigma}F_{i})\right]\tilde{\sigma}d{x}\\
 	& = \sigma\left(u\left[-\frac{1}{\tilde{\sigma}}\sum_{i=1}^N\frac{\partial}{\partial x_{i}}(\tilde{\sigma}F_{i})\right]\right)
 	= \sigma\left(u\left[-\Div_{\sigma}{F}\right]\right)
\end{align}
\end{subequations}
}
\red{
from which we see
\begin{align*}
  	\Div_{\sigma}{F} = \frac{\nabla\cdot (\tilde{\sigma}{F})}{\tilde{\sigma}} = \frac{1}{\tilde{\sigma}}\sum_{i=1}^{N}\frac{\partial}{\partial x_{i}}(\tilde{\sigma}F_{i}) = \nabla \cdot {F} + {F}\cdot\nabla \left(\ln\tilde{\sigma}\right).
 \end{align*}
 Therefore,
 \begin{align*}
	 \langle v,-u\Div_{\sigma}{F}\rangle_{\sigma} = \sigma(-v^{*}u\Div_{\sigma}{F}) = (\PD{\LV}\sigma)(v^{*}u) = \sigma(\LV(v^{*}u)) = \sigma(\LV(v)^{*}u + v^{*}\LV(u)) = \langle \LV v, u\rangle_{\sigma} + \langle v, \LV u\rangle_{\sigma},
 \end{align*}
 and the following are equivalent: i) $\PD{\LV}\sigma = 0$ (i.e., $\sigma$ is invariant); ii) $\Div_{\sigma}{F} = 0$; iii) $\LV$ is skew-adjoint with respect to the $\sigma$ inner product. More generally, on $\mathcal{H} = L^{2}(\mathcal{M},\sigma)$, we find that 
\begin{align*}
	\LV + \LV^{\dag}  = -\Div_{\sigma}{F}
\end{align*}
so that the numerical abscissa $\omega$ \cite{Trefethen1997, Trefethen2005} (i.e., logarithmic norm \cite{Davies2005, Soederlind2006}) of $\LV$ is given by
\begin{align}\label{numerical_abscissa}
	\omega := \sup_{0\neq u\in\mathcal{H}}\frac{\Re\langle u, \LV u\rangle_{\sigma}}{\langle u, u\rangle_{\sigma}} = \sup_{0\neq u\in\mathcal{H}}\frac{\langle u, (\LV + \LV^{\dag}) u\rangle_{\sigma}}{2\langle u, u\rangle_{\sigma}} = \sup_{0\neq u\in\mathcal{H}}\frac{\langle u, - u\Div_{\sigma}{F}\rangle_{\sigma}}{2\langle u, u\rangle_{\sigma}} = -\frac{1}{2}\inf_{{x}}\Div_{\sigma}{F}({x}).
\end{align}
Using this numerical abscissa $\omega$, we obtain the following $L^2_{\sigma}$ estimation of the Koopman semigroup
\begin{align}
\|e^{t\LV}\|_{L^2_{\sigma}} \leq e^{\omega t}
\label{koopmanEstimate}
\end{align}
and moreover, $e^{\omega t}$ is the smallest exponential function that bounds $\|e^{t\LV}\|_{\sigma}$ \cite{Davies2005}.  When $\Prj$  and 
$\PrjC = \I - \Prj$ are orthogonal projections on $L^{2}(\mathcal{M},\sigma)$, we can observe that the numerical abscissa of $\PrjC\LV\PrjC$ is bounded by that of $\LV$. In fact, 
\begin{align*}
	\sup_{0\neq u\in\mathcal{H}}\frac{\Re\langle u, \PrjC\LV\PrjC u\rangle_{\sigma}}{\langle u, u\rangle_{\sigma}} = \sup_{0\neq u\in\mathcal{H}}\frac{\Re\langle \PrjC u, \LV\PrjC u\rangle_{\sigma}}{\langle u, u\rangle_{\sigma}} = \sup_{0\neq u\in\Image{\PrjC}}\frac{\Re\langle u, \LV u\rangle_{\sigma}}{\langle u, u\rangle_{\sigma}} \leq \sup_{0\neq u\in\mathcal{H}}\frac{\Re\langle u, \LV u\rangle_{\sigma}}{\langle u, u\rangle_{\sigma}} = \omega,
\end{align*}
(see equation \eqref{numerical_abscissa}) so that 
\begin{align}\label{orth_semigroup_bound}
	\|e^{t\PrjC\LV\PrjC}\|_{L^2_{\sigma}} \leq e^{\omega t}.
\end{align}
It should be noticed that this bound for the orthogonal semigroup is not necessarily tight. The tightness of the bound depends on the choice of projection $\P$ and comes down to whether functions in the image of $\Q$ can be chosen localized to regions where $\Div_\sigma F$ is close to 
its infimal value.}\par
\red{When estimating the MZ memory integral, we need to deal with the semigroup $e^{t\L\Q}$.  It turns out to be extremely difficult to prove strong continuity of such semigroup in general, due to the unboundedness of $\L\Q$. 
It is shown in Appendix \ref{sec:semigroupBoundsDecomposition} that, when $\Prj\LV\PrjC$ is an unbounded operator, as is typical when $\Prj$ is a conditional expectation, the semigroup $e^{t\LV\PrjC}$ can only be bounded as
\begin{align}
	\|e^{t\LV\PrjC}\|_{{\sigma}} \leq M_{\Q}e^{t\omega_{\Q}}
\end{align}
for some $M_{\Q} > 1$, due to the fact that $\|e^{t\L\Q}\|_{\sigma}$ has infinite slope at $t =0$.  More work is needed to obtain satisfactory, computable values for $M_{\PrjC}$ and $\omega_{\PrjC}$, in the case where  $\Prj$ and $\PrjC$ are infinite-rank projections.  It is also shown that, when either $\Prj\LV\PrjC$ or $\LV\Prj$ is bounded, for example when $\Prj$ is a finite-rank projection, we can get computable semigroup bounds of the form
\begin{subequations}
\begin{align}
	\|e^{t\LV\PrjC}\|_{\sigma} & \leq {e^{\omega_{\PrjC}t}} \leq e^{\frac{1}{2}\left(\sqrt{\omega^{2} + \|\Prj\LV\PrjC \|_{\sigma}^{2}} + \omega\right)t},\\
	\|e^{t\LV\PrjC}\|_{\sigma} & \leq {e^{\omega_{\PrjC}t}} \leq e^{(\omega + \|\LV\Prj\|_{\sigma})t},	
\end{align}
\end{subequations}
where $\omega = -\inf\Div_{\sigma}{F}$ if we use the $L^2_{\sigma}$ estimation of $e^{t\L}$.
}

\subsection{Memory Growth}
We begin by seeking to bound the MZ 
memory integral  \eqref{MemoryPDFSpace}
as a whole and build our analysis from there.
A key assumption of our analysis is that the semigroup 
$e^{t\mathcal{LQ}}$ is strongly continuous\footnote{
As is well known, $e^{t\mathcal{L}}$ (Koopman operator) 
is typically strongly continuous \cite{engel1999one}. 
However, no such result exists for $e^{t\mathcal{LQ}}$.}, 
i.e., the map $t\mapsto e^{t\mathcal{LQ}}g$ 
is continuous in the norm topology on the space 
of observables for each fixed  $g$ \cite{engel1999one}.
However, as we pointed out in section \ref{sec:semigroup_estimation}, it is a difficult task both to prove strong continuity of $e^{t\mathcal{LQ}}$ and to obtain a computable upper bound for unbounded generators of the form $\L\Q$, we leave this as an open problem and assume that there exist constants $M_{\Q}$ and $\omega_{\Q}$ such that $\|e^{t\L\Q}\|\leq M_{\Q}e^{t\omega_{\Q}}$.
Throughout this section, $\|\cdot\|$ denotes a general Banach norm. We begin with the following simple estimate:

\begin{theorem}
\label{memory_theorem}
{\bf (Memory growth)}
Let $e^{t\mathcal{L}}$ and $e^{t\mathcal{LQ}}$ be 
strongly continuous semigroups with upper bounds $\|e^{t\L}\|\leq Me^{t\omega}$ and $\|e^{t\L\Q}\|\leq M_{{\Q}}e^{t\omega_{\Q}}$. Then 
\begin{align}\label{memory_estimation}
\left\|\int_0^t\mathcal{P}e^{s\mathcal{L}}\mathcal{PL}
e^{(t-s)\mathcal{QL}}\mathcal{QL}u_{0}ds\right\|\leq M_0(t),
\end{align}
where 
\begin{align}
	M_0(t)=
	\begin{cases}
	\displaystyle C_1te^{t\omega_{\Q}},\quad &\omega=\omega_{\Q}\\
\displaystyle \frac{C_1}{\omega-\omega_{\Q}}[e^{t\omega}-e^{t\omega_{\Q}}],\quad &\omega\neq\omega_{\Q}
	\end{cases}
\end{align} 
and $C_{1} = MM_{\Q}\|\L\Q\L u_{0}\|$ is a 
constant. Clearly, $\displaystyle \lim_{t\rightarrow 0} M_0(t)=0$.
\end{theorem}
\begin{proof}
We first rewrite the memory integral in the 
equivalent form
\begin{align*}
\int_0^t\mathcal{P}e^{s\mathcal{L}}\mathcal{PL}
e^{(t-s)\mathcal{QL}}\mathcal{QL}u_{0}ds
=\int_0^t\mathcal{P}e^{s\mathcal{L}}\mathcal{P}
e^{(t-s)\mathcal{LQ}}\L\mathcal{QL}u_{0}ds.
\end{align*}
Since $e^{t\L}$ and $e^{t\L\Q}$ are assumed to be strongly continuous semigroups, we have the upper bounds $\|e^{t\L}\|\leq Me^{t\omega}$, $\|e^{t\L\Q}\|\leq M_{\Q}e^{t\omega_{\Q}}$. Therefore 
\begin{align*}
\bigg\|\int_0^t\mathcal{P}e^{s\mathcal{L}}\mathcal{P}
e^{(t-s)\mathcal{LQ}}\L\mathcal{QL}u_{0}ds\bigg\|
&\leq
\int_0^t\|e^{s\mathcal{L}}\mathcal{P}
e^{(t-s)\mathcal{LQ}}\L\mathcal{QL}u_{0}\|ds\\
&\leq
MM_{\Q}\|\mathcal{LQL}u_0\|
\int_0^te^{s(\omega-\omega_{\Q})}ds\\
&=
\begin{cases}
\displaystyle C_1te^{t\omega_{\Q}},\quad &\omega=\omega_{\Q}\\
\displaystyle \frac{C_1}{\omega-\omega_{\Q}}[e^{t\omega}-e^{t\omega_{\Q}}],\quad &\omega\neq\omega_{\Q}
\end{cases}
\end{align*}
where $C_1=MM_{\Q}\|\P\|^2\|\mathcal{LQL}u_0\|$.
\end{proof}
\red{
Theorem \ref{memory_theorem} provides an upper bound for the growth of the memory integral based 
on the assumption that $e^{t\mathcal{L}}$ and $e^{t\mathcal{LQ}}$ are strongly continuous semigroups. We emphasize that only for simple cases can such upper bounds can be computed analytically (we will compute one of the cases later in section \ref{sec:application}), because of the fundamental difficulties in computing the upper bound of $e^{t\L\Q}$. However, it will be shown later that, although the specific expression for $M_0(t)$ is unknown, the {\em form} of it is already useful as it enables us to derive some verifiable theoretical predictions for general nonlinear systems.}

\subsection{Short Memory Approximation and the \texorpdfstring{$t$}{t}-model}
Theorem \ref{memory_theorem} can be employed to obtain 
upper bounds for well-known approximations of the memory integral. 
Let us begin with the $t$-model proposed in \cite{Chorin1}. This model 
relies on the approximation 
\begin{align}\label{t_model_operator}
\int_0^t\mathcal{P}e^{s\mathcal{L}}\mathcal{PL}
e^{(t-s)\mathcal{QL}}\mathcal{QL}u_{0}ds\simeq te^{t\mathcal{L}}\mathcal{PLQL}u_0
\qquad \textrm{($t$-model).} 
\end{align}

\begin{theorem}\label{t-model_estimation}
{\bf (Memory approximation via the $t$-model \cite{Chorin1})}
Let $e^{t\mathcal{L}}$ and $e^{t\mathcal{LQ}}$ be strongly 
continuous semigroups with upper bounds $\|e^{t\L}\|\leq Me^{t\omega}$ and $\|e^{t\L\Q}\|\leq M_{{\Q}}e^{t\omega_{\Q}}$. Then
\begin{align*}
\bigg\|\int_0^t \mathcal{P}e^{s\mathcal{L}}\mathcal{PL}e^{(t-s)\mathcal{LQ}}
\mathcal{LQL}u_0ds - t\mathcal{P}e^{t\L}\L\Q\L u_0\bigg\|
\leq M_1(t),
\end{align*}
where 
\begin{align*}
M_{1}(t) = 
\begin{cases}
\displaystyle C_{1}\left(\frac{e^{t\omega_{\Q}} - e^{t\omega}}{\omega_{\Q}-\omega} + \frac{te^{t\omega}}{M_{\Q}}\right) & \omega \neq \omega_{\Q}\\
 \displaystyle  C_{1}\frac{M_{\Q}+1}{M_{\Q}}te^{t\omega} & \omega = \omega_{\Q}\end{cases}, 
\end{align*}
and $C_{1} = MM_{\Q}\|\P\|^2\|\L\Q\L u_{0}\|$.
\end{theorem}
\begin{proof}
By applying the triangle inequality, we obtain that 
\begin{align*}
\left\|\int_{0}^{t} \P e^{s\L}\P e^{(t-s)\L\Q}\L\Q\L u_{0}ds - t\P e^{t\L}\P \L\Q\L u_{0}\right\|
  & \leq \left(\int_{0}^{t}\|\P\|\left\|e^{s\L}\right\|\left\|\P\right\|\left\|e^{(t-s)\L\Q}\right\|ds + t\|\P\|\left\|e^{t\L}\right\|\left\|\P\right\|\right)\left\|\L\Q\L u_{0}\right\|\\
  & \leq \|\P\|^2\left\|\L\Q\L u_{0}\right\|\left(MM_{\Q}\int_{0}^{t}e^{s\omega}e^{(t-s)\omega_{\Q}}ds + tMe^{t\omega}\right)\\
  & = C_{1}e^{t\omega}\left(\int_{0}^{t}e^{s(\omega_{\Q}-\omega)}ds + \frac{t}{M_{\Q}}\right)\\
  &= \begin{cases}
  \displaystyle 
 C_{1}\left(\frac{e^{t\omega_{\Q}} - e^{t\omega}}{\omega_{\Q}-\omega} + \frac{te^{t\omega}}{M_{\Q}}\right) & \omega \neq \omega_{\Q}\\
  \displaystyle  C_{1}\frac{M_{\Q}+1}{M_{\Q}}te^{t\omega} & \omega = \omega_{\Q}
  \end{cases}
\end{align*}
where  $C_1=MM_{\Q}\|\P\|^2\|\mathcal{LQL}u_0\|$.
\end{proof}

\noindent
Theorem \ref{t-model_estimation} provides an upper bound for 
the error associated with the $t$-model. The limit 
\begin{equation}
 \lim_{t\rightarrow0}M_1(t)=0,
\end{equation}
guarantees the convergence of the $t$-model for short integration times. 
On the other hand, depending on the semigroup 
constants $M$, $\omega$, $M_\Q$ and $\omega_\Q$ (which may be estimated numerically), the error of the $t$-model may 
remain small for longer integration times (see the numerical results
in section \ref{sec:Lor63})
Next, we study the short-memory approximation proposed in \cite{stinis2004stochastic}. The main idea is to replace 
the integration interval $[0,t]$ in 
\eqref{MemoryPhaseSpace} by a shorter time 
interval $[t-\Delta t,t]$, i.e. 
\begin{align*}
\int_0^t\mathcal{P}e^{s\mathcal{L}}\mathcal{PL}e^{(t-s)\mathcal{QL}}
\mathcal{QL}u_{0}ds\simeq \int_{t-\Delta t}^t\mathcal{P}e^{s\mathcal{L}}\mathcal{PL}
e^{(t-s)\mathcal{QL}}\mathcal{QL}u_{0}ds\qquad \textrm{(short-memory approximation),}
\end{align*} 
where $\Delta t\in[0,t]$ identifies the effective {\em memory length}.
The following result  provides an upper bound to the error associated
with  the short-memory approximation.
\begin{theorem}
\label{short_memory_approximation}
{\bf (Short memory approximation \cite{stinis2004stochastic})}
Let $e^{t\L}$ and $e^{t\mathcal{LQ}}$ be strongly continuous 
semigroups with upper bounds  $\|e^{t\L}\|\leq Me^{t\omega}$ and 
$\|e^{t\L\Q}\|\leq M_{{\Q}}e^{t\omega_{\Q}}$.
Then the following error estimate holds true
\begin{align*}
\bigg\|\int_0^t\mathcal{P}e^{s\mathcal{L}}\mathcal{PL}
e^{(t-s)\mathcal{QL}}\mathcal{QL}u_{0}ds- \int_{t-\Delta t}^t\mathcal{P}e^{s\mathcal{L}}
\mathcal{PL}e^{(t-s)\mathcal{QL}}\mathcal{QL}u_{0}ds\bigg\|\leq M_2(t-\Delta t,t),
\end{align*}
where 
\begin{align*}
M_2(\Delta t,t)=
\begin{dcases}
C_1(t-\Delta t)e^{t\omega_{\Q}} \quad &\omega=\omega_{\Q}\\
C_1e^{\Delta t\omega_{\Q}}\frac{e^{(t-\Delta t)\omega}-e^{(t-\Delta t)\omega_{\Q}}}{\omega-\omega_{\Q}}
 \quad &\omega\neq\omega_{\Q}
\end{dcases}
\end{align*}
and $C_{1} = MM_{\Q}\|\P\|^2\|\L\Q\L u_{0}\|$.
\end{theorem}
\noindent
We omit the proof due to its similarity to that of Theorem \ref{memory_theorem}.  Note that $\displaystyle\lim_{\Delta t\rightarrow t}M_2(\Delta t,t)=0$ for all finite $t>0$. 

\subsection{Hierarchical Memory Approximation}
\label{sec:hierarchical}
An alternative way to approximate the memory 
integral \eqref{MemoryPhaseSpace} was proposed 
by Stinis in \cite{stinis2007higher}. The key idea is to 
repeatedly differentiate \eqref{MemoryPhaseSpace} 
with respect to time, and establish a hierarchy of PDEs which 
can eventually be truncated or approximated at some level to 
provide an approximation of the memory. 
In this section, we derive this hierarchy of memory equations 
and perform a thorough theoretical analysis to establish 
accuracy and convergence of the method. 
To this end, let us first define
\begin{align}
w_0(t)=\int_0^t\mathcal{P}e^{s\mathcal{L}}\mathcal{PL}
e^{(t-s)\mathcal{QL}}\mathcal{QL}u_{0}ds
\label{w0}
\end{align}
to be the memory integral \eqref{MemoryPhaseSpace}. 
By differentiating $w_0(t)$ with respect to time 
we obtain\footnote{Here we are implicitly assuming that 
$w_0(t)$ is differentiable with respect to time. 
For the hierarchical approach to the finite memory 
approximation to be applicable, we must assume 
that $w_0(t)$ is differentiable with respect to time as 
many times as needed.}
\begin{equation*}
\frac{dw_0(t)}{dt}=
\mathcal{P}e^{t\mathcal{L}}\mathcal{PLQL}u_0+w_1(t),
\end{equation*}
where 
\begin{align*}
w_1(t)=\int_0^t\mathcal{P}e^{s\mathcal{L}}
\mathcal{PL}e^{(t-s)\mathcal{QL}}\mathcal{(QL)}^2u_0ds.
\end{align*}
By iterating this procedure $n$ times we obtain 
\begin{align}
\frac{dw_{n-1}(t)}{dt}=
\mathcal{P}e^{t\mathcal{L}}\mathcal{PL}(\mathcal{QL})^{n-1}u_0
+w_n(t),
\label{d1wn}
\end{align}
where
\begin{align}
w_n(t)=\int_0^t\mathcal{P}e^{s\mathcal{L}}
\mathcal{PL}e^{(t-s)\mathcal{QL}}(\mathcal{QL})^{n+1}u_0ds.
\label{wn}
\end{align}
The hierarchy of equations
\eqref{d1wn}-\eqref{wn} is equivalent to the following 
infinite-dimensional system of PDEs 
\begin{equation}
\left\{
\begin{array}{lcl}
\displaystyle \frac{dw_0(t)}{dt} &=&
\mathcal{P}e^{t\mathcal{L}}\mathcal{PLQL}u_0+w_1(t)\\
\displaystyle\frac{dw_1(t)}{dt} &=&
\mathcal{P}e^{t\mathcal{L}}\mathcal{PLQLQL}u_0+w_2(t)\\
&\vdots&\\
\displaystyle\frac{dw_{n-1}(t)}{dt} &=&
\mathcal{P}e^{t\mathcal{L}}\mathcal{PL}(\Q\L)^{n}u_0+w_n(t)\\
&\vdots&
\end{array}\right.
\label{hier_equation}
\end{equation}
evolving from the initial condition $w_i(0)=0$, $i=1,2,\dots$ 
(see equation \eqref{wn}). With such initial condition available, 
we can solve \eqref{hier_equation} with backward substitution, i.e., 
from the last equation to the first one, to 
obtain the following (exact) {\em Dyson series representation} 
of the memory integral \eqref{w0}
\begin{align}
\label{wn(t)}
w_0(t)=&\int_0^t\mathcal{P}
e^{s\mathcal{L}}\mathcal{PLQL}u_0ds
+
\int_0^t\int_0^{\tau_1}\mathcal{P}
e^{s\mathcal{L}}\mathcal{PLQLQL}u_0dsd\tau_1\nonumber\\
&+\dots+
\int_0^t\int_0^{\tau_{n-1}}\dots\int_0^{\tau_1}\mathcal{P}
e^{s\mathcal{L}}\mathcal{PL}(\mathcal{QL})^nu_0ds
d\tau_1\dots d\tau_{n-1}+\dots .
\end{align}
So far no approximation was introduced, i.e., the infinite-dimensional 
system \eqref{hier_equation} and the corresponding formal solution \eqref{wn(t)} are {\em exact}. 
To make progress in developing a computational scheme 
to estimate the memory integral \eqref{w0}, it is 
necessary to introduce approximations. 
The simplest of these rely on truncating the 
hierarchy \eqref{hier_equation} after $n$ equations, 
while simultaneously introducing an approximation 
of the $n$-th order memory integral $w_n(t)$. We denote such an
approximation as $w_n^{e_n}(t)$. The truncated system takes 
the form  
\begin{equation}
\left\{
\begin{array}{lcl}
\displaystyle\frac{dw_0^n(t)}{dt} &=&
\mathcal{P}e^{t\mathcal{L}}\mathcal{PLQL}u_0+w_1^n(t),\vspace{0.2cm}\\
\displaystyle\frac{dw_1^n(t)}{dt} &=&
\mathcal{P}e^{t\mathcal{L}}\mathcal{PLQLQL}u_0+w_2^n(t),\\
&\vdots\\
\displaystyle\frac{dw_{n-1}^n(t)}{dt} &=&
\mathcal{P}e^{t\mathcal{L}}\mathcal{PL}(\Q\L)^{n}u_0+w_n^{e_n}(t).
\end{array}
\right.
\label{hier_equation1}
\end{equation}
The notation $w_j^n(t)$ ($j=0,..,n-1$) emphasizes 
that the solution to \eqref{hier_equation1}
is, in general, different from the solution to 
\eqref{hier_equation}.  
The initial condition of the system can be set as  
$w_i^{n}(0)=0$, for all $i=0,\dots,n-1$. By using 
backward substitution, this yields the following formal solution 
\begin{align}
\label{w0n(t)}
w_0^n(t)=&\int_0^t\mathcal{P}
e^{s\mathcal{L}}\mathcal{PLQL}u_0ds
+
\int_0^t\int_0^{\tau_1}\mathcal{P}
e^{s\mathcal{L}}\mathcal{PLQLQL}u_0dsd\tau_1\nonumber\\
&+\dots+
\int_0^t\int_0^{\tau_{n-1}}\dots\int_0^{\tau_1}\mathcal{P}
e^{s\mathcal{L}}\mathcal{PL}(\mathcal{QL})^nu_0ds
d\tau_1\dots d\tau_{n-1}\nonumber\\
&+\int_0^t\int_0^{\tau_{n-1}}\dots\int_0^{\tau_1}
w_n^{e_n}(s)dsd\tau_1\dots d\tau_{n-1}
\end{align}
representing an approximation of the memory integral \eqref{w0}.
Note that, for a given system, such approximation depends 
only on the number of equations $n$ in \eqref{hier_equation1}, 
and on the choice of approximation $w_n^{e_n}(t)$.
In the present paper, we consider the following 
choices\footnote{The quantities $t_n$ and $\Delta t_n$ appearing in 
\eqref{n-th_order_approxi} and \eqref{type_2} will be defined in 
Theorem \ref{Thm_type_1} and Theorem
\ref{finite_memory_approximation}, respectively.}
\begin{enumerate}
\item Approximation by truncation ($H$-model)
\begin{align}\label{truncation_appro} 
w_n^{e_n}(t)=0.
\end{align}
\item Type-\rom{1} finite memory approximation
\begin{align}\label{n-th_order_approxi}
w_n^{e_n}(t)=\int_{\max(0,t-\Delta t_n)}^t\P e^{s\mathcal{L}}
\mathcal{PL}e^{(t-s)\mathcal{QL}}(\mathcal{QL})^{n+1}u_0ds.
\end{align}
\item Type-\rom{2} finite memory approximation
\begin{align}\label{type_2}
w_n^{e_n}(t)=\int_{\min(t,t_n)}^t\P e^{s\mathcal{L}}
\mathcal{PL}e^{(t-s)\mathcal{QL}}(\mathcal{QL})^{n+1}u_0ds.
\end{align} 
\item $H_t$-model 
\begin{align}\label{Ht_state}
w_n^{e_n}(t)=t\P e^{t\L}\P\L(\Q\L)^{n+1}u_0.
\end{align} 
\end{enumerate}
The first approximation is a truncation of the 
hierarchy obtained by assuming that $w_n(t)=0$. \red{Such approximation 
was originally proposed by Stinis in \cite{stinis2007higher}, and we shall 
call it the $H$-model. The Type-\rom{1} finite memory approximation (FMA) 
is obtained by applying the short memory approximation to the $n$-th order memory integral $w_n(t)$.}
The Type-\rom{2} finite memory approximation (FMA) is a modified 
version of the Type-\rom{1}, with a larger memory band. The $H_t$-
model approximation is based on replacing the $n$-th order memory 
integral $w_n(t)$ with a classical $t$-model.  Note that in this setting 
the classical $t$-model approximation proposed by Chorin and Stinis 
\cite{Chorin1} is equivalent to a zeroth-order $H_t$-model 
approximation.  

Hereafter, we present a thorough mathematical analysis  
that aims at estimating the error $\|w_0(t)-w_0^n(t)\|$, 
where $w_0(t)$ is full memory at time $t$ (see \eqref{w0} or \eqref{wn(t)}), 
while $w_0^{n}(t)$ is the solution of the truncated 
hierarchy \eqref{hier_equation1}, with $w_n^{e_n}(t)$ 
given by \eqref{truncation_appro}, \eqref{n-th_order_approxi}, 
\eqref{type_2} or \eqref{Ht_state}. With such error estimates available, 
we can infer whether the approximation of the full memory $w_0(t)$ with 
$w_0^n(t)$ is accurate and, more importantly, 
if the algorithm to approximate the memory integral 
converges.  To the best of our knowledge, this is the first time a rigorous 
convergence analysis is performed on various approximations of the 
MZ memory integral. It turns out that the distance $\|w_0(t)-w_0^n(t)\|$ can 
be controlled through the construction of the hierarchy under 
some constraint on the initial condition. 

\subsubsection{The \texorpdfstring{$H$}{Lg}-Model}
\label{sec:Hmodel} 
Setting $w_n^{e_n}(t)=0$ in \eqref{hier_equation1} yields an
approximation by truncation, which we will refer to 
as the $H$-model (hierarchical model). Such model was originally 
proposed by Stinis in \cite{stinis2007higher}. Hereafter 
we provide error estimates and convergence results 
for this model. In particular, we derive an upper 
bound for the error $\|w_0(t)-w_0^n(t)\|$, and 
sufficient conditions for convergence of the reduced-order 
dynamical system. Such conditions are problem dependent, i.e., they 
involve the Liouvillian $\L$, the initial condition $u_0$, and the 
projection $\P$.
\begin{theorem}\label{Thm_Decaying_error}
{\bf (Accuracy of the $H$-model)}
Let $e^{t\mathcal{L}}$ and $e^{t\L\Q}$ be 
strongly continuous semigroups with upper bounds 
$\|e^{t\L}\|\leq Me^{t\omega}$ and $\|e^{t\L\Q}\|\leq M_{{\Q}}e^{t
\omega_{\Q}}$, and let $T>0$ be a fixed integration time.  For some fixed $n$, let 
\begin{align}
\alpha_{j} = \frac{\|(\mathcal{LQ})^{j+1}\L u_0\|}{\|
(\mathcal{LQ})^{j}\L u_0\|}, \quad 1\leq j\leq n.
\label{condition}
\end{align}
Then, for any $1\leq p\leq n$ and all $t\in[0,T]$, we have
\begin{align*}
\|w_0(t)-w_0^{p}(t)\|\leq M_3^{p}(t) \leq M_{3}^{p}(T),
\end{align*}
where 
\begin{align*}	
	M_3^{p}(t)=C_1A_1A_2\frac{t^{p+1}}{(p+1)!}\prod_{j=1}^p\alpha_j, \qquad\qquad  C_1=\|\L\Q\L u_0\|MM_{\Q},
\end{align*}
 and
\begin{align}
\label{AoneAtwo}
A_1&=\max_{s\in[0,T]} e^{s(\omega-\omega_{\Q})}=
\begin{dcases}
1\quad &\omega\leq \omega_{\Q} \\
e^{T(\omega-\omega_{\Q})}\quad &\omega\geq \omega_{\Q}
\end{dcases},
& 
A_2&=\max_{s\in[0,T]} e^{s\omega_{\Q}}=
\begin{dcases}
1\quad & \omega_{\Q}\leq 0 \\
e^{T\omega_{\Q}}\quad &\omega_{\Q}\geq0
\end{dcases}.
\end{align}
\end{theorem}

\begin{proof}
We begin with the expression for the difference between the memory 
term $w_{0}$ and its approximation $w_{0}^{p}$ 
\begin{align}
w_{0}(t) - w_{0}^{p}(t) & = \int_{0}^{t}\int_{0}^{\tau_{p}}\cdots \int_{0}^{\tau_{2}}\int_{0}^{\tau_{1}} \P e^{s\L} \P e^{(\tau_{1}-s)\L\Q }(\L\Q)^{n+1}\L u_{0}dsd\tau_{1}\cdots d\tau_{p}.	
\label{w0w1}
\end{align}
Since $e^{t\LV}$ and $e^{t\LV\PrjC}$ are strongly continuous semigroups we have $\|e^{t\LV}\|\leq Me^{\omega t}$ and $\|e^{t\LV\PrjC}\|\leq M_{\PrjC}e^{\omega_{\PrjC}t}$.  By using Cauchy's formula for repeated integration,  we bound the norm of the error \eqref{w0w1} as
\begin{align}\label{esti}
\|w_{0}(t) - w_{0}^{p}(t)\| & \leq \int_{0}^{t}\frac{(t-\sigma)^{p-1}}{(p-1)!}\int_{0}^{\sigma}\|\P e^{s\LV}\Prj e^{(\sigma-s)\LV \PrjC}(\LV\PrjC)^{p+1}\LV u_{0}\|dsd\sigma\nonumber\\
	& \leq \|\P\|^2MM_{\PrjC}\|(\LV\PrjC)^{p+1}\LV u_{0}\|\int_{0}^{t}\frac{(t-\sigma)^{p-1}}{(p-1)!}\int_{0}^{\sigma} e^{s\omega} e^{(\sigma-s)\omega_{\PrjC}}dsd\sigma\nonumber\\
	& \leq C_{1}\left(\prod_{j=1}^p\alpha_j\right)\underbrace{\int_{0}^{t}\frac{(t-\sigma)^{p-1}}{(p-1)!}\int_{0}^{\sigma} e^{s\omega} e^{(\sigma-s)\omega_{\PrjC}}dsd\sigma}_{f_p(t,\omega,\omega_{\Q})}\nonumber	\\
	&=C_1\left(\prod_{j=1}^p\alpha_j\right)f_p(t,\omega,\omega_{\Q}),
\end{align}
where $C_1=\|\P\|^2\|\L\Q\L u_0\|MM_{\Q}$ as before. The function $f_p(t,\omega,\omega_{\Q})$, may be bounded from above as
\begin{align*}
f_p(t,\omega,\omega_{\Q})&\leq A_1A_2\int_{0}^{t}\frac{(t-\sigma)^{p-1}}{(p-1)!}\int_{0}^{\sigma} dsd\sigma	\\
&=A_1A_2\frac{t^{p+1}}{(p+1)!}.
\end{align*}
Hence, we have 
\begin{align*}
\|w_{0}(t) - w_{0}^{p}(t)\| & \leq C_1A_1A_2
	\left(\prod_{j=1}^p\alpha_j\right)
	\frac{t^{p+1}}{(p+1)!}=M_3^{p}(t).
\end{align*}

\end{proof}

Theorem \ref{Thm_Decaying_error} states that
for a given dynamical system (represented by $\L$) and 
quantity of interest (represented by $\P$)  the 
error bound $M_3^{p}(t)$ is strongly related 
to $\{\alpha_j\}$ which is ultimately determined by 
the initial condition $x_0$. It turns out that by 
bounding $\{\alpha_j\}$, we can control $M_3^{p}(t)$, 
and therefore the overall error $\|w_0(t)-w^p_0(t)\|$.
The following corollaries discuss sufficient conditions 
such that the error $\|w_0(T)-w_0^n(T)\|$ decays as we 
increase the differentiation order $n$ for fixed time $T>0$.

\begin{corollary}\label{coro_state_sequence}
{\bf (Uniform convergence of the $H$-model)} 
If $\{\alpha_j\}$ in Theorem\ref{Thm_Decaying_error} satisfy
\begin{align}\label{Thm_Decaying_suff}
\alpha_j<\frac{j+1}{T},\quad 1\leq j\leq n,
\end{align} 
for any fixed time $T > 0$, then there exists a sequence of 
constants $\delta_1>\delta_2>\dots >\delta_n$ such that 
\begin{align*}
\|w_0(T)-w_0^{p}(T )\|\leq \delta_p \qquad 1\leq p\leq n.
\end{align*}
\end{corollary}

\begin{proof}
Evaluating \eqref{esti} at any fixed (finite) time $T>0$ yields 
\begin{align*}
\|w_{0}(T) - w_{0}^{p}(T)\| & \leq
C_2\left(\prod_{j=1}^p\alpha_i\right)f_p(T,\omega,\omega_{\Q})
\leq 
C_2\left(\prod_{j=1}^p\alpha_j\right)\frac{T^{p+1}}{(p+1)!},\\
\|w_{0}(T) - w_{0}^{p+1}(T)\| & \leq
C_2\left(\prod_{j=1}^{p+1}\alpha_j\right)\frac{T^{p+2}}{(p+2)!},
\end{align*}
where $C_2=C_2(T)=C_1A_1A_2$. If there 
exists $\delta_p\geq 0$ such that 
\begin{align*}
\|w_{0}(T) - w_{0}^{p}(T)\| &\leq C_2\left(\prod_{j=1}^p\alpha_j\right)\frac{T^{p+1}}{(p+1)!}\leq   \delta_p,
\end{align*}
then there exist a $\delta_{p+1}$ such that 
\begin{align*}
\|w_{0}(T) - w_{0}^{p+1}(T)\| &\leq
C_2\left(\prod_{j=1}^p\alpha_j\right)\frac{T^{p+1}}{(p+1)!} \frac{\alpha_{p+1} T}{p+2}\leq \delta_{p+1}< \delta_p,
\end{align*}
since $\alpha_{p+1}< (p+2)/T$. Moreover, 
the condition $\alpha_j < (j+1)/T$ holds for all $1\leq j\leq n$. Therefore, we conclude that for any fixed time $T>0$, there exists a sequence of constants $\delta_1>\delta_2>\dots >\delta_n$ such that $\|w_0(T)-w_0^{p}(T )\|\leq \delta_p$, where $1\leq p\leq n$.

\end{proof}

\noindent
Corollary \ref{coro_state_sequence} provides a sufficient condition 
for the error $\|w_0(t)-w_0^p(t)\|$ to decrease monotonically as we increase $p$ in \eqref{hier_equation1}. A stronger condition that yields an asymptotically decaying error bound is given by the following Corollary.

\begin{corollary}\label{cor_decaying_1}
{\bf (Asymptotic convergence of the $H$-model)}
If $\alpha_j$ in Theorem \ref{Thm_Decaying_error} satisfies
\begin{align}\label{cor_conver}
\alpha_j<C,\quad 1\leq j< +\infty
\end{align}
for some positive constant $C$, then for any fixed time $T>0$, 
and arbitrary $\delta>0$, there exists a 
constant $1\leq p<+\infty$ such that for all $n>p$,
\begin{align*}
\|w_0(T)-w_0^{n}(T )\|\leq \delta.
\end{align*}
\end{corollary}
\begin{proof}
By introducing the condition $\alpha_j<C$  in 
the proof of Theorem \ref{Thm_Decaying_error} we obtain
\begin{align*}
\|w_{0}(T) - w_{0}^{p}(T)\| &\leq C_2\left(\prod_{j=1}^p\alpha_j\right)\frac{T^{p+1}}{(p+1)!}\leq C_2T\frac{(CT)^p}{(p+1)!}\qquad  \textrm{for all $1<p<+\infty$}.
\end{align*}
The limit 
\begin{align*}
\lim_{p\rightarrow+\infty}C_2T\frac{(CT)^p}{(p+1)!}=0
\end{align*}
allows us to conclude that there exists a constant $1<p<+\infty$ such that for all $n>p$, $\|w_0(T)-w_0^{n}(T )\|\leq \delta$.
\end{proof}
An interesting consequence of Corollary \ref{cor_decaying_1} is
the existence of a {\em convergence barrier}, i.e., a ``hump'' in the 
error plot $\|w_0(T)-w_0^p(T)\|$ versus $p$ generated by 
the $H$-model. While Corollary \ref{cor_decaying_1} only shows that behavior for an upper bound of the error, not directly the error itself, the feature is often found in the actual errors associated with numerical methods based on these ideas.
The following Corollary shows that the requirements on
$\{\alpha_j\}$ can be dropped (we still need $\alpha_j<+\infty$) 
if we consider relatively short integration times $T$.
\begin{corollary}\label{cor_decaying}
{\bf (Short-time convergence of the $H$-model)} For any 
integer $n$ for which $\alpha_{j}<\infty$ for $1\leq j\leq n$, and any sequence of constants 
$\delta_1>\delta_2>\dots >\delta_n>0$, there exists 
a fixed time $T>0$ such that 
\begin{align*}
\|w_0(T)-w_0^{p}(T)\|\leq
\delta_p
\end{align*}
for $1\leq p\leq n$.
\end{corollary}
\begin{proof}
Since $\alpha_j<+\infty$, we can choose
$\displaystyle C=\max_{1\leq j\leq n}\alpha_j$. 
By following the same steps we used in the proof of Theorem \ref{Thm_Decaying_error}, we conclude that, for
\begin{align*}
	T \leq \frac{1}{C}\min_{1\leq p\leq n}\left[\frac{C(p+1)!}{C_{2}}\delta_{p}\right]^\frac{1}{p+1},
\end{align*}
the errors satisfy
\begin{align*}
\|w_0(T)-w_0^{p}(T )\|&\leq C_2\left(\prod_{j=1}^p\alpha_j\right)\frac{T^{p+1}}{(p+1)!}\leq \frac{C_2}{C}\frac{(CT)^{p+1}}{(p+1)!}\leq \delta_p
\end{align*}
as desired, for all $1\leq p\leq n$.

\end{proof}

Corollary \ref{coro_state_sequence} and Corollary \ref{cor_decaying_1} provide sufficient conditions for the error $\|w_0(T)-w_0^{n}(T)\|$ 
generated by the $H$-model to decay as we increase the truncation 
order $n$. However, we still need to answer the important question of 
whether the $H$-model actually provides accurate results for a given 
nonlinear dynamics ($\L$), quantity of intererest ($\P$) and initial 
state $x_0$. Corollary \ref{cor_decaying} provides a partial answer to 
this question by showing that, at least in the short time period, condition 
\eqref{Thm_Decaying_suff} is always satisfied (assuming that 
$\{\alpha_j\}$ are finite). This guarantees the short-time 
convergence of the $H$-model for any reasonably smooth 
nonlinear dynamical system and almost any observable. However, 
for longer integration times $T$, convergence of the $H$-model 
for arbitrary nonlinear dynamical systems cannot be 
established in general, which means that we need to 
proceed on a case-by-case basis by applying Theorem \ref{Thm_Decaying_error} or by checking whether the 
hypotheses of Corollary \ref{coro_state_sequence} or 
Corollary \ref{cor_decaying_1} are 
satisfied\footnote{The implementation of the $H$-model requires 
computing $(\L\Q)^n\L x_0$ to high-order in $n$. 
This is not straightforward in nonlinear dynamical systems. 
However, such terms can be easily and effectively computed 
for linear dynamical systems. This yields a fast and practical memory approximation scheme for linear systems.}.
On the other hand, convergence of the $H$-model can be 
established for any finite integration time in 
the case of linear dynamical systems, as we 
have recently shown in \cite{YuanFaber2018}. 

\subsubsection{Type-\rom{1} Finite Memory Approximation (FMA)}
The Type-\rom{1} finite memory approximation is obtained by solving 
the system \eqref{hier_equation1} with $w_n^{e_n}(t)$ given by 
\eqref{n-th_order_approxi}. As before, we first derive 
an upper bound for $\|w_0(t)-w_0^n(t)\|$ and then discuss sufficient conditions for convergence. Such conditions basically control the growth of an upper bound on $\|w_0(t)-w_0^n(t)\|$.

\begin{theorem}\label{Thm_type_1}
{\bf(Accuracy of the Type-\rom{1} FMA)}
Let $e^{t\mathcal{L}}$ and $e^{t\L\Q}$ be 
strongly continuous semigroups and let $T>0$ be a fixed integration time. 
If 
\begin{align}
\alpha_{j} = \frac{\|(\mathcal{LQ})^{j+1}\L u_0\|}{\|
(\mathcal{LQ})^{j}\L u_0\|}, \quad 1\leq j\leq n,
\label{condition1}
\end{align}
then for each $1\leq p\leq n$ \red{and for $\Delta t_{p}\leq t\leq T$}
\begin{align*}
\|w_0(t)-w_0^{p}(t )\|\leq M_4^{p}(t),
\end{align*}
where $$\displaystyle M_4^{p}(t)=
C_1A_1A_2\left(\prod_{i=1}^p\alpha_i\right)
\frac{(t-\Delta t_p)^{p+1}}{(p+1)!},$$ and  $C_1,A_1,A_2$ are as in Theorem \ref{Thm_Decaying_error}.
\end{theorem}

\begin{proof}
The error at the $p$-th level is of the form
\begin{align*}
 	w_{p}(t) - w_{p}^{e_{p}}(t) & = \int_{0}^{\max(0,t-\Delta t_{p})}\Prj e^{s\LV}\Prj\LV e^{(t-s)\PrjC\LV}(\PrjC\LV)^{p+1}u_{0}ds
\end{align*}
and the error at the zeroth level is
\begin{subequations}
\begin{align*}
	w_{0}(t) - w_{0}^{p}(t) & = \int_{0}^{t}\int_{0}^{\tau_{p}}\cdots\int_{0}^{\tau_{2}}\left[w_{n}(\tau_{1}) - w_{n}^{e_{p}}(\tau_{1})\right]d\tau_{1}\cdots d\tau_{p}\\
	& = \int_{0}^{t}\int_{0}^{\tau_{p}}\cdots\int_{0}^{\tau_{2}}\int_{0}^{\max(0,\tau_{1}-\Delta t_{p})}\Prj e^{s\LV}\Prj e^{(\tau_{1}-s)\LV\PrjC}(\LV\PrjC)^{p+1}\LV u_{0}dsd\tau_{1}\cdots d\tau_{p}\\
	& = \int_{\Delta t_{p}}^{t}\int_{\Delta t_{p}}^{\tau_{p}}\cdots\int_{\Delta t_{p}}^{\tau_{2}}\int_{0}^{\tau_{1}-\Delta t_{p}}\Prj e^{s\LV}\Prj e^{(\tau_{1}-s)\LV\PrjC}(\LV\PrjC)^{p+1}\LV u_{0}dsd\tau_{1}\cdots d\tau_{p}\\
	& = \int_{\Delta t_{p}}^{t}\int_{\Delta t_{p}}^{\tau_{p}}\cdots\int_{0}^{\tau_{2} - \Delta t_{p}}\int_{0}^{\tilde{\tau}_{1}}\Prj e^{s\LV}\Prj e^{(\tilde{\tau}_{1}+\Delta t_{p}-s)\LV\PrjC}(\LV\PrjC)^{p+1}\LV u_{0}dsd\tilde{\tau}_{1}\cdots d\tau_{p}\\
	& \quad \vdots\notag\\
	& = \int_{0}^{\max(0,t-\Delta t_{p})}\int_{0}^{\tilde{\tau}_{p}}\cdots\int_{0}^{\tilde{\tau}_{2}}\int_{0}^{\tilde{\tau}_{1}}\Prj e^{s\LV}\Prj e^{(\tilde{\tau}_{1}+\Delta t_{p}-s)\LV\PrjC}(\LV\PrjC)^{p+1}\LV u_{0}ds d\tilde{\tau}_{1}\cdots d\tilde{\tau}_{p}.
\end{align*}
\end{subequations}
The norm of this error may be bounded as
\begin{subequations}
\begin{align*}
	\|w_{0}(t) - w_{0}^{p}(t)\| & \leq \int_{0}^{\max(0,t-\Delta t_{p})}\int_{0}^{\tilde{\tau}_{p}}\cdots\int_{0}^{\tilde{\tau}_{2}}\int_{0}^{\tilde{\tau}_{1}}\left\| \P e^{s\LV}\Prj e^{(\tilde{\tau}_{1}+\Delta t_{p}-s)\LV\PrjC}(\LV\PrjC)^{p+1}\LV u_{0}\right\|ds d\tilde{\tau}_{1}\cdots d\tilde{\tau}_{p}\\
	& \leq C_{1}\left(\prod_{j=1}^{p}\alpha_{j}\right)\int_{0}^{\max(0,t-\Delta t_{p})}\int_{0}^{\tilde{\tau}_{p}}\cdots\int_{0}^{\tilde{\tau}_{2}}\int_{0}^{\tilde{\tau}_{1}}e^{s(\omega-\omega_{\PrjC})}e^{(\tilde{\tau}_{1} + \Delta t_{p})\omega_{\PrjC}}dsd\tilde{\tau}_{1}\cdots d\tilde{\tau}_{p}\\
	& \leq C_{1}\left(\prod_{j=1}^{p}\alpha_{j}\right)f_p(t,\Delta t_{p}, \omega, \omega_{\PrjC}),
\end{align*}
\end{subequations}
where
\begin{subequations}
\begin{align*}
f_{p}(t,\Delta t_{p}, \omega, \omega_{\PrjC}) & = \int_{0}^{\max(0,t-\Delta t_{p})}\int_{0}^{\tilde{\tau}_{p}}\cdots\int_{0}^{\tilde{\tau}_{2}}\int_{0}^{\tilde{\tau}_{1}}e^{s(\omega-\omega_{\PrjC})}e^{(\tilde{\tau}_{1} + \Delta t_{p})\omega_{\PrjC}}dsd\tilde{\tau}_{1} \cdots d\tilde{\tau}_{p}.
\end{align*}
\end{subequations}
If we bound $f_{p}$ as
\begin{subequations}
\begin{align*}
f_{p}(t,\Delta t_{p}, \omega, \omega_{\PrjC}) & \leq A_{1}A_{2}\int_{0}^{\max(0,t-\Delta t_{p})}\int_{0}^{\tilde{\tau}_{p}}\cdots\int_{0}^{\tilde{\tau}_{2}}\int_{0}^{\tilde{\tau}_{1}}dsd\tilde{\tau}_{1}\cdots d\tilde{\tau}_{p}\\
& = \begin{cases}0 & 0\leq t \leq \Delta t_{p}\\
\displaystyle
A_{1}A_{2}\frac{(t-\Delta t_{p})^{p+1}}{(p+1)!} & t\geq \Delta t_{p}
\end{cases}
\end{align*}
\end{subequations}
where $A_1,A_2$ are defined  in \eqref{AoneAtwo}, then we have that
\begin{align*}
\|w_{0}(t) - w_{0}^{p}(t)\| & \leq
	C_1A_1A_2\left(\prod_{j=1}^{p}\alpha_{j}\right)\frac{(t-\Delta t_{p})^{p+1}}{(p+1)!}=M_4^{p}(t).
\end{align*}
\end{proof}
We notice that if the effective memory band at each 
level decreases as we increase the differentiation 
order $p$, then we can control the error $\|w_0(t)-w_0^n(t)\|$. 
The following corollary provides a sufficient condition that 
guarantees this sort 
of control of the error.

\begin{corollary}\label{coro_Type1}
{\bf (Uniform convergence of the Type-\rom{1} FMA)}
 If $\alpha_j$ in Therorem \ref{Thm_type_1} satisfy 
\begin{align}\label{convergence_condition1}
\alpha_j<(j+1)\left[\frac{\delta j!}{ C_1A_1A_2\left(\prod_{k=1}^{j-1}\alpha_k\right)}\right]^{-\frac{1}{j}}\qquad 1\leq j\leq n
\end{align} 
then for any $T>0$ and $\delta>0$, 
there exists an ordered sequence $\Delta t_n<\Delta t_{n-1}<\dots <\Delta t_1<T$ such that 
\begin{align*}
\|w_0(T)-w_0^{p}(T)\|\leq \delta, \quad 1\leq p\leq n,
\end{align*}
and which satisfies
\begin{align}
\label{delta_t_n_upper_bound}
\Delta t_p\leq T-\left[\frac{\delta(p+1)!}{ C_1A_1A_2\left(\prod_{j=1}^{p}\alpha_j\right)}\right]^{\frac{1}{p+1}}.
\end{align}
\end{corollary}
\begin{proof}
For $1\leq p\leq n$ we set 
\begin{align*}
\|w_{0}(t) - w_{0}^{p}(t)\| & \leq
	C_1A_1A_2\left(\prod_{j=1}^{p}\alpha_{j}\right)\frac{(t-\Delta t_{p})^{p+1}}{(p+1)!}\leq \delta.
\end{align*}
This yields the following requirement on $\Delta t_p$
\begin{align}\label{delta_t_n_bound}
\Delta t_p\geq T-\left[\frac{\delta(p+1)!}{ C_1A_1A_2\left(\prod_{j=1}^{p}\alpha_j\right)}\right]^{\frac{1}{p+1}}.
\end{align}
Since hypothesis \eqref{convergence_condition1} holds, it is easy to check that the lower bound on each $\Delta t_p$  satisfies
\begin{align*}
 T-\left[\frac{\delta(p+1)!}{ C_1A_1A_2\left(\prod_{j=1}^{p}\alpha_j\right)}\right]^{\frac{1}{p+1}}
<
  T-\left[\frac{\delta p!}{ C_1A_1A_2\left(\prod_{j=1}^{p-1}\alpha_j\right)}\right]^{\frac{1}{p}}\qquad \Delta t_p > \Delta t_{p-1}.
\end{align*}
Therefore, by using the equality in \eqref{delta_t_n_bound} to define a sequence of $\Delta t_{n}$, we find that it is a decreasing time sequence $0<\Delta t_n<\Delta t_{n-1}<\dots <\Delta t_1<T$ such that $\|w_0(T)-w_0^n(T)\|\leq \delta$ holds for all $t\in[0,T]$ and which satisfies \eqref{delta_t_n_upper_bound}.
\end{proof}
\paragraph{Remark} The sufficient condition provided in 
Corollary \ref{coro_Type1} guarantees {\em uniform convergence} of 
the Type-\rom{1} finite memory approximation. 
If we replace condition  \eqref{convergence_condition1} with 
\begin{align*}
\alpha_j<C,\qquad \textrm{for all }\qquad 1\leq j<+\infty,
\end{align*}
where $C$ is a positive constant (independent on $T$), 
then we obtain asymptotic convergence. In other words, 
for each $\delta>0$, there exists an integer $p$ such that for all $n>p$
we have $\left\|w_0(t)-w_0^n(t)\right\|<\delta$.
This result is based on the limit  
\begin{align*}
\lim_{p\rightarrow+\infty}\frac{\delta(p+1)!}{C_1A_1A_2\left(\prod_{j=1}^{p}\alpha_j\right)}
>
\lim_{p\rightarrow+\infty}\frac{\delta(p+1)!}{C_1A_1A_2C^p}=+\infty
\end{align*}
which guarantees the existence of an integer $p$ for which the upper bound on $\Delta t_p$ is smaller or equal to zero. In such case, the Type \rom{1} FMA degenerates to the $H$-model, for which Corollary 
\ref{cor_decaying_1} holds.

\subsubsection{Type-\rom{2} Finite Memory Approximation}
\label{sec:finmemPhase_space}
The Type-\rom{2} finite memory approximation is obtained by solving 
the system \eqref{hier_equation1} with $w_n^{e_n}(t)$ given in 
\eqref{type_2}. We first derive an upper bound 
for $\|w_0(t)-w_0^n(t)\|$ and then discuss sufficient conditions for convergence.
\begin{theorem}\label{finite_memory_approximation}
{\bf (Accuracy of the Type-\rom{2} FMA)}
Let $e^{t\mathcal{L}}$ and $e^{t\L\Q}$ be 
strongly continuous semigroups with upper bounds 
$\|e^{t\L}\|\leq Me^{t\omega}$ and 
$\|e^{t\L\Q}\|\leq M_\Q e^{t\omega_\Q}$.
If  
\begin{align}
\alpha_{j} = \frac{\|(\mathcal{LQ})^{j+1}\L u_0\|}{\|
(\mathcal{LQ})^{j}\L u_0\|}, \quad 1\leq j\leq n,
\label{condition1-2}
\end{align}
then for $1\leq p\leq n$
\begin{align*}
\|w_0(t)-w_0^{p}(t)\|\leq M_5^{p}(t),
\end{align*}
where $$\displaystyle M_5^{p}(t)= C_{1}\left(\prod_{j=1}^{p}\alpha_{j}\right)f_{p}(\omega_{\PrjC},t)h(\omega - \omega_{\PrjC},t_{p}),$$ 
\begin{align*}
f_{p}(\omega_{\PrjC},t) & = \int_{0}^{t}\frac{(t-\sigma)^{p-1}}{(p-1)!}e^{\sigma\omega_{\PrjC}}d\sigma, & 	h(\omega - \omega_{\PrjC},t_{p}) & = \int_{0}^{t_{p}}  e^{s(\omega-\omega_{\PrjC})} ds,
\end{align*}
and $C_1=MM_{\PrjC}\|\P\|^2\|(\LV\PrjC)^{p+1}\LV u_{0}\|$.
\end{theorem}

\begin{proof}
By following the same procedure as in the proof of the Theorem \ref{Thm_Decaying_error} we obtain
\begin{align*}
	w_{0}(t) - w_{0}^{p}(t) & = \int_{0}^{t}\int_{0}^{\tau_{p}}\cdots \int_{0}^{\tau_{2}}\int_{0}^{\min(\tau_{1},t_{p})} \P e^{s\L} \P e^{(\tau_{1}-s)\L\Q }(\L\Q)^{p+1}\L u_{0}dsd\tau_{1} \cdots d\tau_{p}.	
\end{align*}
By applying Cauchy's formula for repeated integration, 
this expression may be simplified to
\begin{align*}
	w_{0}(t) - w_{0}^{p}(t) & = \int_{0}^{t}\frac{(t-\sigma)^{p-1}}{(p-1)!}\int_{0}^{\min(\sigma,t_{p})} \P e^{s\L} \P e^{(\sigma-s)\L\Q }(\L\Q)^{p+1}\L u_{0} dsd\sigma.	
\end{align*}
Thus, 
\begin{align*}
\|w_{0}(t) - w_{0}^{p}(t)\| & \leq \int_{0}^{t}\frac{(t-\sigma)^{p-1}}{(p-1)!}\int_{0}^{\min(\sigma,t_{p})} \left\| \P e^{s\L} \P e^{(\sigma-s)\L\Q }(\L\Q)^{p+1}\L u_{0}\right\|dsd\sigma\\
	& \leq MM_{\PrjC}\|\P\|^2\|(\LV\PrjC)^{p+1}\LV u_{0}\|\int_{0}^{t}\frac{(t-\sigma)^{p-1}}{(p-1)!}\int_{0}^{t_{p}}  e^{s\omega} e^{(\sigma-s)\omega_{\PrjC}}dsd\sigma\\
	& \leq C_{1}\left(\prod_{j=1}^{p}\alpha_{j}\right)\left(\int_{0}^{t}\frac{(t-\sigma)^{p-1}}{(p-1)!}e^{\sigma\omega_{\PrjC}}d\sigma\right)\left(\int_{0}^{t_{p}}  e^{s(\omega-\omega_{\PrjC})} ds\right)\\
	& = C_{1}\left(\prod_{j=1}^{p}\alpha_{j}\right)f_{p}(\omega_{\PrjC},t)h(\omega - \omega_{\PrjC},t_{p})=M_5^{p}(t),
\end{align*}
where $C_1=MM_{\PrjC}\|\P\|^2\|(\LV\PrjC)^{p+1}\LV u_{0}\|$, 
\begin{align}\label{f_p}
	f_{p}(\omega_{\PrjC},t) & = \int_{0}^{t}\frac{(t-\sigma)^{p-1}}{(p-1)!}e^{\sigma\omega_{\PrjC}}d\sigma
	= \begin{cases}
		\displaystyle \frac{t^{p}}{p!} & \omega_{\PrjC} = 0\\
		\displaystyle\frac{1}{\omega_{\PrjC}^{p}}\left[e^{t \omega_{\PrjC}} - \sum_{k=0}^{p-1}\frac{(t \omega_{\PrjC})^{k}}{k!}\right] & \omega_{\PrjC} \neq 0
	\end{cases}
\end{align}
and
\begin{align*}
	h(\omega - \omega_{\PrjC},t_{p}) & : = \int_{0}^{t_{p}}  e^{s(\omega-\omega_{\PrjC})} ds = \begin{cases}
		\displaystyle t_{p} & \omega = \omega_{\PrjC}\\
		\displaystyle\frac{e^{t_{p}(\omega-\omega_{\PrjC})}- 1}{\omega - \omega_{\PrjC}} & \omega\neq \omega_{\PrjC}
	\end{cases}
\end{align*}
are both strictly increasing functions of $t$ and $t_{p}$, respectively. 
\end{proof}

\begin{corollary}\label{coro_type2}
{\bf (Uniform convergence of the Type-\rom{2} FMA})
If $\alpha_j$ in Theorem \ref{finite_memory_approximation} satisfy 
\begin{align}
\label{convergence_condition}
\alpha_j<\frac{j}{T} \quad (\omega_{\Q}=0) \qquad 
\textrm {or} \qquad 
\displaystyle \alpha_j< \omega_{\Q}\frac{e^{T\omega_{\Q}}- \displaystyle\sum_{k=0}^{j-2}\frac{(T\omega_{\Q})^k}{k!}}
{e^{T\omega_{\Q}}- \displaystyle\sum_{k=0}^{j-1}\frac{(T\omega_{\Q})^k}{k!}} \quad (\omega_{\Q}\neq 0)
\end{align}  
for $1\leq j\leq n$, then for any arbitrarily small $\delta>0$, 
there exists an ordered sequence $0<t_0<t_1<\dots <t_n\leq T$ such that
\begin{align*}
\|w_0(T)-w_0^{p}(T)\|\leq \delta,\quad 1\leq p\leq n
\end{align*}
and which satisfies
\begin{align*}
	t_{j} &\geq \begin{dcases}
		\frac{j! \delta}{C_{1}\left(\prod_{i=1}^{j}\alpha_{i}\right)T^{j}} & \omega_{\PrjC} = 0,\\
		\frac{\omega_{\PrjC}^{j}\delta}{ C_{1}\left(\prod_{i=1}^{j}\alpha_{i}\right)\left[e^{T \omega_{\PrjC}} - \sum_{k=0}^{j-1}\frac{(T \omega_{\PrjC})^{k}}{k!}\right]} & \omega_{\PrjC} \neq 0,
	\end{dcases}
\end{align*}
when $\omega = \omega_{\PrjC}$, and 
\begin{align*}
	t_{j} & \geq \begin{dcases}
		\frac{1}{\omega}\ln\left[1+\frac{j!\omega\delta}{ C_{1}\left(\prod_{i=1}^{j}\alpha_{i}\right)T^{j}}\right] & \omega_{\PrjC} = 0,\\
		\frac{1}{\omega-\omega_{\PrjC}}\ln\left[1+\frac{(\omega-\omega_{\PrjC})\omega_{\PrjC}^{j}\delta}{ C_{1} \left(\prod_{i=1}^{j}\alpha_{i}\right)\left[e^{T \omega_{\PrjC}} - \sum_{k=0}^{j-1}\frac{(T \omega_{\PrjC})^{k}}{k!}\right]}\right] & \omega_{\PrjC} \neq 0,
	\end{dcases}
\end{align*}
when $\omega\neq \omega_{\PrjC}$.
\end{corollary}
\begin{proof}
We now consider separately the two cases where $\omega = \omega_{\PrjC}$ and where $\omega\neq\omega_{\PrjC}$. 
If $\omega=\omega_{\PrjC}$, then
\begin{align*}
\|w_{0}(t) - w_{0}^{p}(t)\| & \leq C_{1}\left(\prod_{i=1}^{p}\alpha_{i}\right)\int_{0}^{t}\int_{0}^{\tau_{p}}\cdots \int_{0}^{\tau_{2}}\int_{0}^{t_{p}} e^{\tau_{1}\omega_{\PrjC}} e^{s(\omega - \omega_{\PrjC})}dsd\tau_{1}\cdots d\tau_{p}\\
	& = t_{p}C_{1}\left(\prod_{i=1}^{p}\alpha_{i}\right)\int_{0}^{t}\int_{0}^{\tau_{p}}\cdots \int_{0}^{\tau_{2}} e^{\tau_{1}\omega_{\PrjC}} d\tau_{1}\cdots d\tau_{p}\\
	& = t_{p}C_{1}\left(\prod_{i=1}^{p}\alpha_{i}\right)f_{p}(\omega_{\PrjC},t),
\end{align*}
where $f_{p}(\omega_{\PrjC},t)$ is defined in \eqref{f_p}.
To ensure that $\|w_{0}(t) - w_{0}^{p}(t)\| \leq \delta$ for all $0\leq t\leq T$, we can take
\begin{align*}
	t_{p}C_{1}\left(\prod_{i=1}^{p}\alpha_{i}\right)f_{p}(\omega_{\PrjC},T)=\max_{t\in[0,T]}t_{p}C_{1}\left(\prod_{i=1}^{p}\alpha_{i}\right)f_{p}(\omega_{\PrjC},t) \leq \delta,
\end{align*}
so that
\begin{align*}
	t_{p}  \leq  \frac{\delta}{\displaystyle C_{1}\left(\prod_{i=1}^{p}\alpha_{i}\right)f_{p}(\omega_{\PrjC},T)}
	 = \begin{dcases}
		\frac{p! \delta}{C_{1}\left(\prod_{i=1}^{p}\alpha_{i}\right)T^{p}} & \omega_{\PrjC} = 0,\\
		\frac{\omega_{\PrjC}^{p}\delta}{ C_{1}\left(\prod_{i=1}^{p}\alpha_{i}\right)\left[e^{T \omega_{\PrjC}} - \sum_{k=0}^{p-1}\frac{(T \omega_{\PrjC})^{k}}{k!}\right]} & \omega_{\PrjC} \neq 0.
	\end{dcases}
\end{align*}
On the other hand,  if $\omega \neq \omega_{\PrjC}$ then
\begin{align*}
\|w_{0}(t) - w_{0}^{p}(t)\| & \leq C_{1}\left(\prod_{i=1}^{p}\alpha_{i}\right)\int_{0}^{t}\int_{0}^{\tau_{p}}\cdots \int_{0}^{\tau_{2}}\int_{0}^{t_{p}} e^{\tau_{1}\omega_{\PrjC}} e^{s(\omega - \omega_{\PrjC})}dsd\tau_{1}\cdots d\tau_{p}\\
	& = \frac{e^{t_{p}(\omega-\omega_{\PrjC})}-1}{\omega-\omega_{\PrjC}}C_{1}\left(\prod_{i=1}^{p}\alpha_{i}\right)\int_{0}^{t}\int_{0}^{\tau_{p}}\cdots \int_{0}^{\tau_{2}} e^{\tau_{1}\omega_{\PrjC}} d\tau_{1} \cdots d\tau_{p}\\
	& = \frac{e^{t_{p}(\omega-\omega_{\PrjC})}-1}{\omega-\omega_{\PrjC}}C_{1}\left(\prod_{i=1}^{p}\alpha_{i}\right)f_{p}(\omega_{\PrjC},t).
\end{align*}
To ensure that $\|w_{0}(t) - w_{0}^{p}(t)\| \leq \delta$ for all $0\leq t\leq T$, we can take
\begin{align*}
	\frac{e^{t_{p}(\omega-\omega_{\PrjC})}-1}{\omega-\omega_{\PrjC}}C_{1}\left(\prod_{i=1}^{p}\alpha_{i}\right)f_{p}(\omega_{\PrjC},T) = \max_{t\in[0,T]} \frac{e^{t_{p}(\omega-\omega_{\PrjC})}-1}{\omega-\omega_{\PrjC}}C_{1}\left(\prod_{i=1}^{p}\alpha_{i}\right)f_{p}(\omega_{\PrjC},t) \leq \delta.
\end{align*}
Let us now consider the two cases $\omega > \omega_{\PrjC}$ and $\omega < \omega_{\PrjC}$ separately.
When $\omega > \omega_{\PrjC}$, we have
\begin{align*}
	e^{t_{p}(\omega-\omega_{\PrjC})} & \leq  1+\frac{(\omega-\omega_{\PrjC})\delta}{\displaystyle C_{1}\left(\prod_{i=1}^{p}\alpha_{i}\right)f_{p}(\omega_{\PrjC},T)},
\end{align*}
and
\begin{align*}	
	t_{p} & \leq \begin{dcases}
		\frac{1}{\omega}\ln\left[1+\frac{p!\omega\delta}{ C_{1}\left(\prod_{i=1}^{p}\alpha_{i}\right)T^{p}}\right] & \omega_{\PrjC} = 0,\\
		\frac{1}{\omega-\omega_{\PrjC}}\ln\left[1+\frac{(\omega-\omega_{\PrjC})\omega_{\PrjC}^{p}\delta}{  C_{1} \left(\prod_{i=1}^{p}\alpha_{i}\right)\left[e^{T \omega_{\PrjC}} - \sum_{k=0}^{p-1}\frac{(T \omega_{\PrjC})^{k}}{k!}\right]}\right] & \omega_{\PrjC} \neq 0.
	\end{dcases}
\end{align*}
On the other hand, when $\omega < \omega_{\PrjC}$, we have
\begin{align*}
	\frac{1-e^{-t_{p}(\omega_{\PrjC}-\omega)}}{\omega_{\PrjC}-\omega}C_{1}\left(\prod_{i=1}^{p}\alpha_{i}\right)f_{p}(\omega_{\PrjC},T) = \max_{t\in[0,T]} \frac{1-e^{-t_{p}(\omega_{\PrjC}-\omega)}}{\omega_{\PrjC}-\omega}C_{1}\left(\prod_{i=1}^{p}\alpha_{i}\right)f_{p}(\omega_{\PrjC},t) \leq \delta,
\end{align*}
so that
\begin{align*}
	e^{-t_{p}(\omega_{\PrjC}-\omega)} & \geq  1-\frac{(\omega_{\PrjC}-\omega)\delta}{\displaystyle C_{1}\left(\prod_{i=1}^{p}\alpha_{i}\right)f_{p}(\omega_{\PrjC},T)}
\end{align*}
i.e., 
\begin{align*}
	-t_{p}(\omega_{\PrjC}-\omega) & \geq \ln\left[1-\frac{(\omega_{\PrjC}-\omega)\delta}{ C_{1}\left(\prod_{i=1}^{p}\alpha_{i}\right)f_{p}(\omega_{\PrjC},T)}\right].
\end{align*}
Hence,
\begin{align*}
	t_{p} & \leq \begin{dcases}
		\frac{1}{\omega}\ln\left[1-\frac{p!(-\omega)\delta}{  C_{1}\left(\prod_{i=1}^{p}\alpha_{i}\right)T^{p}}\right] & \omega_{\PrjC} = 0,\\
		-\frac{1}{\omega_{\PrjC}-\omega}\ln\left[1-\frac{(\omega_{\PrjC}-\omega)\omega_{\PrjC}^{p}\delta}{ C_{1}\left(\prod_{i=1}^{p}\alpha_{i}\right)\left[e^{T \omega_{\PrjC}} - \sum_{k=0}^{p-1}\frac{(T \omega_{\PrjC})^{k}}{k!}\right]}\right] & \omega_{\PrjC} \neq 0.
	\end{dcases}
\end{align*}
For all the four cases, if $\omega_{\Q}=0$ then 
we have condition $\alpha_p<p/T$, and the upper bound of the time sequence satisfies:
\begin{align*}
\frac{p!\delta}{  C_1\left(\prod_{i=1}^p\alpha_i\right)T^p}<\frac{(p-1)!\delta}{  C_1\left(\prod_{i=1}^{p-1}\alpha_i\right)T^{p-1}}, \qquad & p\geq 2.
\end{align*} 
If  $\omega_{\Q}\neq 0$ then we have the condition
\begin{align*}
\alpha_p< \omega_{\Q}\frac{\displaystyle  e^{T\omega_{\Q}}-\sum_{k=0}^{p-2}\frac{(T\omega_{\Q})^k}{k!}}
{\displaystyle  e^{T\omega_{\Q}}-\sum_{k=0}^{p-1}\frac{(T\omega_{\Q})^k}{k!}}
\end{align*}
and the upper bound of the time sequence satisfies
\begin{align*}
\frac{\omega_{\Q}^{p-1}}
{\displaystyle \left(e^{T\omega_{\Q}}-\sum_{k=0}^{p-2}\frac{(T\omega_{\Q})^k}{k!}\right)\prod_{i=1}^{p-1}\alpha_i}
<
\frac{\omega_{\Q}^{p}}
{\displaystyle \left(e^{T\omega_{\Q}}-\sum_{k=0}^{p-1}\frac{(T\omega_{\Q})^k}{k!}\right)\prod_{i=1}^{p}\alpha_i}, 
\qquad & p\geq 2.
\end{align*}
Therefore, there always exists a increasing time sequence $0<t_1<\dots <t_n$ such that $\|w_{0}(t) - w_{0}^{p}(t)\| \leq \delta$ for all $0\leq t\leq T$. And since we have proved that this $\delta$-bound on the error holds for all $t_{n}$ upper bounded as in the two cases above, there exists such an increasing time sequence $0<t_1<\dots <t_n$ with $t_{n}$ \emph{lower}-bounded by the same quantities. Indeed, because of the coarseness of the approximations applied in the proof, there may exist such a time sequence with significantly larger $t_{i}$.
\end{proof}

\paragraph{Remark}
If we replace (\ref{convergence_condition}) with the stronger 
condition    
\begin{align}\label{stronger_convergence_condition}
&\begin{dcases}
\alpha_j<\frac{j}{\epsilon T}, \quad &\omega_{\Q}=0 \\
\alpha_j<\frac{1}{\epsilon T}, \quad &\omega_{\Q}\neq 0 
\end{dcases}& & 1\leq j< \infty,
\end{align} 
where $\epsilon$ is some arbitrary constant satisfying $\epsilon >1$, then we 
have
\begin{align*}
	\lim_{j\rightarrow+\infty}t_{j} &\geq \lim_{j\rightarrow+\infty}\begin{dcases}
		\frac{j! \delta}{ C_{1}\left(\prod_{i=1}^{j}\alpha_{i}\right)T^{j}}\geq \frac{\delta}{C_1}\epsilon^j=+\infty & \omega_{\PrjC} = 0,\\
		\frac{\omega_{\PrjC}^{j}\delta}{ C_{1}\left(\prod_{i=1}^{j}\alpha_{i}\right)\left[e^{T \omega_{\PrjC}} - \sum_{k=0}^{j-1}\frac{(T \omega_{\PrjC})^{k}}{k!}\right]}
= 	
\frac{\omega_{\PrjC}^{j}\delta}{  C_{1}\left(\prod_{i=1}^{j}\alpha_{i}\right)o(T^j\omega_{\Q}^j)}= +\infty
		 & \omega_{\PrjC} \neq 0.
	\end{dcases}
\end{align*}
for $\omega=\omega_{\Q}$ and 
\begin{align*}
	\lim_{j\rightarrow+\infty}t_{j} & \geq \lim_{j\rightarrow+\infty}\begin{dcases}
	\frac{j}{\omega}\ln\left[\frac{\delta}{C_1}\omega\epsilon\right]	=+\infty
		 & \omega_{\PrjC} = 0,\\
\frac{j}{\omega-\omega_{\Q}}\ln\left[\frac{(\omega-\omega_{\Q})T^j}{C_1o(T^j)}\right]=+\infty				
		& \omega_{\PrjC} \neq 0.
	\end{dcases}
\end{align*}
Hence, there exists a $j$ such that the upper bound for $t_j$ is greater than or equal 
to $T$. For such case, the Type \rom{2} FMA degenerates to the truncation 
approximation ($H$-model), for which Corollary \ref{cor_decaying_1} grants 
us asymptotic convergence.

\subsubsection{\texorpdfstring{$H_t$}{Ht}-model}
The $H_t$-model is obtained by solving 
the system \eqref{hier_equation1} with $w_n^{e_n}(t)$ 
approximated using Chorin's $t$-model \cite{Chorin1} 
(see equation \eqref{Ht_state}). Convergence analysis can be 
performed by using the mathematical methods we employed 
for the proofs of the $H$-model.  Note that the classical $t$-model is equivalent to a zeroth-order $H_t$-model.
\begin{theorem}\label{Thm_Decaying_error_Ht}
{\bf (Accuracy of the $H_t$-model)}
Let $e^{t\mathcal{L}}$ and $e^{t\L\Q}$ be 
strongly continuous semigroups with upper bounds 
$\|e^{t\L}\|\leq Me^{t\omega}$ and $\|e^{t\L\Q}\|\leq M_{{\Q}}e^{t
\omega_{\Q}}$, and let $T>0$ be a fixed integration time.  For some fixed $n$, let 
\begin{align}
\alpha_{j} = \frac{\|(\mathcal{LQ})^{j+1}\L u_0\|}{\|
(\mathcal{LQ})^{j}\L u_0\|}, \quad 1\leq j\leq n.
\label{conditionHt}
\end{align}
Then, for any $1\leq p\leq n$ and all $t\in[0,T]$, we have
\begin{align*}
\|w_0(t)-w_0^{p}(t)\|\leq M_6^{p}(t) \leq M_{6}^{p}(T),
\end{align*}
where 
\begin{align*}	
M_6^{p}(t)=C_4\left(\prod_j^p\alpha_j\right)\frac{t^{p+1}}{(p+1)!},
\qquad C_4=\left[C_1A_1A_2+\frac{C_1}{M_{\Q}A_3}\right],\qquad 
A_3:=\max_{s\in[0,T]}se^{s\omega}=\begin{cases}
1\quad &\omega\leq 0,\\
e^{T\omega}\quad &\omega>0
\end{cases},
\end{align*}
and $C_1$, $A_1$, $A_2$ are the same as before.
\end{theorem}
\begin{proof}
For $p$-th order $H_t$-model, the difference between the memory 
term $w_{0}$ and its approximation $w_{0}^{p}$ is  
\begin{align}\label{H_t_formula}
w_{0}(t) - w_{0}^{p}(t) & = \int_{0}^{t}\int_{0}^{\tau_{p}}\cdots 
\int_{0}^{\tau_{2}}\left[\int_{0}^{\tau_{1}} \P e^{s\L} \P 
e^{(\tau_{1}-s)\L\Q }(\L\Q)^{p+1}\L u_{0}ds-\tau_1\P e^{\tau_1\L}
\P(\L\Q)^{p+1}\L u_0\right]d\tau_{1}\cdots d\tau_{p}.	
\end{align}
Using Cauchy's formula for repeated integration, 
we can bound the norm of the second term in \eqref{H_t_formula} as
\begin{align}\label{esti1}
\left\|\int_{0}^{t}\frac{(t-\sigma)^{p-1}}{(p-1)!}\sigma\P e^{\sigma\L}
\P(\L\Q)^{p+1}\L x_0 d\sigma \right\|&\leq 
\int_{0}^{t}\frac{(t-\sigma)^{p-1}}{(p-1)!}\|\sigma\P e^{\sigma\L}\P(\L
\Q)^{p+1}\L x_0\| d\sigma		
\nonumber\\
 & \leq \|\Prj\|^{2}M\|(\LV\PrjC)^{p+1}\LV u_{0}\|
 \underbrace{\int_{0}^{t}\frac{(t-\sigma)^{p-1}}{(p-1)!}\sigma 
 e^{\sigma\omega} d\sigma}_{g_p(t,\omega)}\nonumber\\
 &=\frac{C_1}{M_{\Q}}\left(\prod_{j=1}^p\alpha_j\right)g_p(t,\omega),
\end{align}
where $C_1=\|\P\|^2\|\L\Q\L u_0\|MM_{\Q}$ as before. The function $g_p(t,\omega)$, may be bounded from above as
\begin{align*}
g_p(t,\omega)&\leq A_3\int_{0}^{t}\frac{(t-\sigma)^{p-1}}{(p-1)!}\sigma  d\sigma=A_3\frac{t^{p+1}}{(p+1)!},\quad A_3:=\max_{s\in[0,T]}e^{s\omega}=
\begin{cases}
1\quad &\omega\leq 0\\
e^{T\omega}\quad &\omega>0
\end{cases}.
\end{align*}
By applying the triangle inequality to \eqref{H_t_formula}, and taking  \eqref{esti1} into account, we obtain  
\begin{align*}
\|w_{0}(t) - w_{0}^{p}(t)\| & \leq C_1A_1A_2
	\left(\prod_{j=1}^p\alpha_j\right)
	\frac{t^{p+1}}{(p+1)!}+\frac{C_1}{M_{\Q}}A_3\left(\prod_{j=1}^p\alpha_j\right)\frac{t^{p+1}}{(p+1)!}=M_6^p(t).
\end{align*}
\end{proof}
One can see that the upper bounds $M_6^p(t)$ and $M_3^{p}(t)$ (see 
Theorem \ref{Thm_Decaying_error}) share the same structure, the only 
difference being the constant out front.  Hence by changing $C_2$ to 
$C_4$, we can prove of a series of corollaries similar to 
\ref{coro_state_sequence}, \ref{cor_decaying_1}, and 
\ref{cor_decaying}. In summary, what holds for the $H$-model also holds 
for the $H_t$-model. For the sake of brevity, we omit the statement and 
proofs of those corollaries.

\subsection{Linear Dynamical Systems}
\label{sec:linearDyn}
\red{
The upper bounds we obtained above are not easily computable for general nonlinear systems and infinite-rank projections, e.g., Chorin's projection \eqref{Chorin_projection}. However, if the dynamical system is linear, then such upper bounds are explicitly computable and convergence of the $H$-model can be established for linear phase space functions in any finite integration time $T$.} To this end, consider the linear system $\dot x=A x$ with random initial condition $x(0)$ sampled from the joint probability density function
\begin{equation}
\rho_0(x_0) = \delta (x_{01}- x_1(0)) 
\prod_{j=2}^N {\rho}_{0j}(x_{0j}).
\end{equation}  
In other words, the initial condition for the quantity of 
interest $u(x)=x_1(t)$ is 
set to be deterministic, while all other variables $x_2,\dots ,x_N$ are 
zero-mean and statistically independent at $t=0$\footnote{These choices for $\rho_{0}$ are merely for convenience in demonstrating important features.  With a more general choice of $\rho_{0}$, it is convenient to represent $\L$, $\P$, and $\Q$ in terms of an orthonormal basis for $V$ with respect to the $\rho_{0}$ inner product.  Then, e.g., operator norms within the invariant subspace reduce to matrix norms of the associated matrix.}. Here we assume for simplicity that 
$\rho_{0j}$ ($j=2,..,N$) are i.i.d. standard normal distributions. Observe that the Liouville operator associated with the linear system $\dot x=A x$ is 
\begin{align}
\label{linear_liouvillian}
\L=\sum_{i=1}^{N}\sum_{j=1}^N A_{ij}x_j\frac{\partial}{\partial x_i},
\end{align}
where $A_{ij}$ are the entries of the matrix $A$. If we choose observable $u=x_1(t)$, then Chorin's projection operator \eqref{9} yields the evolution equation for the conditional expectation $\mathbb{E}[x_1(t)|x_1(0)]$ , i.e., the {\em conditional mean path} \eqref{conditonal mean path}, which can be explicitly written as
\begin{equation}
\frac{d}{dt}\mathbb{E}[x_1|x_1(0)]=A_{11}\mathbb{E}[x_1|x_1(0)]+w_0(t),
\label{MZlinear}
\end{equation}
where $A_{11}=\P\L x_1(0)$ is the first entry of the matrix $A$, $w_0$ represents the memory integral \eqref{w0}. 
\red{Next, we explicitly compute the upper bounds for the memory growth and the error in the $H$-model for this system. To this end,  we first  notice that the domain of the Liouville operator can be  restricted to the linear space 
\begin{equation}
V=\textrm{span}\{x_1,\dots ,x_N\}.
\label{spaceV}
\end{equation}
 In fact, $V$ is invariant under $\L$, $\P$ and $\Q$, i.e., $\L V \subseteq V$, $\P V\subseteq V$ and $\Q V\subseteq V$.  These operators have the following matrix representations
\begin{align*}
	\L & \simeq A^{T} \simeq \left[\begin{array}{c c} a_{11} & {b}^{T}\\{a} & {M}_{11}^{T}	\end{array}\right],& 
	\P &\simeq \begin{bmatrix}1 & 0 & \cdots & 0\\0 & 0 & \cdots & 0\\\vdots & \vdots & \ddots & \vdots \\ 0 & 0 &\cdots & 0\end{bmatrix}, &
	 \Q &\simeq \begin{bmatrix} 0 & 0 & \cdots & 0\\0 & 1 & \cdots & 0\\\vdots & \vdots & \ddots & \vdots \\ 0 & 0 &\cdots & 1\end{bmatrix},
\end{align*}
where $M_{11}$ is the minor of the matrix of $A$ 
obtained by removing the first column and the first row, 
while  
\begin{align}
a =[A_{12}\cdots A_{1N}]^T, \qquad 
b^T =[A_{21}\cdots A_{N1}].
\label{ab}
\end{align}
Therefore,
\begin{align}
	\L\Q & \simeq 
	\left[\begin{array}{c c}
	0 & b^T \\
	\vspace{2pt}	{0} & {M}_{11}^{T}
\end{array}\right],\qquad
\L(\Q\L)^{n}x_{1}(0) \simeq 
				\left[\begin{array}{c}
				{b}^{T}\left({M}_{11}^{T}\right)^{n-1}{a} \\
				\left({M}_{11}^{T}\right)^{n}{a} 
				\end{array}\right].
\label{matrix_formula}
\end{align}
At this point, we set $x_{01}=x_1(0)$ and 
\begin{align*}
q(t,x_{01},\tilde x_0)=\int_0^t  e^{s\mathcal{L}}\mathcal{PL}
e^{(t-s)\mathcal{QL}}\mathcal{QL}x_{01}ds.
\end{align*} 
Since $\tilde x_0=(x_2(0),...,x_N(0))$ is random, $q(t,x_{01},\tilde x_0)$ is a random variable.  By using Jensen's inequality $[\mathbb{E}(X)]^2\leq \mathbb{E}[X^2]$, we have the following $L^{\infty}$ estimate
\begin{align}
\|(\P q)(t,x_{01})\|_{L^{\infty}} \leq \| q(t,x_{01},\cdot)\|_{L^2_{\rho_0}}.
\label{ineq}
\end{align}
On the other hand, we have 
\begin{align}\label{linear_etl}
\|e^{t\L}\|_{L^2_{\rho_0}(V)}\leq \|e^{t\L}\|_{L^2_{\rho_0}}\leq e^{t\omega},\quad \omega=- \frac{1}{2}\inf\Div_{\rho_0}{F}.
\end{align}
For linear dynamical systems, both $\|\cdot\|_{L^2_{\rho_0}(V)}$ and $\|\cdot\|_{L^2_{\rho_0}}$ upper bounds can be used to estimate the norm of the semigroup $e^{t\L}$. However, for the semigroup $e^{t\L\Q}$, we can only obtain the explicit form of the $\|\cdot\|_{L^2_{\rho_0}(V)}$ bound, which is given by the following perturbation theorem \cite{engel1999one} (see also Appendix \ref{sec:semigroupBoundsDecomposition}):
\begin{align}\label{linear_etql}
\| e^{t\L\Q}\|_{L^2_{\rho_0}(V)}\leq e^{t\omega_{\Q}},\quad \textrm{where}\quad \omega_{\Q}=\omega
+\sqrt{A_{11}^2+\sum_{i=2}^N A_{1i}^2\frac{\langle x_i^2(0)\rangle_{\rho_0}}{x_1^2(0)}}\geq  \omega+\|\L\P\|_{L^2_{\rho_0}(V)}.
\end{align}
}

\red{
\paragraph{Memory growth} It is straightforward at this point to 
compute the upper bound of the memory growth we obtained in Theorem\ref{memory_theorem}.  Since $\|\P\|_{L^2_{\rho_0}}= \|\Q\|_{L^2_{\rho_0}} = 1$ ($\P$ and $\Q$ are orthogonal projections relative to $\rho_0$), we have the following result
\begin{align}\label{prior_estimation_wot}
|w_0(t)|\leq\|\L\Q\L x_1(0)\|\frac{e^{t\omega}-e^{t\omega_{\Q}}}{\omega-\omega_{\Q}}
=\sqrt{ (b^Ta)^2 x_1^2(0)+\left\| {\Lambda}_{x_{i+1}(0)}
M_{11}^T a \right\|^2_{2}}
\frac{e^{t\omega}-e^{t\omega_{\Q}}}{\omega-\omega_{\Q}},
\end{align} 
where ${\Lambda}_{x_{i+1}(0)}$ is a $N-1\times N-1$ diagonal matrix with $\Lambda_{ii}=\langle x_{{i+1}}(0)\rangle_{\rho_0}$,  and $\|\cdot\|_{2}$ is the vector $2$-norm. 
}

\red{
\paragraph{Accuracy of the $H$-model} We are interested in 
computing the upper bound of the approximation error generated by 
the $H$-model (see section \ref{sec:Hmodel} - Theorem \ref{Thm_Decaying_error}).
By using the matrix representation of $\L$, $\P$ and $\Q$, the $n$-th order $H$-model MZ equation \eqref{MZlinear} for linear system can be explicitly written as 
\begin{equation}
\begin{cases}
\displaystyle\frac{d}{dt}\mathbb{E}[x_1|x_1(0)]=A_{11}\mathbb{E}[x_1|x_1(0)]+w_0^n(t)\qquad \textrm{(MZ equation),}\vspace{0.2cm}\\
\displaystyle\frac{dw_j^n(t)}{dt} =b^T(M_{11}^T)^j a^T \mathbb{E}[x_1|x_1(0)]+w_{j+1}^n(t),\qquad j=0,1,\dots ,n-1,\vspace{0.2cm}\\
\displaystyle\frac{d w_{n}^n(t)}{dt} =b^T(M_{11}^T)^n a^T\mathbb{E}[x_1|x_1(0)],
\end{cases}
\label{hier_equation1lor}
\end{equation}
where $M_{11}$, $a$ and $b$ are defined as 
before (see equation \eqref{ab}). 
The upper bound for the memory term approximation error is 
explicitly obtained as\footnote{\red{The error bound for 
$|w_0(t)-w_0^n(t)|$ used here is slightly different from the one we obtained 
in Theorem \ref{Thm_Decaying_error}. 
Instead of bounding the quotient 
$\alpha_n=|\L(\Q\L)^{n+1} u_0\|/\|\L(\Q\L)^{n}u_0\|$, here 
we choose to bound $\|\L(\Q\L)^nu_0\|$ directly, which yields 
the estimate \eqref{prior_estimation_Hmodel}.}} 
\begin{align}
|w_0(t)-w_0^n(t)|&\leq
A_1A_2\|\L(\Q\L)^nx_1(0)\|\frac{t^{n+1}}{(n+1)!}\nonumber\\
&=
A_1A_2\sqrt{\left[b^T\left(M_{11}^T\right)^na\right]^2x_1^2(0)+\left\|{\Lambda}_{x_{i+1}(0)}\left(M_{11}^T\right)^{n+1}a\right\|_{2}^2}
\frac{t^{n+1}}{(n+1)!}
\label{prior_estimation_Hmodel}
\end{align}
where $A_1$, $A_2$ are defined in \eqref{AoneAtwo}, while $\omega$ and $
\omega_{\Q}$  are given in \eqref{linear_etl} and \eqref{linear_etql}, 
respectively.  Note that for each fixed integration time $T$, the upper bound
\eqref{prior_estimation_Hmodel} goes to zero as we send $n$ to infinity, i.e., 
\begin{align*}
\lim_{n\rightarrow+\infty}|w_0(T)-w_0^n(T)|=0.
\end{align*}
This means that the $H$-model converges for all linear dynamical systems 
with observables in the linear space \eqref{spaceV}. 
}

\subsection{Memory Estimates for Finite-Rank Projections and Hamiltonian Systems}\label{sec:mori_GLE}
\red{
The semigroup estimates we obtained in section \ref{sec:semigroup_estimation} allow us to compute explicitly 
an {\em a priori} estimate of the memory kernel in the Mori-Zwanzig 
equation if we employ {\em finite-rank} projections. In this section we 
outline the procedure to obtain such estimate for Hamiltonian dynamical systems. We begin by recalling that such systems are 
divergence-free, i.e., 
\begin{equation}
\Div_{\rho_{eq}}( F) =0.
\label{divFree}
\end{equation}
Here, $F(x)$ is the velocity field in \eqref{eqn:nonautonODE}, 
while $\rho_{eq}=e^{-\beta\H}/Z$ is the canonical Gibbs 
distribution\footnote{Equation \eqref{divFree} is
obtained by noticing that  
\begin{align*}
\Div_{\rho_{eq}}(F) = & e^{\beta\H}\nabla\cdot\left(e^{-\beta \H} F\right)\\
= & e^{\beta\H}\sum_{i=1}^N \left(
\frac{\partial }{\partial q_i}
\left[e^{-\beta \H} \frac{\partial \H}{\partial p_i}\right]- 
\frac{\partial }{\partial p_i}
\left[e^{-\beta \H} \frac{\partial \H}{\partial q_i}\right]
\right)\\
=& 0.
\end{align*}}. Equation \eqref{divFree}, together with equation\eqref{koopmanEstimate} imply that the Koopman 
semigroup of a Hamiltonian dynamical system is a {\em contraction} 
in the $L_{eq}^2$ norm, i.e., 
\begin{equation}
\left\|e^{t\L}\right\|_{L_{eq}^2}\leq 1.
\end{equation}
Moreover, the MZ equation \eqref{reduced order equation} with 
finite-rank projection $\P$ of the form \eqref{MoriProjection} 
(with $\sigma=\rho_{eq}$) can be reduced to the following 
Volterra integro-differential equation 
	\begin{align}\label{gle}
		\frac{d}{dt}\Prj u_i(t) = \sum_{j=1}^M \Omega_{ij}\Prj {u}_j(t) - \sum_{j=1}^M\int_{0}^{t}K_{ij}(t-s)\Prj{u}_j(s)ds, \qquad i=1,...,M
	\end{align}
where 
\begin{subequations}
\begin{align}
		G_{ij} & = \langle u_{i}, u_{j}\rangle_{eq}\label{yo1}\\
		\Omega_{ij} &= \sum_{k=1}^M ({G}^{-1})_{jk}\langle u_{k}, \LV u_{i}\rangle_{{eq}}\\ 
		K_{ij}(t-s)  &=-\sum_{k=1}^M({G}^{-1})_{jk}\langle \PrjC\LV u_{k}, e^{(t-s)\PrjC\LV}\PrjC\LV u_{i}\rangle_{eq}.
	 \label{SFD}
	\end{align}
\end{subequations}
Equation \eqref{gle} is often referred to as the generalized Langevin equation (GLE) \cite{Darve,Snook}. 
To derive \eqref{yo1}-\eqref{SFD}, we used that fact that $\L$ is skew-adjoint 
and $\Q$ is self-adjoint with respect to the $L^2_{\rho_{eq}}$ inner product, 
and that $\Q^2=\Q$. Equation \eqref{SFD} is known in statistical physics as the {\em second
fluctuation dissipation theorem}. 
Next, define the time-correlation matrix
\begin{align}
		C_{ij}(t) = \langle u_{j}(0),u_{i}(t)\rangle_{eq} = \langle u_{j}(0),\Prj u_{i}(t)\rangle_{eq}.
	\end{align}
By applying $\displaystyle \langle u_j,(\cdot)\rangle_{eq}$ to both sides of equation \eqref{gle}, we obtain the following exact evolution equation
\begin{align}\label{gle_C}
\frac{d C_{ij}}{dt} = \sum_{k=1}^M \Omega_{ik} C_{kj} - \sum_{k=1}^M \int_{0}^{t}K_{ik}(t-s) C_{kj}(s)ds.
\end{align}
Moreover, if we employ a one-dimensional Mori's basis, i.e., $M=1$, then we obtain the simplified  equation
 \begin{align}\label{gle_C1D}
\frac{d C(t)}{dt} = \Omega_{1} C(t) - \int_{0}^{t}K(t-s) C(s)ds.
\end{align}
where $C(t)=\langle u_{1}(0), u_{1}(t)\rangle_{eq}$.
The main difficulty in solving the GLE \eqref{gle_C} (or \eqref{gle_C1D})
lies in computing the memory kernel $K_{ij}(t)$. 
Hereafter we show that such memory kernel can be uniformly bounded by 
a computable quantity that depends only on the initial condition. 
For the sake of simplicity, we shall focus on the one-dimensional GLE \eqref{gle_C1D}, where $u_1$ is the quantity of interest. 
\begin{theorem}\label{corollary_bounded}
For a one-dimensional GLE of the form \eqref{gle_C}, the memory 
kernel $K(t)$ is uniformly bounded by $\|\dot{u}_1(0)\|^2_{\rho_{eq}}/\| u_1(0)\|^2_{\rho_{eq}}$, i.e.
\begin{align}\label{upper_bound_memory}
|K(t)| & \leq \frac{\|\dot{u}_1(0)\|^2_{L^2_{\rho_{eq}}}}{\| u_1(0)\|^2_{L^2_{eq}}} \qquad \forall t\geq 0.
\end{align}
\end{theorem}
\begin{proof}
From the second-fluctuation dissipation theorem \eqref{SFD}, the memory kernel $K(t)$ satisfies
\begin{align*}
|K(t)|=\left|\frac{ \langle e^{t\Q\L}\Q\L u_1(0), \Q\L u_1(0)\rangle_{eq}}{\langle u_1(0),u_1(0)\rangle_{eq}}\right|
\leq\|e^{t\Q\L}\Q\|_{L^2_{\rho_{eq}}}
\frac{\|\L u_1(0)\|^2_{L^2_{\rho_{eq}}}}{\|u_1(0)\|^2_{L^2_{eq}}}
=\|e^{t\Q\L}\Q\|_{L^2_{\rho_{eq}}}
\frac{\|\dot u_1(0)\|^2_{L^2_{\rho_{eq}}}}{\|u_1(0)\|^2_{L^2_{eq}}}
\end{align*}
On the other hand, by using the numerical abscissa \eqref{numerical_abscissa} and formula \eqref{orth_semigroup_bound}, we see that the 
semigroup $e^{t\Q\L\Q}$ is contractive, i.e. $\|e^{t\Q\L\Q}\|_{L^2_{\rho_{eq}}}\leq 1$.  Since $\Q$ is an orthogonal projection with respect to $\rho_{eq}$, we have $\|e^{t\Q\L}\Q\|_{L^2_{\rho_{eq}}}=\|\Q  e^{t\Q\L\Q}\|_{L^2_{\rho_{eq}}}\leq \|\Q\|_{L^2_{\rho_{eq}}}\|e^{t\Q\L\Q}\|_{L^2_{\rho_{eq}}}\leq 1$. This yields
\begin{align*}
|K(t)|\leq
\|e^{t\Q\L}\Q\|_{L^2_{\rho_{eq}}}
\frac{\|\dot u_1(0)\|^2_{L^2_{\rho_{eq}}}}{\|u_1(0)\|^2_{L^2_{eq}}}\leq 
 \frac{\|\dot{u_1}(0)\|^2_{L^2_{\rho_{eq}}}}{\| u_1(0)\|^2_{L^2_{eq}}}.
\end{align*}
\end{proof}
Theorem \ref{corollary_bounded} provides an {\em a priori} 
(easily computable) upper bound for the memory kernel defining the dynamics of any quantity of interest $u_1$ that is initially in the Gibbs ensemble 
$\rho_{eq}=e^{-\beta\H}/Z$. In section \ref{sec:application}, we will 
calculate the upper bound \eqref{upper_bound_memory} 
analytically and compare it with the exact memory kernel we obtain in prototype linear and nonlinear Hamiltonian systems. 
\paragraph{Remark.}
We emphasized in section \ref{sec:semigroup_estimation} 
that the semigroup estimate for $e^{t\Q\L\Q}$ is not 
necessarily tight. In the context of high-dimensional Hamiltonian systems 
(e.g., molecular dynamics) it is often empirically assumed that 
the semigroup $e^{t\Q\L}\Q$ is dissipative, i.e. $\| e^{t\Q\L\Q}\| \leq e^{t\omega_0}$, where $\omega_0<0$. In this case, the memory 
kernel turns out to be uniformly bounded by an exponentially 
decaying function since
\begin{align*}
|K(t)|\leq \|e^{t\Q\L}\Q\|_{L^2_{\rho_{eq}}}\frac{\|\L u_1\|^2_{L^2_{\rho_{eq}}}}{\|u_1\|^2_{L^2_{eq}}}
\leq e^{t\omega_{0}} 
\frac{\|\dot u_1\|^2_{L^2_{\rho_{eq}}}}{\| u_1\|^2_{L^2_{eq}}}.
\end{align*} }

\section{Numerical Examples}
\label{sec:application}
\red{
In this section, we provide simple numerical examples 
of the MZ memory approximation methods 
we discussed throughout the paper. Specifically, we study 
Hamiltonian systems (linear and nonlinear) with 
finite-rank projections (Mori's projection), and non-Hamiltonian 
systems with infinite-rank projections (Chorin's projection). 
In both cases we demonstrate the accuracy of the a priori
memory estimation method we developed 
in \S \ref{sec:mori_GLE} and \S \ref{sec:linearDyn}.  We also 
compute the solution to the MZ equation for non-Hamiltonian systems 
with the $t$-model, the $H$-model and the $H_t$-model. 
}

\subsection{Hamiltonian Dynamical Systems with Finite-Rank Projections} 
In this section we consider dimension reduction in linear and nonlinear Hamiltonian dynamical systems with finite-rank projection.  
In particular, we consider the Mori projection and study the 
MZ equation for the temporal auto-correlation function of a 
scalar quantity of interest. 

\subsubsection{Harmonic Chains of Oscillators}
\red{
Consider a one-dimensional chain of harmonic oscillators. 
This is a simple but illustrative example of a linear Hamiltonian 
dynamical system which has been widely studied in statistical 
mechanics, mostly in relation with the microscopic theory of 
Brownian motion 
\cite{baxter2016exactly,florencio1985exact,espanol1996dissipative}. 
The Hamiltonian of the system can be written as   
\begin{align}
\label{bethe_hamiltonian}
\H(p,q)=\frac{1}{2m}\sum_{i=1}^N p_i^2+
\frac{k}{2}\sum_{\substack{i,j=0\\i<j}}^{N+1} (q_i-q_j)^2,
\end{align}
where $q_i$ and $p_i$ are, respectively, the displacement 
and momentum of the $i$-th particle, $m$ is the 
mass of the particles (assumed constant throughout the network), 
and $k$ is the elasticity constant that modulates the intensity of the 
quadratic interactions. We set fixed boundary conditions at 
the endpoints of the chain, 
i.e., $q_0(t)=q_{N+1}(t)=0$ and $p_0(t)=p_{N+1}(t)=0$
(particles are numbered from left to right) and $m=k=1$. 
The Hamilton's equations are
\begin{equation}
\frac{dq_i}{dt} = \frac{\partial \H}{\partial p_i},\qquad
\frac{dp_i}{dt} = -\frac{\partial \H}{\partial q_i},
\end{equation}
which can be written in a matrix-vector form as
\begin{align}
\label{equ:bethe_l}
\left[
\begin{matrix}
\dot{ p}\\
\dot{q}
\end{matrix}
\right]
=
\left[
\begin{matrix}
0&k B-k D\\
I/m& 0
\end{matrix}
\right]
\left[
\begin{matrix}
p\\
q
\end{matrix}
\right]
\end{align}
where $B$ is the adjacency matrix of the chain and $D$ is the degree matrix (see \cite{biggs1993algebraic}). Note that \eqref{equ:bethe_l} is a linear dynamical system. 
We are interested in the velocity auto-correlation function of a tagged 
oscillator, say the one at location $j=1$. Such auto-correlation 
function is defined as 
\begin{equation}
C_{p_1}(t)=\frac{\langle p_1(0)p_1(t)\rangle_{eq}}
{\langle p_1(0)p_1(0)\rangle_{eq}},
\label{autocorr}
\end{equation}
where the average is with respect to the 
Gibbs canonical distribution $\rho_{eq}=e^{-\beta\H}/Z$.  
It was shown in \cite{florencio1985exact} 
that $C_{p_1}(t)$ can be obtained analytically 
by employing Lee's continued fraction method . 
The result is the well-known $J_0-J_{4}$ solution
\begin{align}
C_{p_1}(t)=J_0(2t)-J_{4}(2t),
\label{VCF_analytic}
\end{align}
where $J_{i}(t)$ is the $i$-th Bessel function of the 
first kind. On the other hand, the Mori-Zwanzig equation 
derived by the following Mori's projection 
\begin{align}\label{mori_projection}
\P(\cdot)=\frac{\langle(\cdot),p_1(0)\rangle_{eq}}{\langle p_1(0),p_1(0)\rangle_{eq}}p_1(0)
\end{align}
yields the following GLE for $C_{p_1}(t)$
\begin{align}\label{GLE_linear_0}
\frac{dC_{p_1}(t)}{dt}=\Omega_{p_1}C_{p_1}(t)-\int_0^t K(s)C_{p_1}(t-s)ds.
\end{align}
Here, $$\displaystyle \Omega_{p_1}=\frac{\langle \L p_1(0),p_1(0)\rangle_{eq}}{\langle p_1(0),p_1(0)\rangle_{eq}}=0$$ 
since $\langle p_i(0),q_j(0)\rangle_{eq}=0$, while $K(t)$ is the memory kernel. For the $J_0-J_4$ solution, it is possible to derive the memory kernel $K(t)$ analytically. To this end, we simply insert \eqref{VCF_analytic} into \eqref{GLE_linear_0} and apply the Laplace transform   
\begin{align*}
\mathscr{L}[\cdot](s)=\int_0^\infty  (\cdot)  e^{-st}dt 
\end{align*}
to obtain 
\begin{align}
\hat K(s) = -s +\frac{1}{\hat{C}(s)},
\label{KLAP}
\end{align}
where $\hat C(s)=\mathscr{L}[C_{p_1}(t)]$ and $\hat{K}(s)=\mathscr{L}[K(t)]$. The inverse Laplace transform of \eqref{KLAP} can be computed 
analytically as 
\begin{equation}
K(t) = \frac{J_1(2t)}{t}+1.
\end{equation}
With $K(t)$ available, we can verify the memory estimated we derived in Theorem \ref{corollary_bounded}. To this end,    
\begin{align}\label{upper_bound_chain}
|K(t)|\leq \frac{\|\dot p_1(0)\|^2_{L^2_{\rho_{eq}}}}{\| p_1(0)\|^2_{L^2_{\rho_{eq}}}}
=\frac{\|q_2(0)-2q_1(0)\|^2_{L^2_{\rho_{eq}}}}{\| p_1(0)\|^2_{L^2_{\rho_{eq}}}}
=2.
\end{align} 
Here we used the exact solution of the velocity auto-correlation function and displacement auto-correlation function  of the fixed end harmonic chain given by (see \cite{florencio1985exact}) 
\begin{align*}
\langle v_i(0),v_j(0)\rangle_{eq}=\frac{k_BT}{\pi}\int_0^{\pi}\sin(ix)\sin(jx)dx,\qquad 
\langle q_i(0),q_j(0)\rangle_{eq}=\frac{k_BT}{\pi}\int_0^{\pi}\frac{\sin(ix)\sin(jx)}{4\sin^2(x/2)}dx.
\end{align*}
In Figure \ref{fig:VCF} we plot the absolute value of the memory kernel 
$K(t)$ together with the theoretical bound \eqref{upper_bound_chain}. It is seen that the bound we obtain in this case is of the same order of magnitude as the memory kernel. }
\begin{figure}[t]
\centerline{\hspace{0.5cm}(a)\hspace{7.1cm} (b)}
\centerline{
\includegraphics[height=6cm]{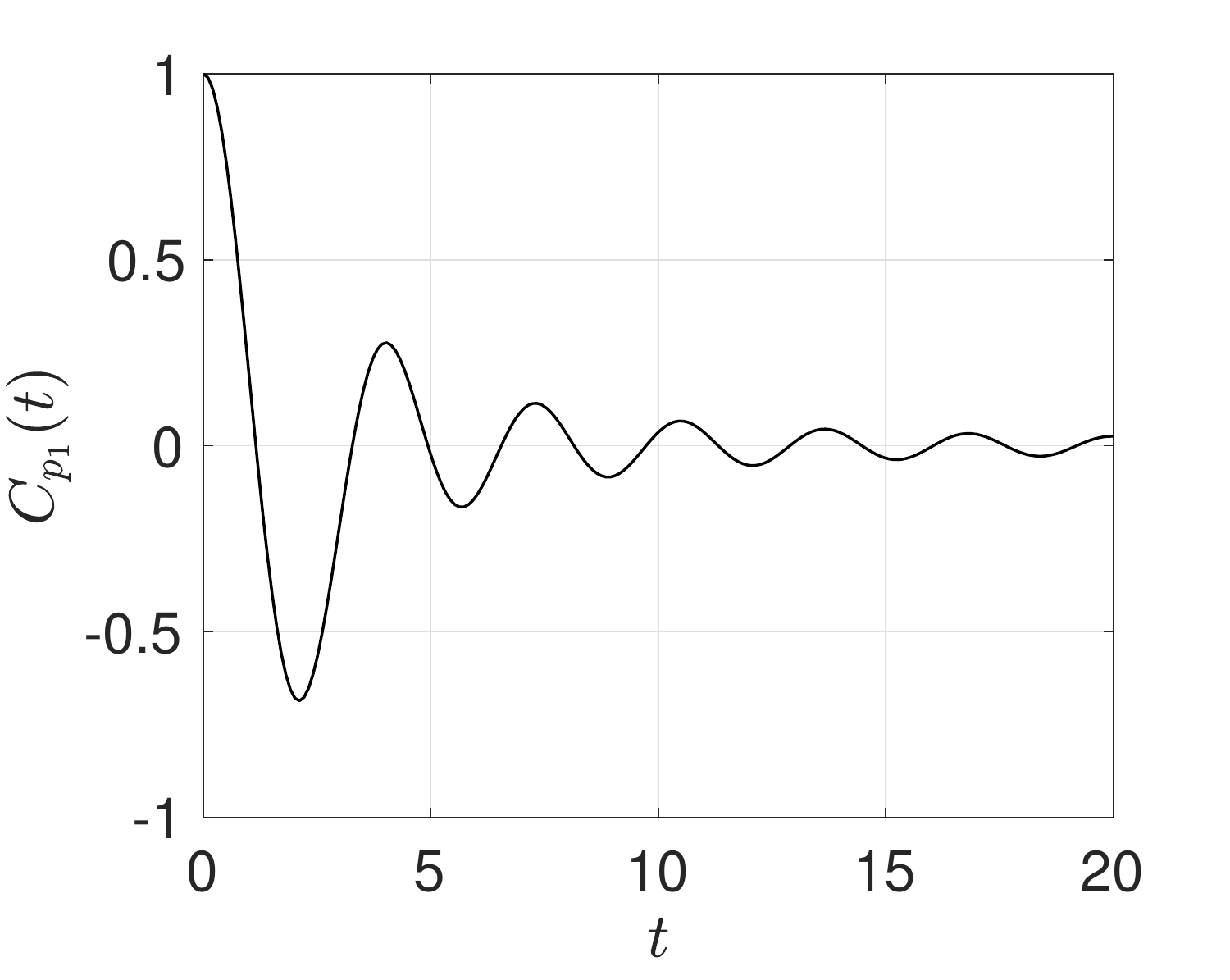} \hspace{0.0cm}
\includegraphics[height=6cm]{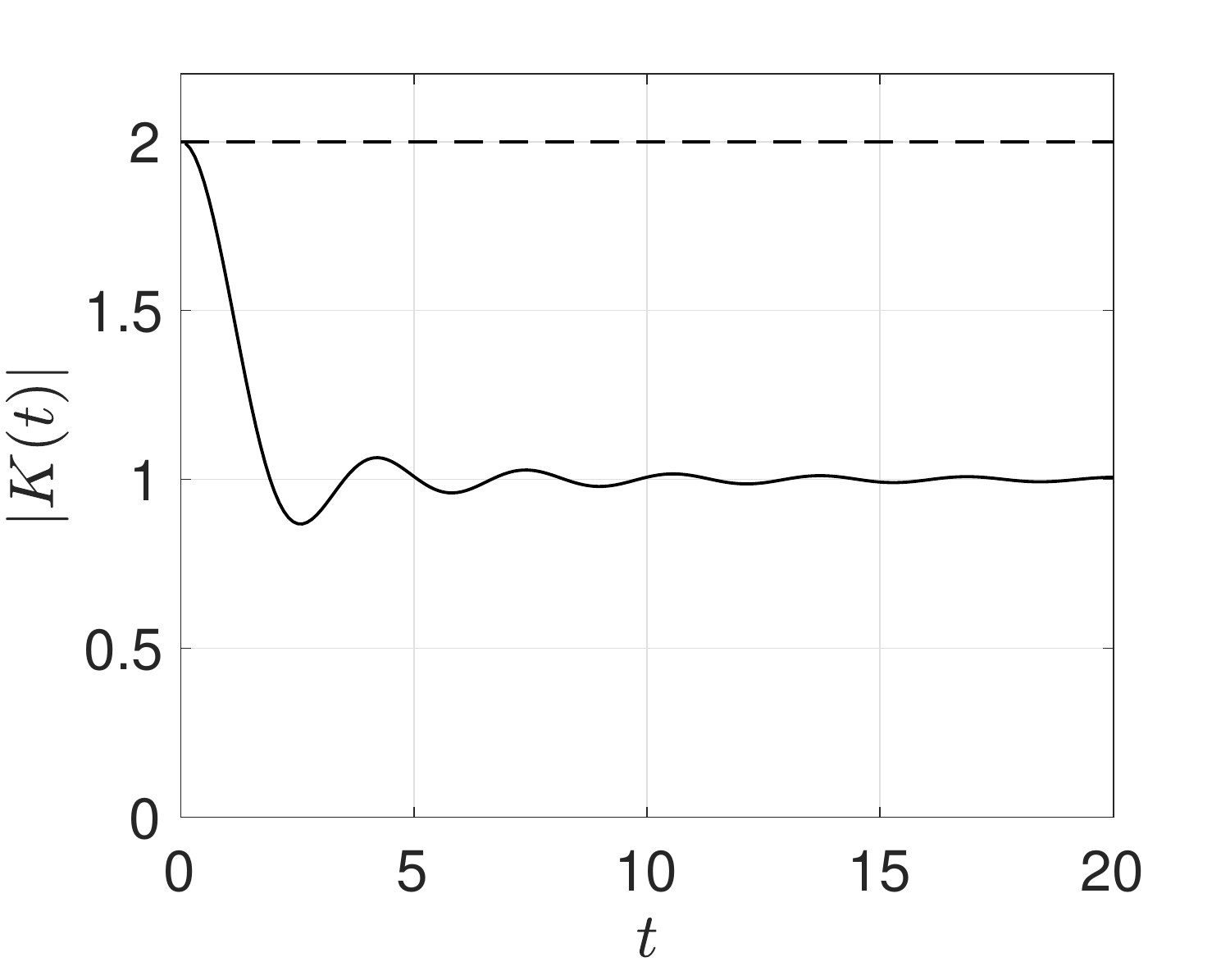} 
}
\caption{\red{Harmonic chain of oscillators. (a) Velocity auto-correlation function $C_{p_1}(t)$ and (b) memory kernel $K(t)$ of the corresponding 
MZ equation. It is seen that our theoretical estimate
\eqref{upper_bound_chain} (dashed line) correctly bounds 
the MZ memory kernel. The upper bound we obtain is of the same order 
of magnitude as the memory kernel.}}
\label{fig:VCF} 
\end{figure}

\subsubsection{Hald System}
\red{
In this section, we study the Hald Hamiltonian system 
studied by Chorin et. al. in \cite{chorin2002optimal,Chorin1}. 
The Hamiltonian is given by 
\begin{align}
\H:=\frac{1}{2}(q_1^2+p_1^2+q_2^2+p_2^2+q_1^2q_2^2),
\end{align}
while the corresponding Hamilton's equations of motion are
\begin{align}
\begin{cases}
\dot q_1&=p_1\\
\dot p_1&=-q_1(1+q_2^2)\\
\dot q_2&=p_2\\
\dot p_2&=-q_2(1+q_1^2)
\end{cases}
\label{Haldsystem}
\end{align}
We assume that the initial state is distributed according to canonical Gibbs distribution $\rho_{eq}=e^{-\beta\H}/Z$, where we set $\beta=1$ for simplicity. The partition function $Z$ is given by 
\begin{equation}
Z=e^{1/4}(2\pi)^{3/2}K_0\left(\frac{1}{4}\right), 
\end{equation}
where $K_0(t)$ is the modified Bessel function of the second kind.  We aim to study the properties of the autocorrelation function of the first component $q_1$, which is defined as 
\begin{align*}
C_{q_1}(t)=\frac{\langle q_1(0),q_1(t)\rangle_{eq}}{\langle q_1(0),q_1(0)\rangle_{eq}}
\end{align*}
Obviously, $C_{q_1}(0)=1$. The evolution equation for $C_{q_1}(t)$ 
is obtained by using the MZ formulation with the Mori's projection 
\begin{align}\label{mori_projection1}
\P(\cdot)=\frac{\langle(\cdot),q_1(0)\rangle_{eq}}{\langle q_1(0),q_1(0)\rangle_{eq}}q_1(0)
\end{align} 
This yields the GLE
\begin{align}\label{GLE_linear}
\frac{dC_{q_1}(t)}{dt}=\Omega_{q_1}C_{q_1}(t)-\int_0^tK(s)C_{q_1}(t-s) ds.
\end{align}
The streaming  term $\Omega_{q_1}C_{q_1}(t)$ is again identically zero, since
$$\displaystyle \Omega_{q_1}=\frac{\langle \L q_1(0),q_1(0)\rangle_{eq}}{\langle q_1(0),q_1(0)\rangle_{eq}}=0.$$ 
Theorem \ref{corollary_bounded} provides the following computable upper bound for the modulus of $K(t)$
\begin{align}\label{upper_bound}
|K(t)|\leq \frac{\|\dot q_1(0)\|^2_{L^2_{\rho_{eq}}}}{\| q_1(0)\|^2_{L^2_{\rho_{eq}}}}
=\frac{\|p_1(0)\|^2_{L^2_{\rho_{eq}}}}{\| q_1(0)\|^2_{L^2_{\rho_{eq}}}}
=\frac{\displaystyle e^{1/4}K_0\left(1/4\right) }{\displaystyle\sqrt{\pi} U\left(1/2,0,1/2\right)  }\approx 1.39786
\end{align} 
where $U(a,b,y)$ is the confluent hypergeometric function of the second kind. 
In Figure \ref{fig:correlation_hald}, we plot the correlation function 
$C_{q_1}(t)$ that we obtain numerically with Markov Chain 
Monte Carlo, and the memory kernel 
$K(t)$\footnote{\red{The memory kernel $K(t)$ 
here is computed by inverting numerically the Laplace 
transform of \eqref{GLE_linear}, i.e., 
\begin{equation}
K(t) = \mathscr{L}^{-1}\left[-s + \frac{1}{\hat C(s)}\right],
\label{IK}
\end{equation}
where $\hat C(s)= \mathscr{L} [C_{q_1}(t)]$. 
In practice, we replaced the numerical solution $C_{q_1}(t)$ within 
the time interval $[0,20]$ with with a high-order interpolating 
polynomial at Gauss-Chebyshev-Lobatto nodes (in $[0,20]$), 
computed $\hat C(s)$ analytically (Laplace transform of a polynomial), 
and then computed the inverse Laplace transform \eqref{IK} 
numerically with the Talbot algorithm \cite{Abate2006}.} and the upper bound \eqref{upper_bound}.}}

\begin{figure}[t]
\centerline{\hspace{0.5cm}(a)\hspace{7.1cm} (b)}
\centerline{
\includegraphics[height=6cm]{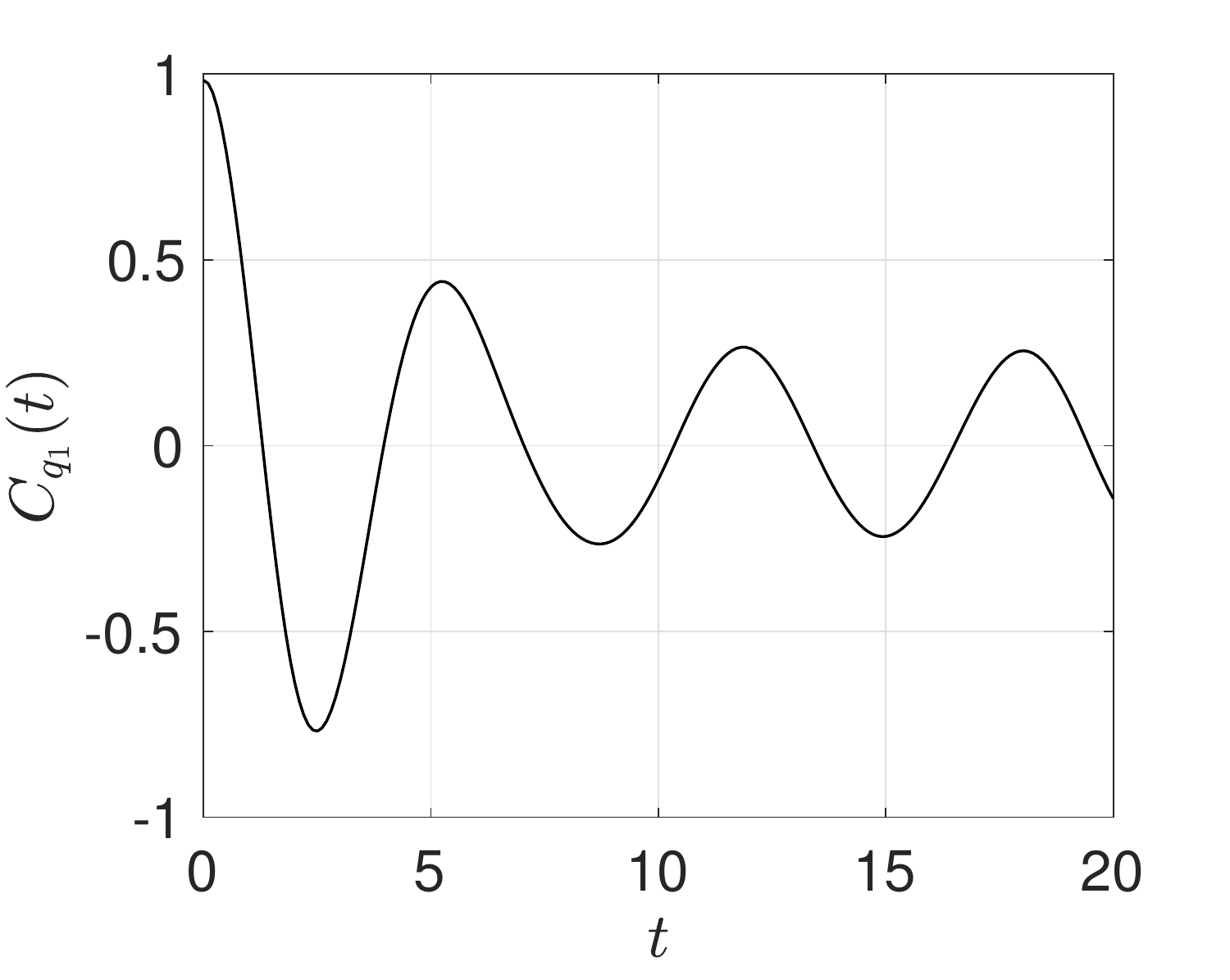}\hspace{0.0cm}
\includegraphics[height=6cm]{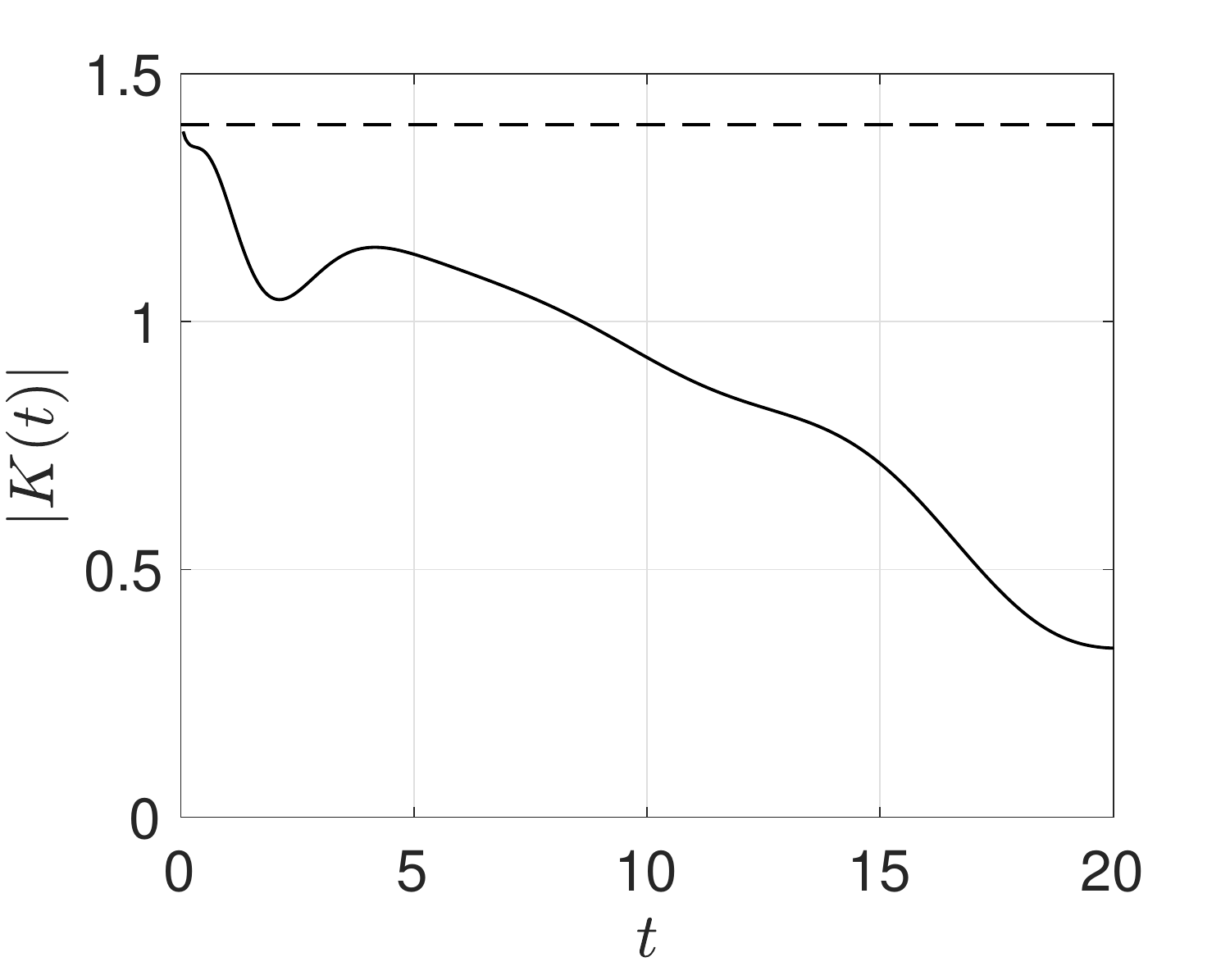}}
\caption{Hald Hamiltonian system \eqref{Haldsystem}.   (a) 
Autocorrelation function of the displacement $q_1(t)$ 
and (b) memory kernel of the governing MZ equation. Here $C_{q_1}(t)$ is computed by Markov chain Monte-Carlo (MCMC) while $K(t)$ is determined by inverting numerically the Laplace transform in \eqref{IK} with the Talbot algorithm. It is seen that the theoretical upper bound 
\eqref{upper_bound} (dashed line) is of the same order of magnitude as 
the memory kernel.}
\label{fig:correlation_hald}
\end{figure}

\subsection{Non-Hamiltonian Systems with Infinite-Rank Projections}
In this section we study the accuracy of the $t$-model, the $H$-model 
and the $H_t$ model in predicting scalar quantities of interest 
in non-Hamiltonian systems. In particular, we consider the MZ formulation 
with Chorin's projection operator. For the particular case of linear 
dynamical systems we also compute the theoretical upper bounds 
we obtained in \S \ref{sec:linearDyn} for the memory growth and 
the error in the $H$-model, and compare such bounds with 
exact results. 

\subsubsection{Linear Dynamical Systems}

\red{
We begin by considering a low-dimensional linear dynamical system 
$\dot x=A x$ evolving from a random initial state with density 
$\rho_0(x)$ to verify the MZ memory estimates we obtained in \S 
\ref{sec:linearDyn}. For  simplicity, we choose $A$ to be negative definite
\begin{align}
\label{matrix_A}
A= e^C B e^{-C},\qquad 
{B}=
\begin{bmatrix}
    -\frac{1}{8} & 0 & 0  \\
    0 & -\frac{2}{3} &  0 \\
 0& 0& -\frac{1}{2}
\end{bmatrix},
\qquad
{C}=
\begin{bmatrix}
    0 & 1 &   0 \\
   -1& 0 &1 \\
    0 & -1 &0
\end{bmatrix}.
\end{align}
In this case, the origin of the phase space is a stable node and 
it is easy to estimate 
$\left\|e^{t\L}\right\|_{\rho_0}$\footnote{\red{
For general matrices $A$, it is more difficult to estimate 
$\| e^{t\L}\|_{L^2_{\rho_0}}$. However, since $\L$ is a bounded linear 
operator in the subspace $V$ where the quantity of interest lives, we can use 
the  norm $\|e^{t\L} \|_{L^2_{\rho_0}(V)}$, which is explicitly computable.}}
We set $x_1(0)=1$ and ${x_2(0), x_3(0)}$  independent standard normal 
random variables. In this setting, the semigroup estimates \eqref{linear_etl} 
and \eqref{linear_etql} are explicit 
\begin{align*}
\| e^{t\L}\|&\leq e^{t\omega},\quad \omega=-\frac{1}{2}\Tr(A)=0.6458,\\
\|e^{t\L\Q}\|&\leq e^{t\omega_{\Q}},\quad \omega_{\Q}=\omega
+\sqrt{A_{11}^2+\sum_{i=2}^N\frac{A_{1i}^2}{x_1^2(0)}}=1.1621.
\end{align*}
Therefore, we obtain the following explicit upper bounds for the 
memory integral and the error of the $H$-model (see equations  
\eqref{prior_estimation_wot} and \eqref{prior_estimation_Hmodel}) 
\begin{align}\label{prior_1}
|w_0(t)|&\leq 0.1964\left(e^{1.1621t}-e^{0.6458t}\right),\\
|w_0(t)-w_0^n(t)| &\leq e^{1.1624t}
\sqrt{\left( b^T\left(M_{11}^T\right)^na_1\right)^2 x_1^2(0)+
\left\|\left(M_{11}^T\right)^{n+1}a\right\|_{2}^2}
\frac{t^{n+1}}{(n+1)!}.
\label{prior_2}
\end{align}
Next, we compare these error bounds with numerical results obtained by solving numerically the $H$-model \eqref{hier_equation1lor}. For example, the  second-order $H$-model reads  
\begin{align}
\begin{dcases}
\frac{d}{dt}\mathbb{E}[x_1(t)|x_{1}(0)]=-0.4560\mathbb{E}[x_1(t)|x_{1}(0)]+w_0^2(t),\\
\frac{dw_0^2(t)}{dt}=0.0586\mathbb{E}[x_1(t)|x_{1}(0)]+w_1^2(t),\\
\frac{dw_1^2(t)}{dt}=-0.0192\mathbb{E}[x_1(t)|x_{1}(0)].
\end{dcases}
\label{Hm2}
\end{align}
In Figure \ref{fig:linear_3d} we demonstrate convergence 
of the $H$-model to the benchmark solution computed by 
Monte-Carlo simulation as we increase the $H$-model 
differentiation order. In Figure \ref{fig:Linearbounds} we 
plot the bound on the memory growth (equation \eqref{prior_1}) 
and the bound in the memory error (equation \eqref{prior_2}) together 
with exact results.}

\begin{figure}[t]
\centerline{
\includegraphics[height=6cm]{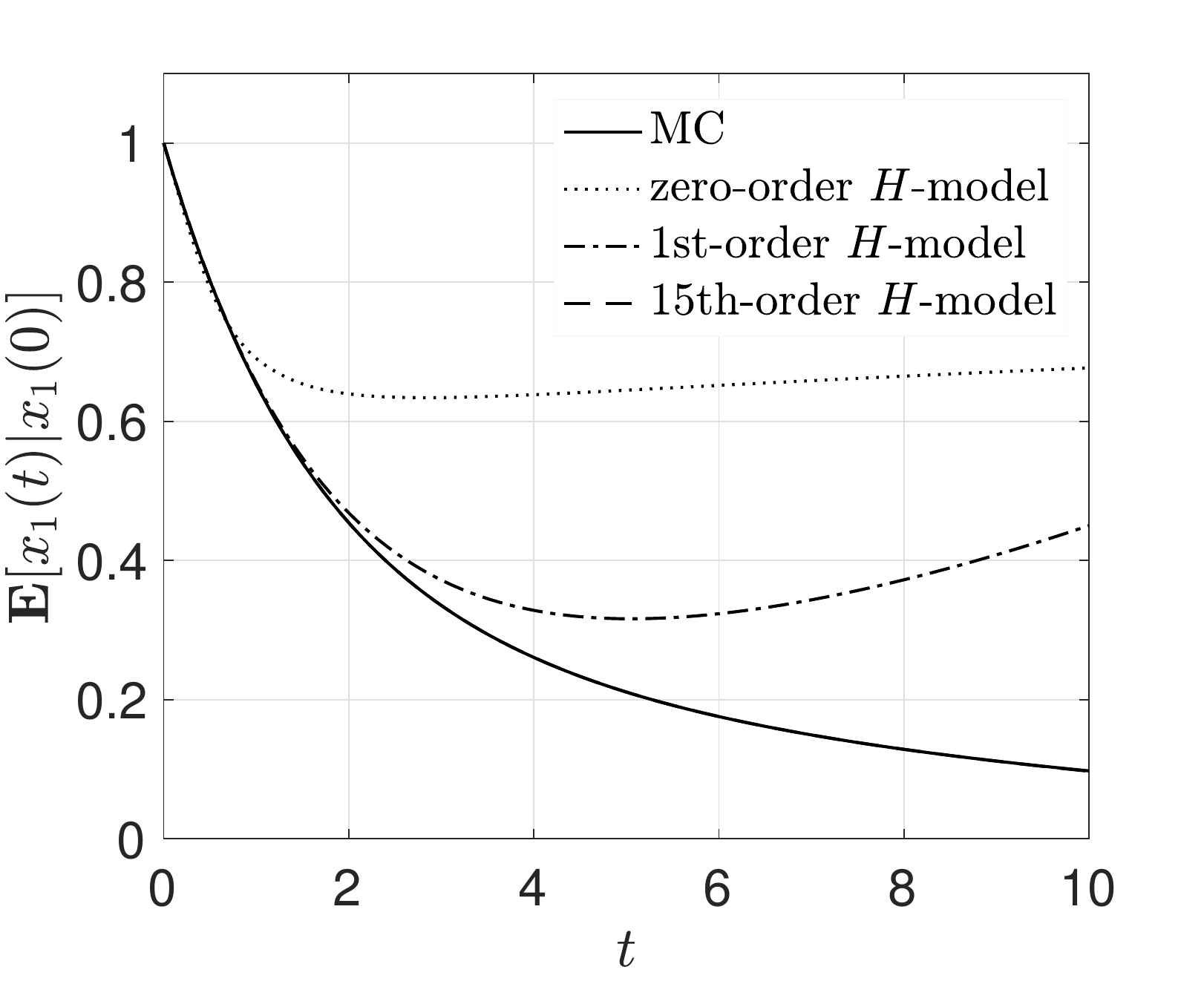} 
\includegraphics[height=6cm]{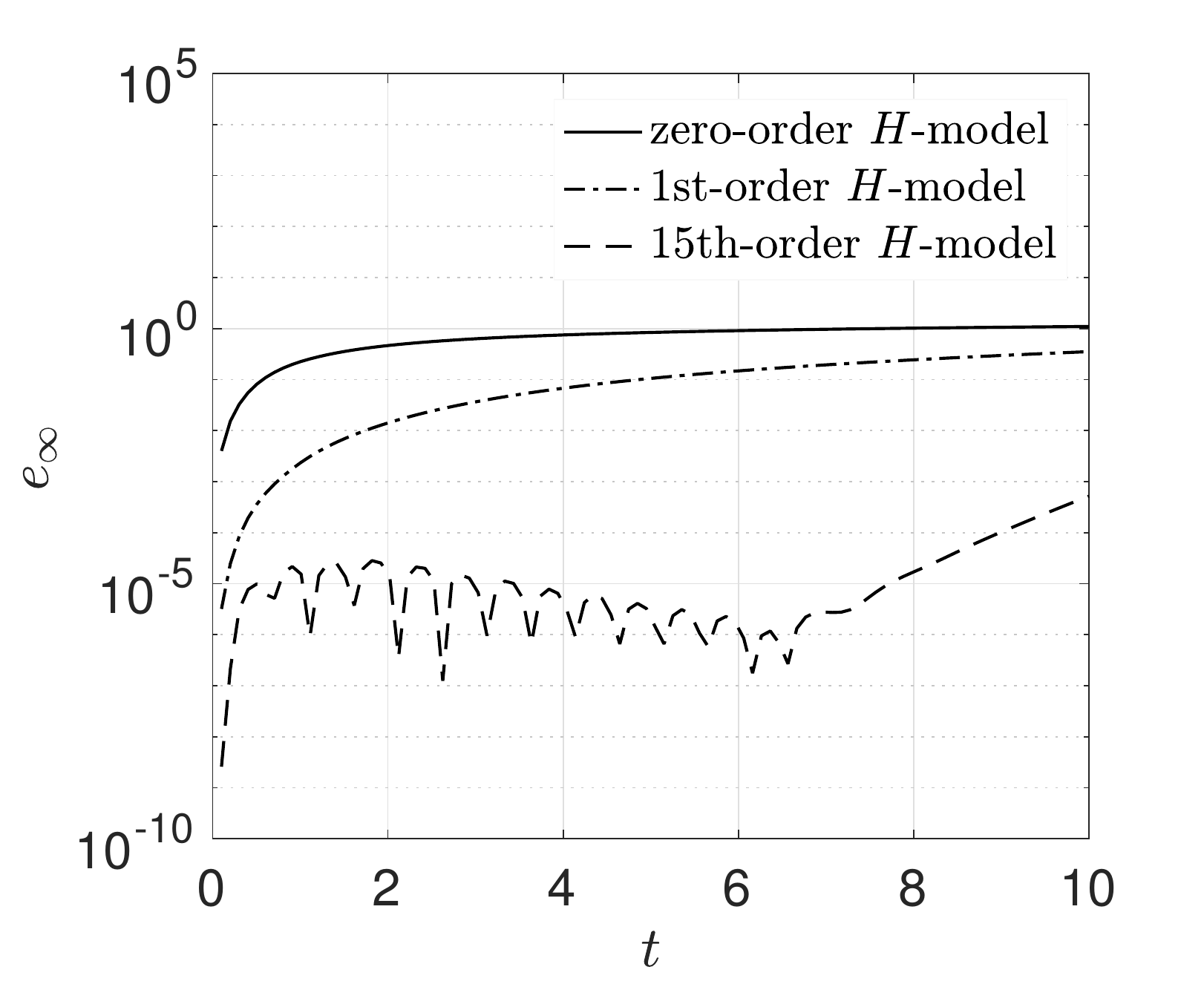}
}

\caption{\red{Convergence of the $H$-model for the 
linear dynamical system with matrix \eqref{matrix_A}. The benchmark 
solution is computed with Monte-Carlo (MC) simulation. Also, the 
zero-order $H$-model represents the Markovian approximation 
to the MZ equation, i.e. the MZ equation without the memory term.}}
\label{fig:linear_3d}
\end{figure}

\begin{figure}[t]
\centerline{\hspace{0.4cm}(a)\hspace{5.3cm}(b)\hspace{5.2cm}(c)}
\centerline{
\includegraphics[height=4.5cm]{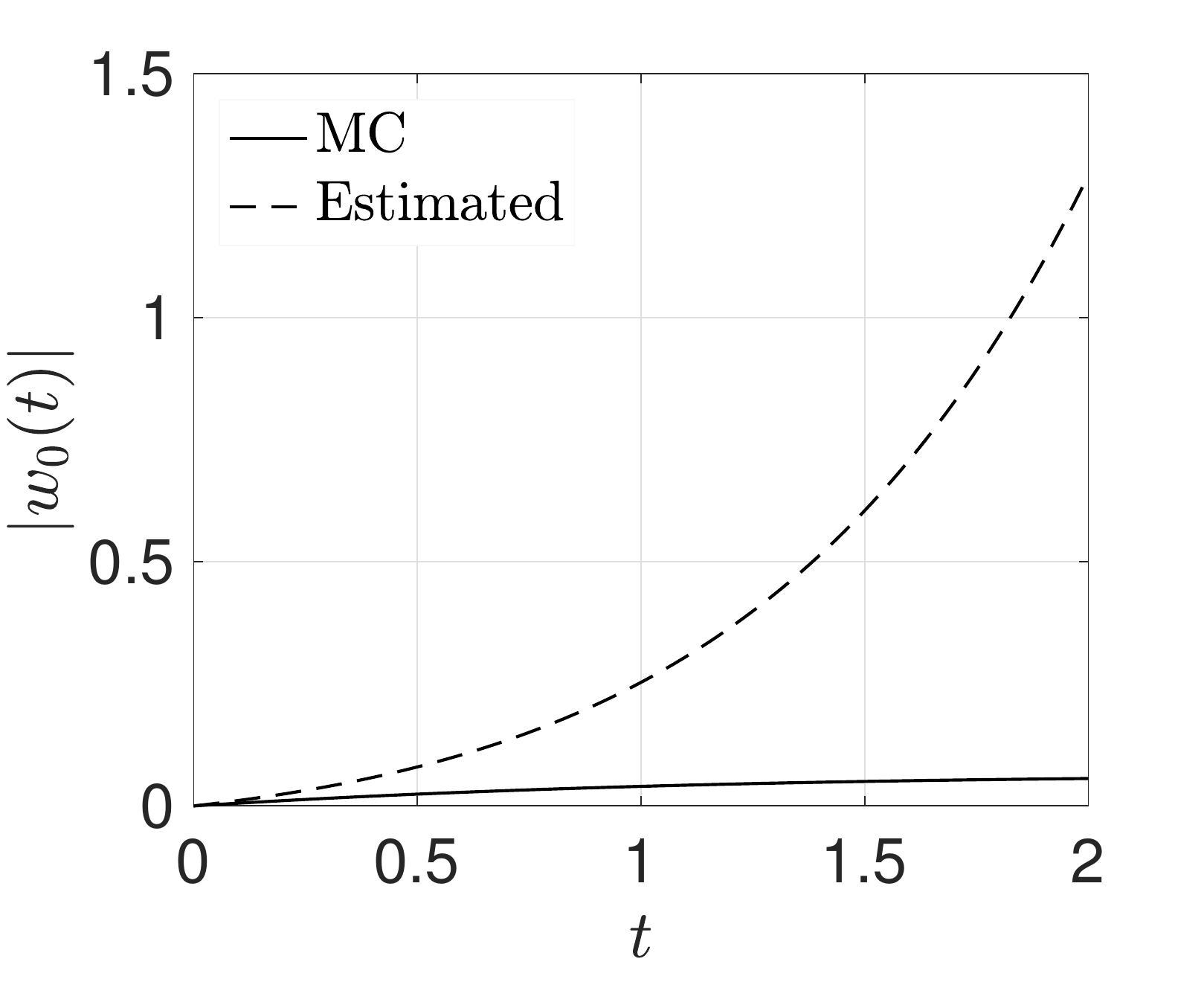}
\includegraphics[height=4.5cm]{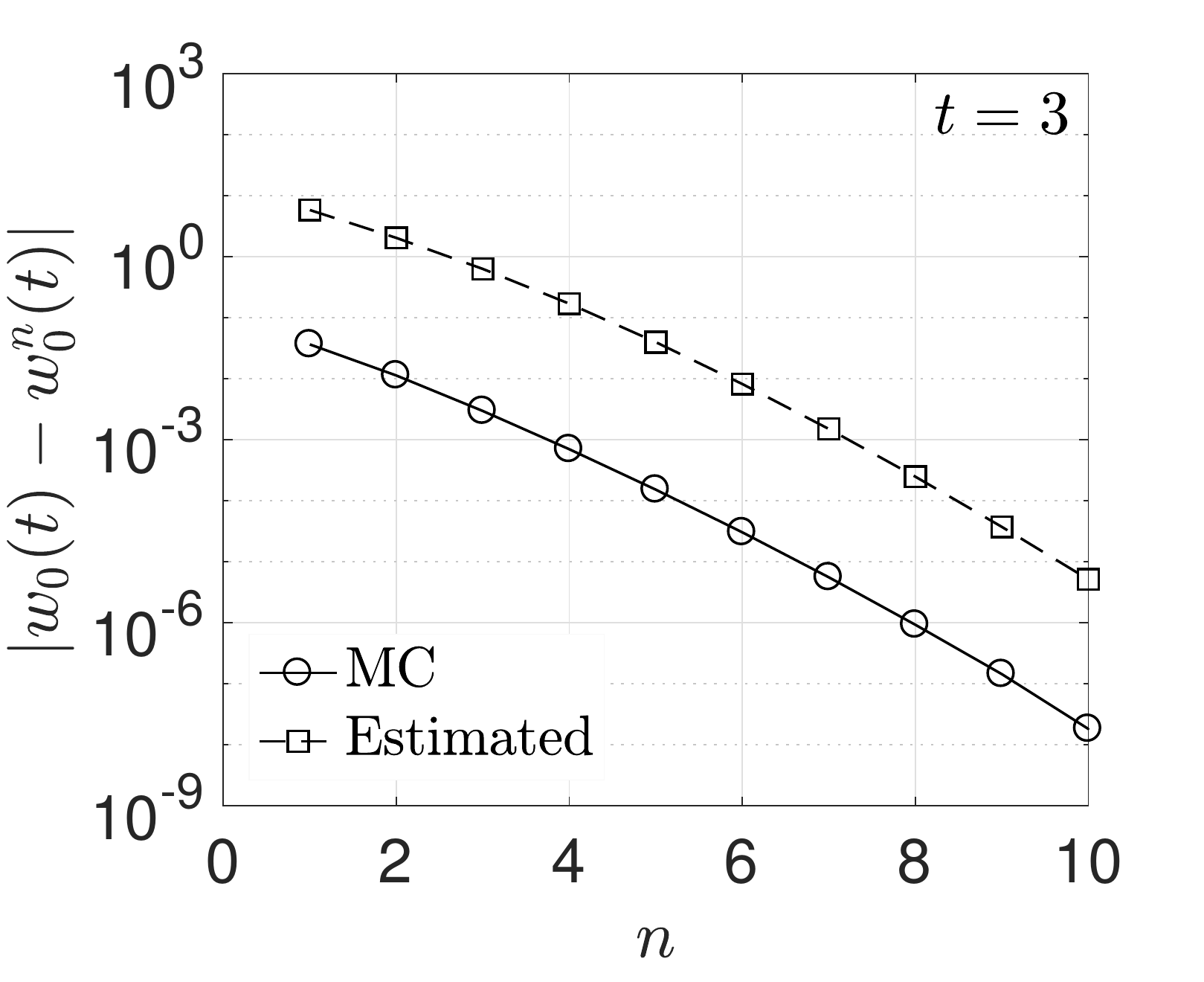} 
\includegraphics[height=4.5cm]{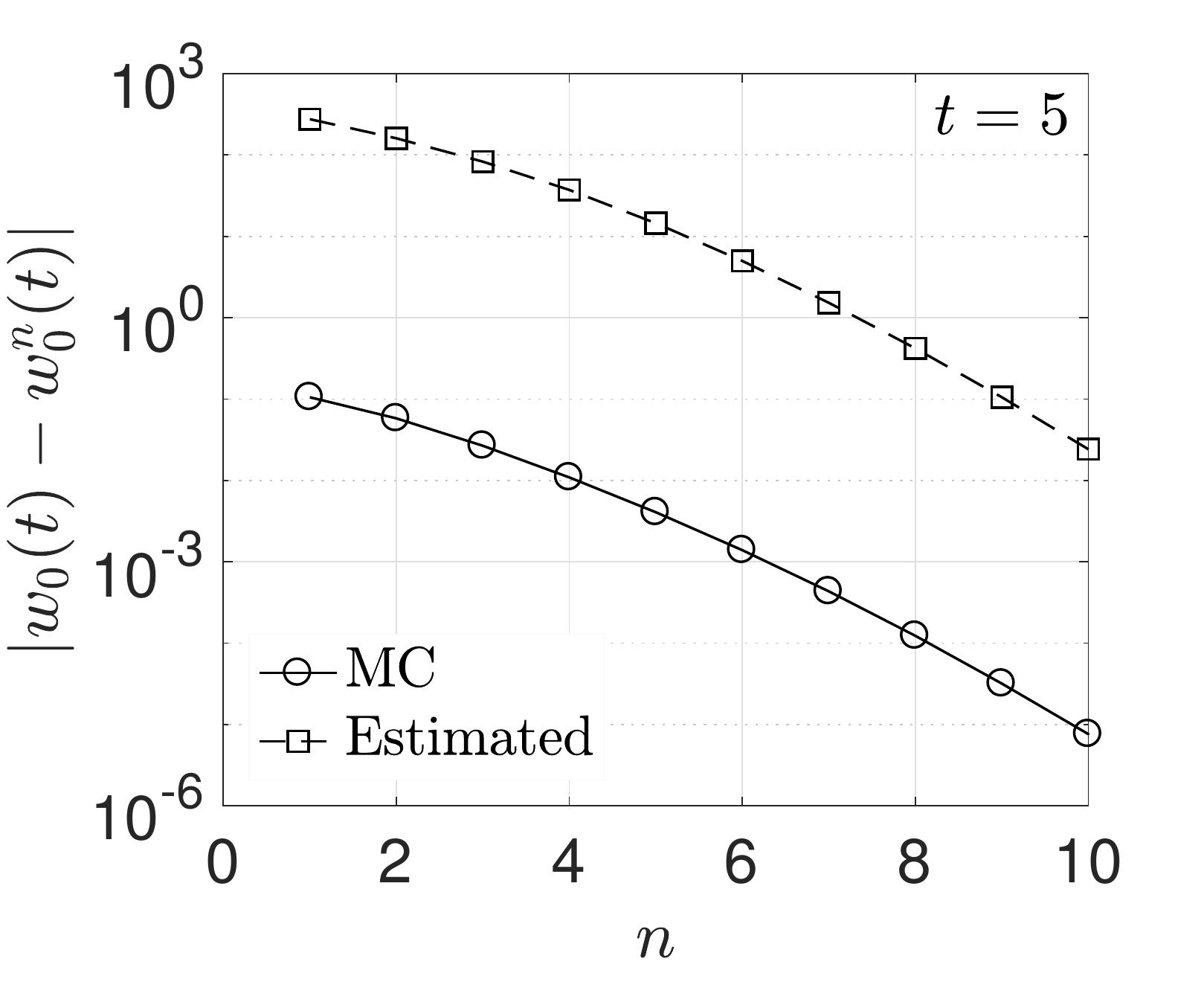}
}
\caption{\red{Linear dynamical system with matrix \eqref{matrix_A}. 
In (a) we plot the memory term $w_0(t)$ we obtain from Monte Carlo 
simulation together with the estimated upper bound \eqref{prior_1}. 
In (b) and (c) we plot $H$-model approximation error $|w_0(T)-w_0^n(T)|$ together with the upper bound \eqref{prior_2} for different differentiation orders $n$ and at different times $t$.}}
\label{fig:Linearbounds}
\end{figure}

\paragraph{Remark}  
The results we just obtained can be obviously extended to 
higher-dimensional linear dynamical systems. 
In Figure \ref{fig:linear_nd} we plot the benchmark conditional 
mean path we obtained through Monte Carlo simulation
together with the solution of the $H$-model \eqref{hier_equation1lor}  
for the $100$-dimensional linear dynamical system defined by the matrix
($N=100$)
\begin{align}
A=
\begin{bmatrix}
    -1 & 1 &  & \dots  & (-1)^N\\
    1 &  & &  & \\
    \vdots & &  & B &  \\
    1 &  &  &  & 
\end{bmatrix},
\label{Alin1}
\end{align}
where $B=e^{C}\Lambda e^{-{C}}$ and  
\begin{align*}
{\Lambda}=
\begin{bmatrix}
    -\frac{1}{8} & 0  & \cdots & 0  \\
    0 & -\frac{2}{9} &  & \vdots \\
     \vdots & & \ddots &0\\
    0 &\cdots & 0&-\frac{N-1}{N+6}
\end{bmatrix},
\qquad
{C}=
\begin{bmatrix}
    0 & 1 &  &   0 \\
     -1& 0 & \ddots& \\
     & \ddots& \ddots &1\\
    0 & & -1&0 
\end{bmatrix}.
\end{align*}
It is seen that the $H$-model converges as we increase the differentiation order in any finite time interval, in agreement with the theoretical prediction 
of section \ref{sec:linearDyn}.
\begin{figure}[t]
\centerline{
\includegraphics[ height=6cm]{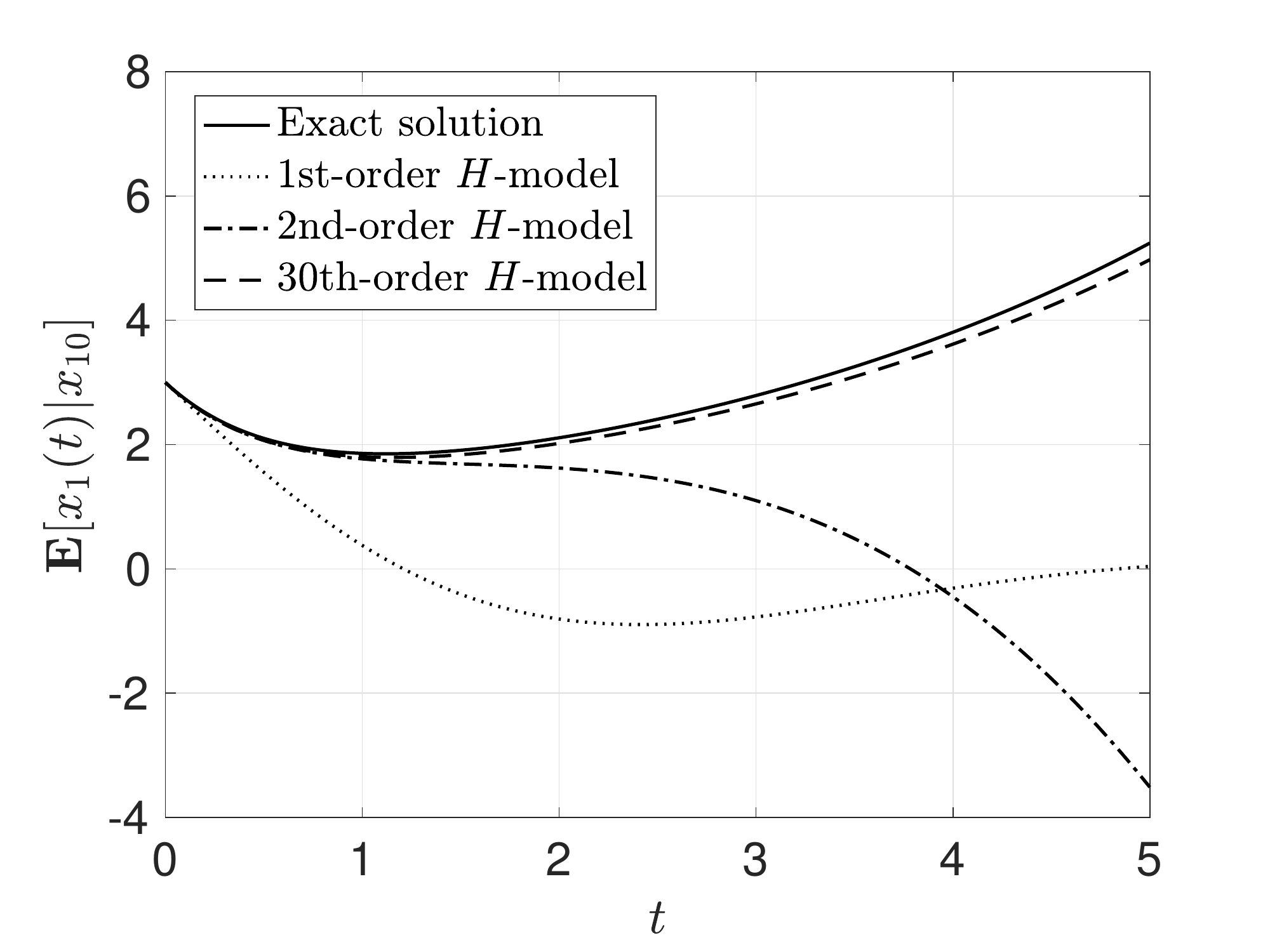} 
\includegraphics[ height=6cm]{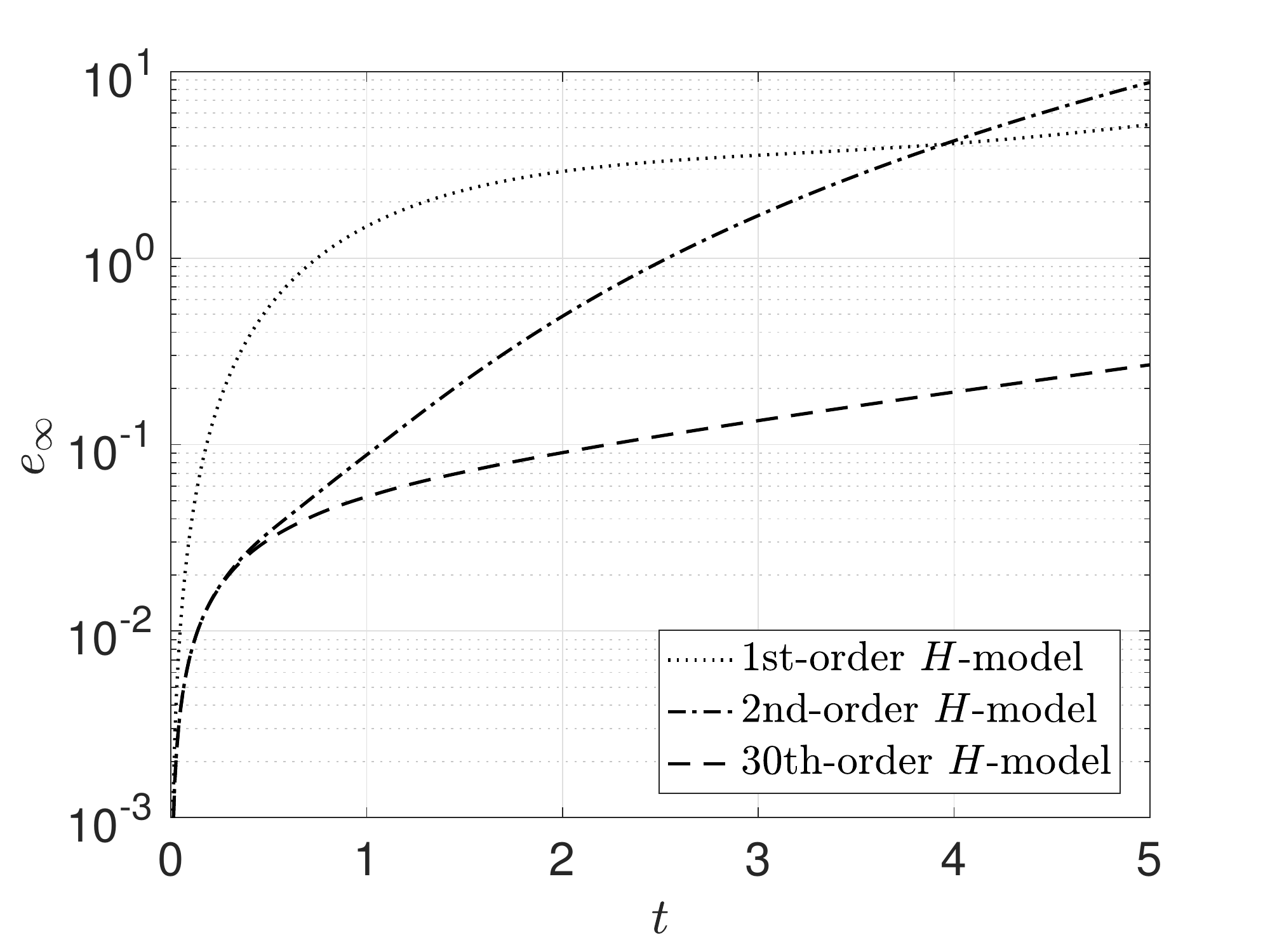}
}
\caption{
Linear dynamical system with matrix $A$ \eqref{Alin1}. 
Convergence of the  $H$-model to the conditional mean path solution $\mathbb{E}[x_1(t)|x_1(0)]$. The initial condition is set as $x_{1}(0)=3$, while $\{x_{2}(0), \dots , x_{100}(0)\}$ are i.i.d. Normals.}
\label{fig:linear_nd}
\end{figure}
\subsubsection{Nonlinear Dynamical Systems}
The hierarchical memory approximation method we 
discussed in section \ref{sec:hierarchical} 
can be applied to nonlinear dynamical systems
in the form \eqref{eqn:nonautonODE}. 
As we will see, if we employ the $H_t$-model 
then the nonlinearity introduces a closure 
problem that needs to be addressed properly. 

\paragraph{Lorenz-63 System}
\label{sec:Lor63}
Consider the classical Lorenz-63 model 
\begin{align}\label{lorenz_equation}
\begin{dcases}
\dot{x}_1=\sigma(x_2-x_1)\\
\dot{x}_2=x_1(r-x_3)-x_2\\
\dot{x}_3=x_1x_2-\beta x_3
\end{dcases}
\end{align}
where $\sigma=10$ and $\beta=8/3$.
The phase space Liouville operator for this ODE is
\begin{align*}
\L=\sigma(x_2-x_1)\frac{\partial}{\partial x_1}+
(x_1(r-x_3)-x_2)\frac{\partial}{\partial x_2}
+(x_1x_2-\beta x_3)\frac{\partial}{\partial x_3}.
\end{align*}
We choose the resolved variables to be $ \hat{x}=\{x_1,x_2\}$ 
and aim at formally integrating out $\tilde{x}={x_3}$ by using 
the Mori-Zwanzig formalism. To this end, we set 
$x_{3}(0)\sim \mathcal{N}(0,1)$ and consider 
the zeroth-order $H_t$-model ($t$-model)
\begin{align}
\begin{dcases}
\frac{dx_{1m}}{dt}=\sigma(x_{1m}-x_{2m}),\\
\frac{dx_{2m}}{dt}=-x_{2m}+rx_{1m}-tx_{1m}^2x_{2m},
\end{dcases}
\label{lorenz_mean_path_equation}
\end{align}
where $x_{1m}(t)=\mathbb{E}[x_1(t)|x_{1}(0),x_{2}(0)]$ and $x_{2m}(t)= \mathbb{E}[x_2(t)|x_{1}(0),x_{2}(0)]$ are {\em conditional mean paths}.
To obtain this system we introduced the following 
mean field closure approximation
\begin{align}
t\P e^{t\L}\P\L\Q\L x_2(0) = & -t \mathbb{E}[x_1(t)^2 x_2(t)| x_1(0),x_2(0)],\nonumber\\
\simeq &  -t \mathbb{E}[x_1(t)| x_1(0),x_2(0)]^2 
\mathbb{E}[x_2(t)| x_1(0),x_2(0)],\nonumber\\
=& -t x_{1m}^2 x_{2m}. \label{MFC}
\end{align}
Higher-order $H_t$-models can be derived based 
on \eqref{MFC}. 
\begin{figure}[t]
\centerline{
\includegraphics[height=6cm]{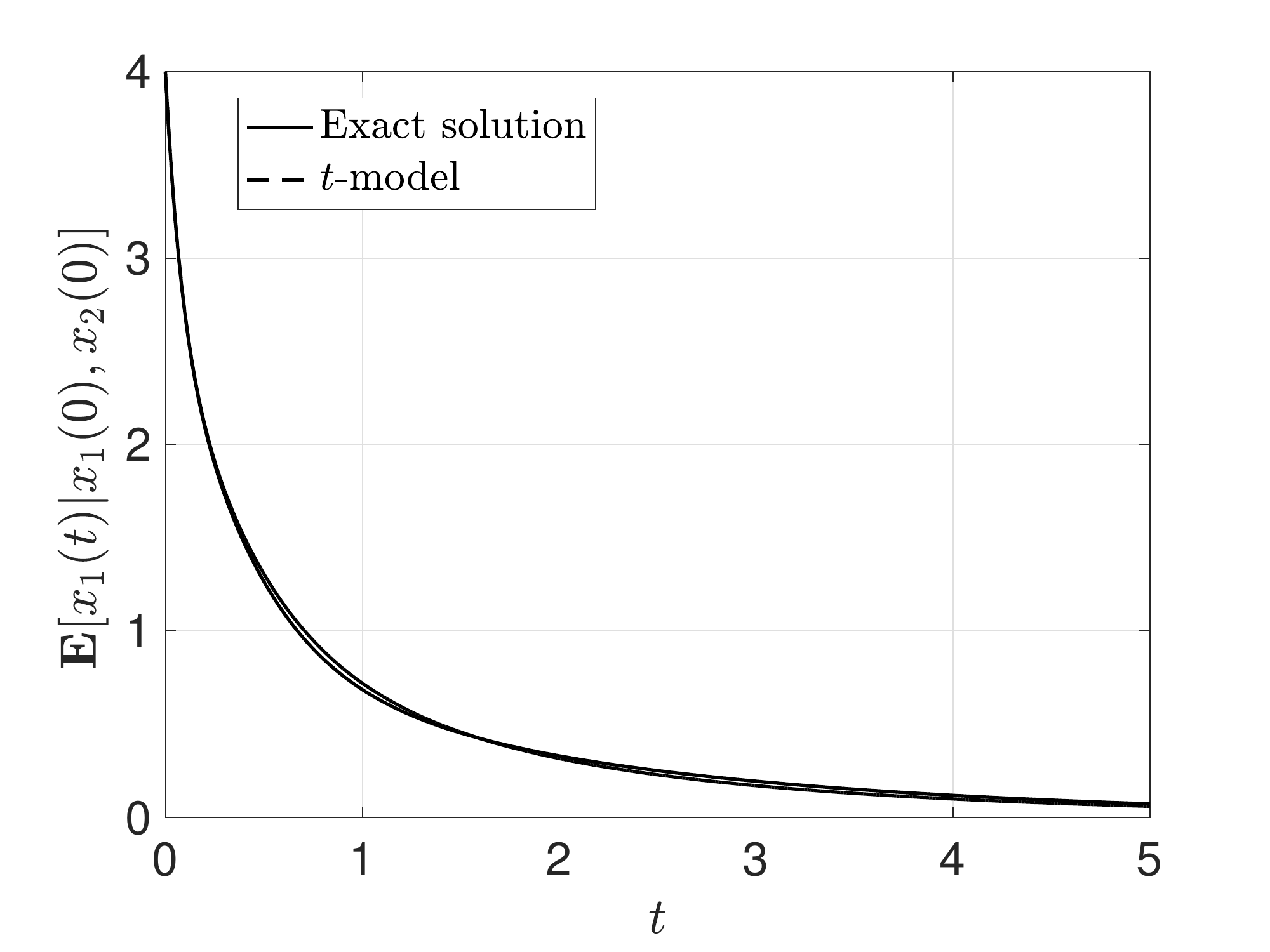} 
\includegraphics[height=6cm]{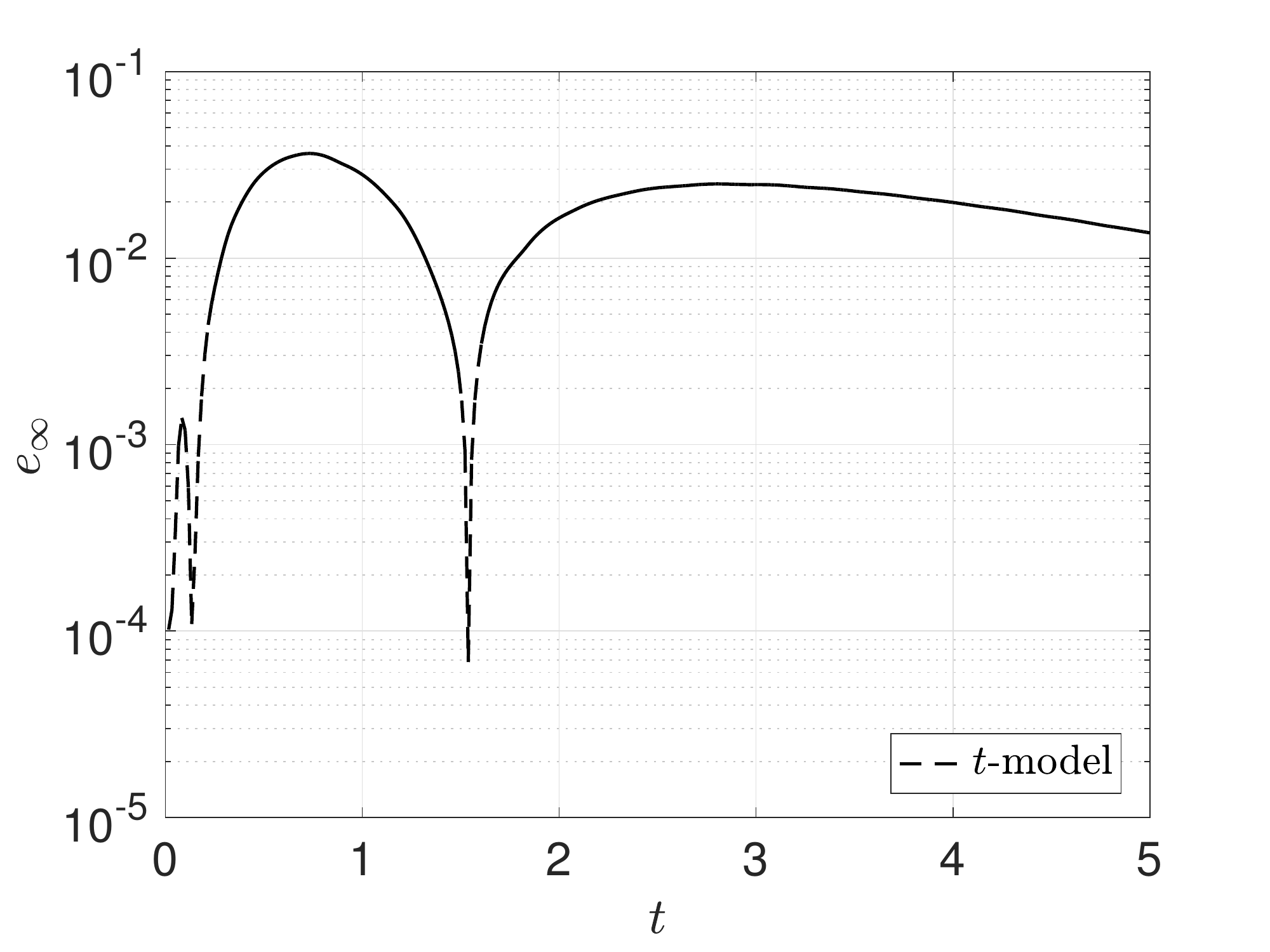}}
\centerline{
\includegraphics[height=6cm]{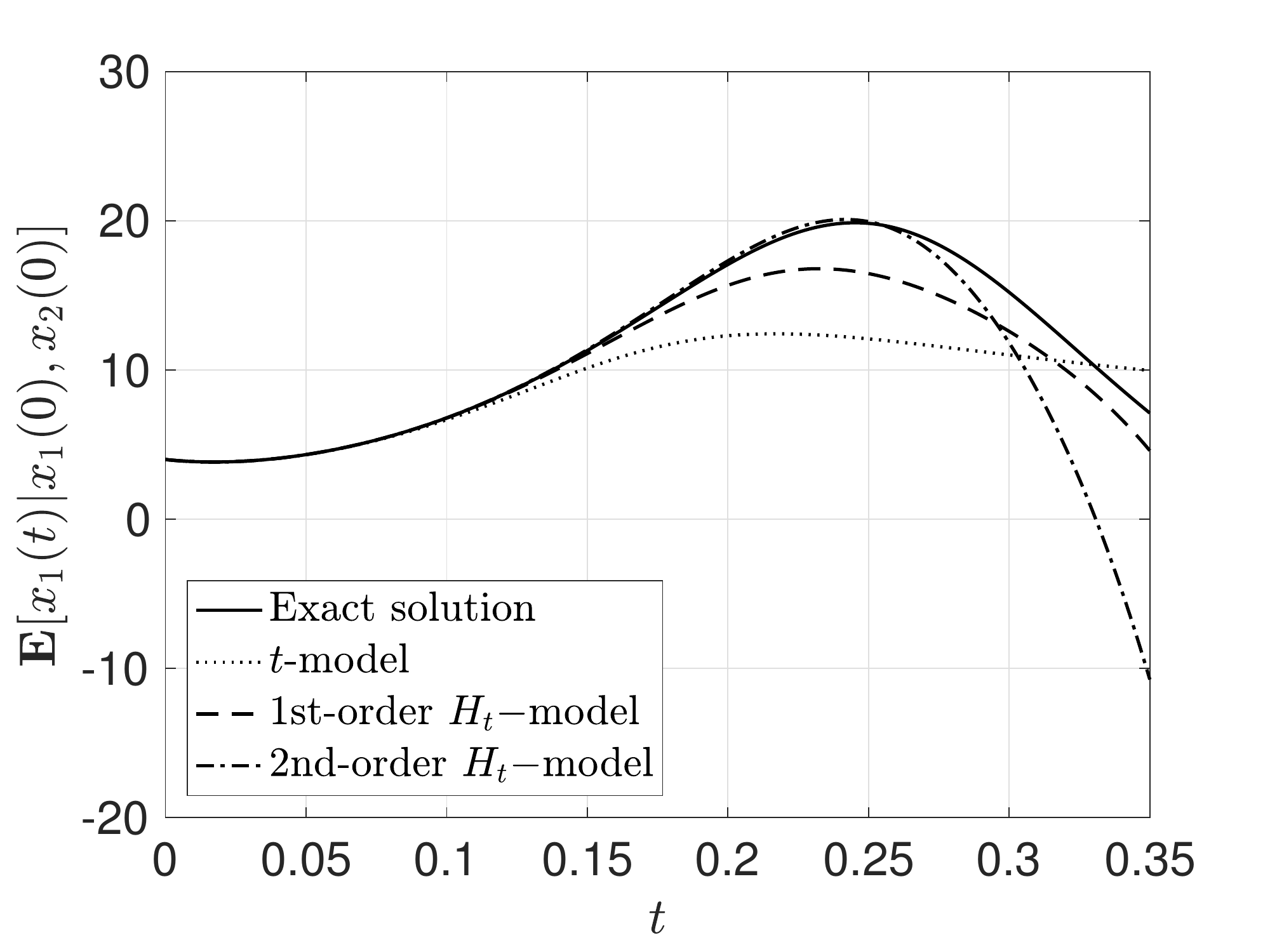}
\includegraphics[height=6cm]{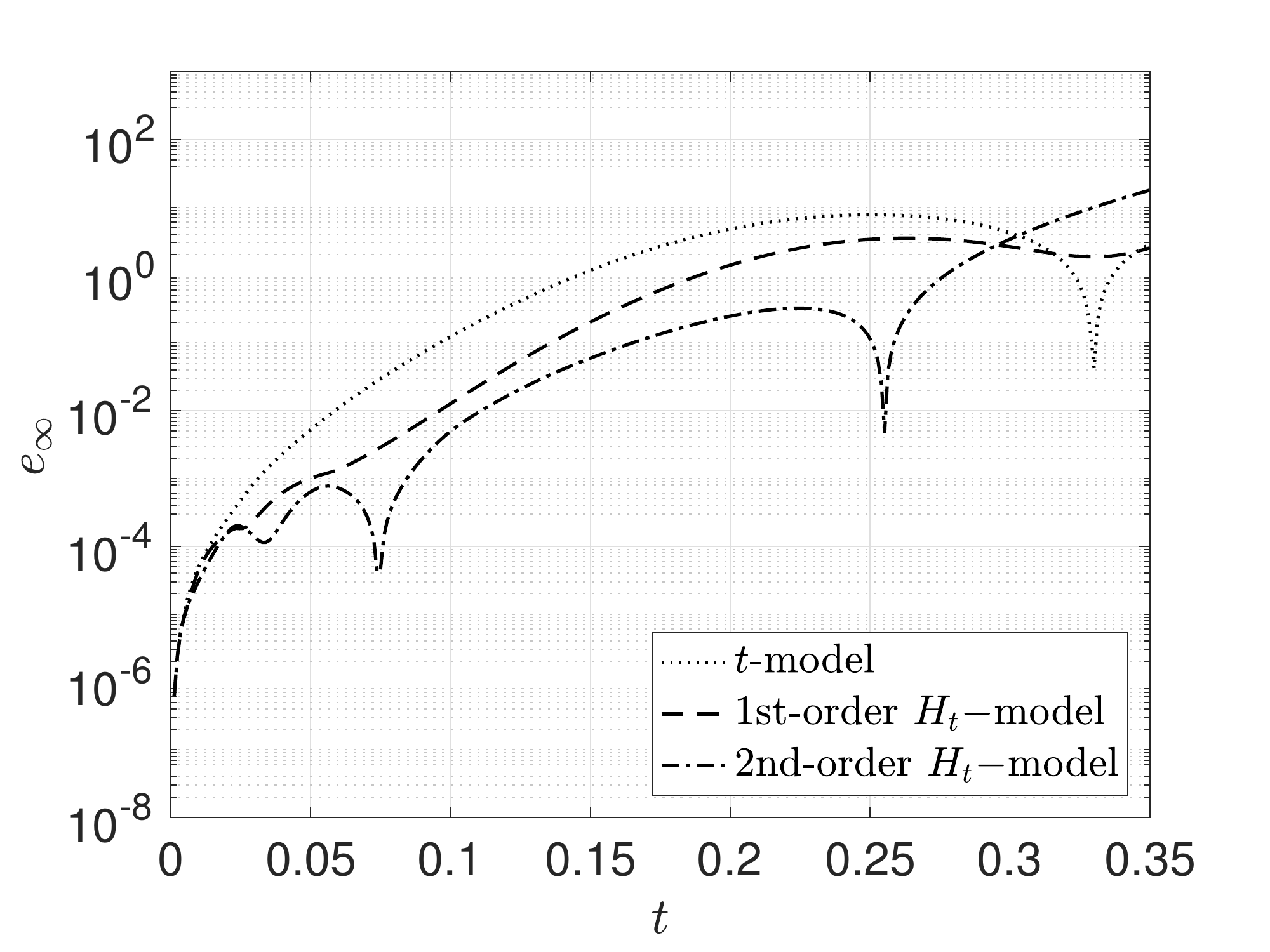}
}

\caption{Accuracy of the $H_t$ model in representing the conditional 
mean path in the Lorenz-63 system \eqref{lorenz_equation}. 
It is seen that if $r=0.5$ (first row), then the zeroth-order 
$H_t$-model, i.e., the $t$-model, is accurate for long integration times. 
On the other hand, if we consider the chaotic regime 
at $r=28$ (second row) then we see that the $t$-model 
and its high-order extension ($H_t$-model) are 
accurate only for relatively short time.}
\label{fig:lorenz_results}
\end{figure}
\noindent
As is well known, if  $r<1$, the fixed point $(0,0,0)$ 
is a global attractor and exponentially stable. 
In this case, the $t$-model (zeroth-order $H_t$-model) 
yields accurate prediction of the conditional mean path 
for long time (see Figure \ref{fig:lorenz_results}).
On the other hand, if we consider the chaotic regime at 
$r=28$ then the $t$-model and its higher-order extension, 
i.e., the $H_t$-model, are accurate only for 
relatively short time. This is in agreement with 
our theoretical predictions. In fact, 
different from linear systems where the hierarchical representation 
of the memory integral can be proven to be convergent for long time, in nonlinear systems the memory hierarchy is, in general,  
provably convergent only in a short time period (Theorem \ref{Thm_Decaying_error_Ht} and Corollary \ref{cor_decaying}). 
This doesn't mean that the $H$-model or the 
$H_t$-model are not accurate for nonlinear systems. 
It just means that the accuracy depends on the system, 
the quantity of interest, and the initial condition.

\paragraph{Modified Lorenz-96 system.}
\red{As an example of a high dimensional nonlinear dynamical system, we consider the following modified Lorenz-96 system \cite{Karimi,Lorenz96}}
\begin{align}
\begin{dcases}
\dot{x}_1=-x_1+x_1x_2+F\\
\dot{x}_2=-x_2+x_1x_3+F\\
\quad\vdots\\
\dot{x_i}=-x_i+(x_{i+1}-x_{i-2})x_{i-1}+F\\
\quad\vdots\\
\dot{x}_N=x_N-x_{N-2}x_{N-1}+F
\end{dcases}
\label{equ:lorenz96}
\end{align}
where $F$ is constant.  As is well known, depending on 
the values of $N$ and $F$ this system can exhibit a wide range 
of behaviors \cite{Karimi}.  
Suppose we take the resolved variables to be $\hat {x}=\{x_1,x_2\}$. Correspondingly, 
the unresolved ones, i.e., those we aim at integrating through the MZ framework, are $\tilde {x}=\{x_3,\dots ,x_N\}$, which we set to be independent standard normal random variables. 
By using the mean field approximation \eqref{MFC}, we obtain 
the following zeroth-order $H_t$-model ($t$-model) of the 
modified Lorenz-96 system is \eqref{equ:lorenz96} 
\begin{align}
\begin{dcases}
\dot{x}_{1m}=-x_{1m}+x_{1m}x_{2m}+F,\\
\dot{x}_{2m}=-x_{2m}+F+t(x_{1m}^2x_{2m}-x_{1m}F).\\
\end{dcases}
\label{lorenz96_mean_path_equation}
\end{align}
In Figure \ref{fig:lorenz_r05} we study the accuracy of 
the $H_t$-model in representing the conditional 
mean path for with $F=5$ and $N=100$. 
 It is seen that the the $H_t$-model 
converges only for short time (in agreement with the 
theoretical predictions) and it provides results that 
are more accurate that the classical $t$-model.

\begin{figure}[t]
\centerline{
\includegraphics[height=6cm]{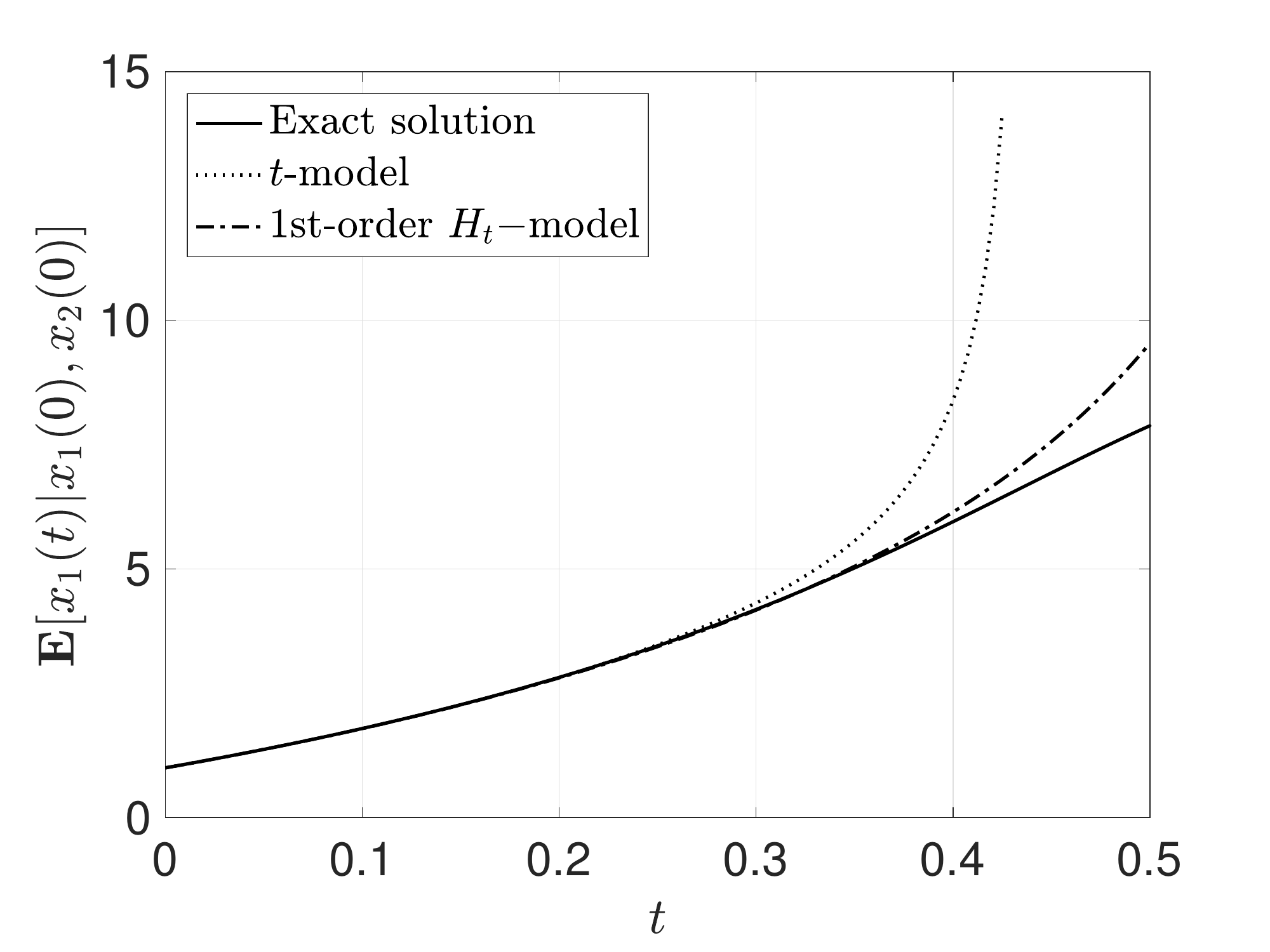} 
\includegraphics[height=6cm]{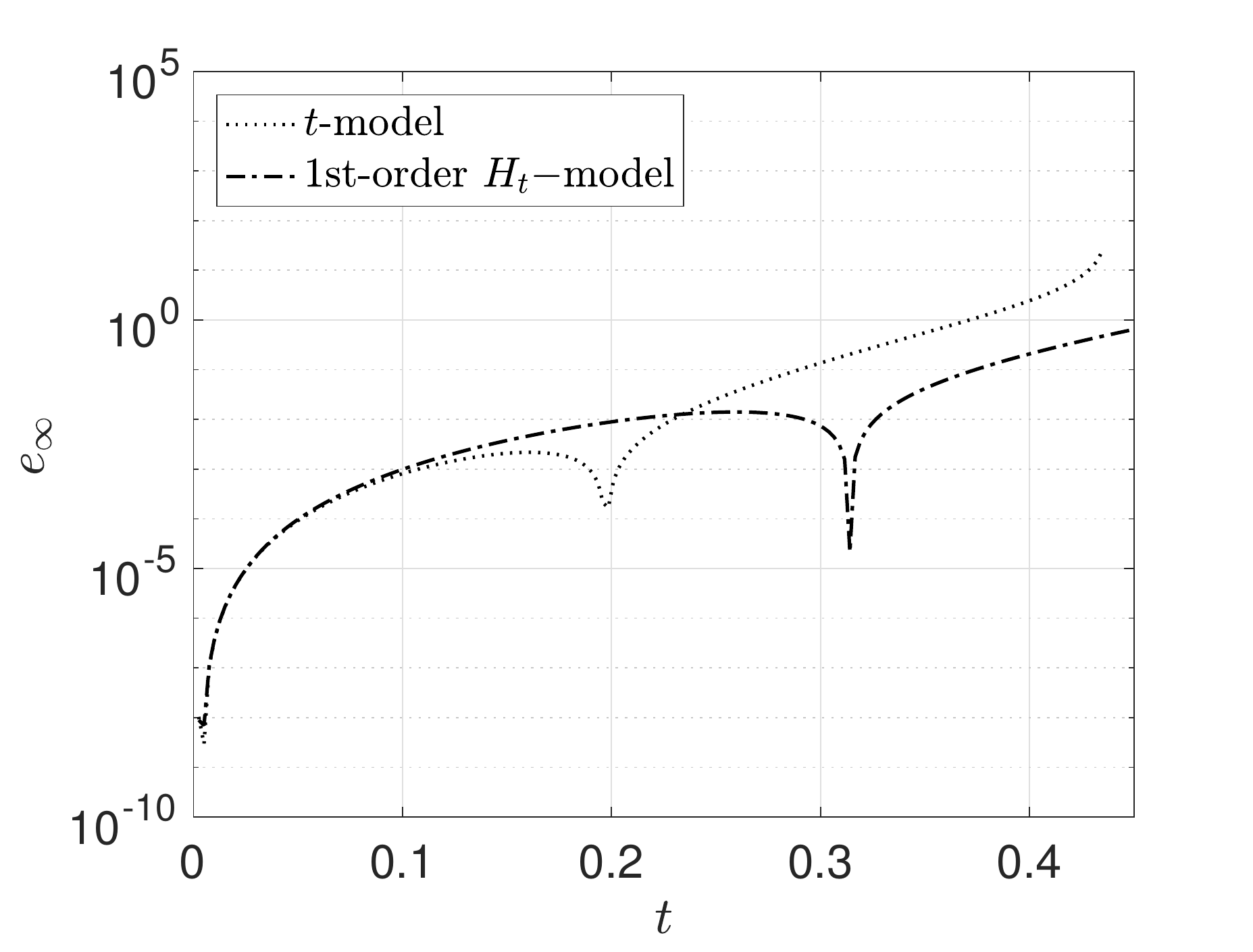}}

\caption{Accuracy of the $H_t$-model in representing the conditional 
mean path in the Lorenz-96 system \eqref{lorenz_equation}. 
Here we set $F=5$ and $N=100$.  It is seen that the $H_t$-model 
converges only for short time and provides results that 
are more accurate that the classical $t$-model.}
\label{fig:lorenz_r05}
\end{figure}

\section{Summary}
\label{sec:summary}
\red{
In this paper we developed a thorough 
mathematical analysis to deduce conditions 
for accuracy of different approximations of the memory integral 
in the Mori-Zwanzig equation, and, more importantly, 
whether the algorithms to approximate such memory 
integral converge. In particular, we studied the short memory 
approximation, the $t$-model and various hierarchical 
memory approximation techniques. We also derived computable 
upper bounds for the MZ memory integral, which 
allowed us to estimate a priori the contribution of the memory 
to the dynamics.
To the best of our knowledge, this is the first time 
rigorous convergence analysis is presented on 
approximations of the MZ memory integral. 
We found that for a given nonlinear dynamical 
system and quantity of interest, the approximation error
can be controlled by setting constraints 
on the initial condition of the system, i.e., by preparing 
the system appropriately. 
We have also established rigorous 
convergence results for hierarchical memory 
approximation methods such as the $H$-model, 
the Type-I and Type II finite memory approximations, 
and the $H_t$ model. These methods converge for 
any finite integration time in the case of linear 
dynamical systems. However, for general nonlinear systems, the memory 
approximation problem remains challenging and convergence 
of the hierarchical methods we discussed in this paper 
can be granted only for short time, or on a case-by-case 
basis. We presented simple numerical examples 
demonstrating convergence of the $H$-model and $H_t$-model for prototype linear and nonlinear  dynamical systems. 
The numerical results are found to be in agreement 
with the theoretical predictions.
}

\subsection*{Acknowledgements}
This work was supported by the Air Force Office of Scientific 
Research Grant No. FA9550-16-586-1-0092.

\appendix
\section{Semigroup Bounds via  Function Decomposition}
\label{sec:semigroupBoundsDecomposition}
\red{
In looking for the \emph{numerical abscissa} \cite{Davies2005} (i.e., the logarithmic norm) of $\LV\PrjC$, we seek to bound 
\begin{align*}
	\sup_{\mathcal{D}(\LV\PrjC)\ni x\neq 0}\Re\frac{\langle x, \LV\PrjC x\rangle_{\sigma}}{\langle x, x\rangle_{\sigma}}.
\end{align*}
Notice that, if $\Prj \LV \PrjC =0$, then $\LV\PrjC = \PrjC\LV\PrjC$, so that we have the previously proven bound 
\begin{align*}
	\sup\Re\frac{\langle x, \PrjC \LV \PrjC x \rangle_{\sigma}}{\langle x, x\rangle_{\sigma}} \leq -\frac{1}{2}\inf\Div_{\sigma}\mathbf{F}.
\end{align*}
In this section, we consider what happens when $\Prj \LV\PrjC \neq 0$.  To that end, let us note that $x\in\mathcal{D}(\LV\PrjC)$ may be decomposed as
\begin{align*}
	x = \PrjC x + \alpha \Prj \LV\PrjC x + \Prj y
\end{align*}
where $\alpha\in\mathbb{C}$ and $\Prj y$ is orthogonal to $\Prj\LV\PrjC x$.  In other words, we define $\Prj y$ as
\begin{align*}
	\Prj y := \Prj x - \frac{\langle \Prj\LV\PrjC x, \Prj x\rangle_{\sigma}}{\langle \Prj\LV\PrjC x, \Prj\LV\PrjC x\rangle_{\sigma}}\Prj\LV\PrjC x,
\end{align*} 
and define $\alpha$ as 
\begin{align*}
	\alpha := \frac{\langle \Prj\LV\PrjC x, \Prj x\rangle_{\sigma}}{\langle \Prj\LV\PrjC x, \Prj\LV\PrjC x\rangle_{\sigma}}.
\end{align*}
Then
\begin{subequations}
\begin{align*}
	\Re\frac{\langle x, \LV\PrjC x\rangle_{\sigma}}{\langle x, x\rangle_{\sigma}} & = \Re\frac{\langle\PrjC x + \alpha \Prj \LV\PrjC x + \Prj y, \LV\PrjC x\rangle_{\sigma}}{\langle x, x\rangle_{\sigma}},\\
	& = \frac{\Re\langle\PrjC x, \LV\PrjC x\rangle_{\sigma} + \Re(\alpha)\|\Prj \LV\PrjC x\|_{\sigma}^{2}}{\|\PrjC x\|_{\sigma}^{2} + |\alpha|^{2}\|\Prj \LV\PrjC x\|_{\sigma}^{2} + \|\Prj y\|_{\sigma}^{2}},\\
	& \leq \max\left[0,\;\frac{\Re\langle\PrjC x, \LV\PrjC x\rangle_{\sigma} + \Re(\alpha)\|\Prj \LV\PrjC x\|_{\sigma}^{2}}{\|\PrjC x\|_{\sigma}^{2} + |\alpha|^{2}\|\Prj \LV\PrjC x\|_{\sigma}^{2}}\right].
\end{align*}
\end{subequations}
Since we assume $\Prj \LV\PrjC \neq 0$, there exists $x$ such that $\Prj \LV\PrjC x \neq 0$ and then, for any $\alpha$ such that 
\begin{align*}
	\Re(\alpha) \geq \frac{\Re\langle\PrjC x, \LV\PrjC x\rangle_{\sigma}}{\|\Prj \LV\PrjC x\|_{\sigma}^{2}},
\end{align*}
we have
\begin{align*}
	\Re\langle\PrjC x, \LV\PrjC x\rangle_{\sigma} + \Re(\alpha)\|\Prj \LV\PrjC x\|_{\sigma}^{2} \geq 0,
\end{align*}
so that, for $\Prj \LV\PrjC \neq 0$,
\begin{subequations}
\begin{align*}
	\sup_{\mathcal{D}(\LV\PrjC)\ni x\neq 0}\Re\frac{\langle x, \LV\PrjC x\rangle_{\sigma}}{\langle x, x\rangle_{\sigma}}  = \sup_{\mathcal{D}(\LV\PrjC)\ni x\neq 0}\frac{\Re\langle\PrjC x, \LV\PrjC x\rangle_{\sigma} + \Re(\alpha)\|\Prj \LV\PrjC x\|_{\sigma}^{2}}{\|\PrjC x\|_{\sigma}^{2} + |\alpha|^{2}\|\Prj \LV\PrjC x\|_{\sigma}^{2}}.
\end{align*}
\end{subequations}
Now, fix any $\PrjC x\in\mathcal{D}(\LV)\neq 0$ and consider 
the expression
\begin{subequations}
\begin{align*}
	\frac{\Re\langle\PrjC x, \LV\PrjC x\rangle_{\sigma} + \Re(\alpha)\|\Prj \LV\PrjC x\|_{\sigma}^{2}}{\|\PrjC x\|_{\sigma}^{2} + |\alpha|^{2}\|\Prj \LV\PrjC x\|_{\sigma}^{2}}
	= \frac{\xi + \Re(\alpha)\beta^{2}}{1 + |\alpha|^{2}\beta^{2}}
\end{align*}
\end{subequations}
where
\begin{subequations}
\begin{align*}
	\xi  = \xi(\PrjC x) = \Re\frac{\langle\PrjC x, \LV\PrjC x\rangle_{\sigma}}{\|\PrjC x\|_{\sigma}^{2}}, \qquad  \beta  = \beta(\PrjC x) = \frac{\|\Prj \LV\PrjC x\|_{\sigma}}{\|\PrjC x\|_{\sigma}}.
\end{align*}
\end{subequations}
Then, for this fixed $\PrjC x$, 
\begin{align*}
	\frac{\xi + \Re(\alpha)\beta^{2}}{1 + |\alpha|^{2}\beta^{2}} \leq \max_{a\in\mathbb{R}}\frac{\xi + a\beta^{2}}{1 + a^{2}\beta^{2}}.
\end{align*}
Differentiating w.r.t. $a$ and setting equal to zero, we find that the latter expression is extremized when
\begin{align*}
	0 & = \beta^{2}(1+a^{2}\beta^{2}) - 2a\beta^{2}(\xi + a\beta^{2}),
\end{align*}
i.e., when
\begin{align*}
	\beta^{2}a^{2} + 2\xi a - 1  & = 0,
\qquad 
	a = \frac{-\xi \pm\sqrt{\xi^{2}+\beta^{2}}}{\beta^{2}}.
\end{align*}
Since $\beta^{2}> 0$,  $\displaystyle \frac{\xi + a\beta^{2}}{1 + a^{2}\beta^{2}} $ is maximized at $\displaystyle \hat{a} = \frac{-\xi +\sqrt{\xi^{2}+\beta^{2}}}{\beta^{2}}$. Then
\begin{subequations}
\begin{align*}
	\xi + \hat{a}\beta^{2} = \sqrt{\xi^{2}+\beta^{2}},\qquad
	1 + \hat{a}^{2}\beta^{2} = 2\left(1 - \xi \hat{a}\right) = 2\left(\frac{\xi^{2} + \beta^{2} - \xi\sqrt{\xi^{2} + \beta^{2}}}{\beta^{2}}\right),
\end{align*}
\end{subequations}
so that
\begin{subequations}
\begin{align*}
\max_{a\in\mathbb{R}}\frac{\xi + a\beta^{2}}{1 + a^{2}\beta^{2}} = \frac{\xi + \hat{a}\beta^{2}}{1 + \hat{a}^{2}\beta^{2}}
 = \frac{1}{2}\frac{\beta^{2}}{\sqrt{\xi^{2} + \beta^{2}} - \xi}
 = \frac{1}{2}\left(\sqrt{\xi^{2} + \beta^{2}} + \xi\right).
\end{align*}
\end{subequations}
Therefore,
\begin{align*}
	\sup_{\mathcal{D}(\LV\PrjC)\ni x\neq 0}\Re\frac{\langle x, \LV\PrjC x\rangle_{\sigma}}{\langle x, x\rangle_{\sigma}} = \sup_{\mathcal{D}(\LV)\ni (\PrjC x)\neq 0}\frac{1}{2}\left[\sqrt{\xi^{2}(\PrjC x) + \beta^{2}(\PrjC x)} + \xi(\PrjC x)\right].
\end{align*}
When $\Prj\LV\PrjC$ is unbounded, which is the typical case when $\Prj$ is an infinite-rank projection, such as most conditional expectations, there is unlikely to be a finite numerical abscissa for $\LV\PrjC$.  In particular, notice that if $\Div_{\sigma}({F})$ is a bounded function (bounded both above and below), then $\xi(\PrjC x)$ is bounded for all $\PrjC x$ while $\beta(\PrjC x)$ is unbounded, in which case
\begin{align*}
	\sup_{\mathcal{D}(\LV\PrjC)\ni x\neq 0}\Re\frac{\langle x, \LV\PrjC x\rangle_{\sigma}}{\langle x, x\rangle_{\sigma}} = \infty.
\end{align*}
It follows \cite{Trefethen1997, Pazy1992} that in these cases, $\|e^{t\LV\PrjC}\|_{\sigma}$ has infinite slope at $t = 0$, and therefore 
there is no finite $\omega$ such that $\|e^{t\LV\PrjC}\|_{\sigma} \leq e^{\omega t}$ for all $t\geq 0$ (see \cite{Davies2005}).  
Assuming still that $\LV\PrjC$ generates a strongly continuous semigroup, we must look to bound the semigroup as $\|e^{t\LV\PrjC}\|_{\sigma} \leq Me^{\omega t}$, where 
\begin{align*}
	\omega > \omega_{0} = \lim_{t\to\infty}\frac{\ln\|e^{t\LV\PrjC}\|_{\sigma}}{t}
\end{align*}
and 
\begin{align*}
	M \geq M(\omega) = \sup\{\|e^{t\LV\PrjC}\|_{\sigma}e^{-\omega t} : t\geq0\}.
\end{align*}
On the other hand, if $\Prj\LV\PrjC$ is a bounded operator, for example when $\Prj$ is a finite-rank projection (e.g., Mori's projection \eqref{MoriProjection}), there exists a finite value for the numerical abscissa.  Indeed, in this case, since $\zeta$ is bounded by $\omega = -\inf\Div_{\sigma}({F})$, the numerical abscissa of $\LV\PrjC$ may be bounded as
\begin{align}
	\omega_{\LV\PrjC} := \sup_{\mathcal{D}(\LV\PrjC)\ni x\neq 0}\Re\frac{\langle x, \LV\PrjC x\rangle_{\sigma}}{\langle x, x\rangle_{\sigma}} \leq \frac{1}{2}\left[\sqrt{\omega^{2} + \|\Prj\LV\PrjC \|_{\sigma}^{2}} + \omega\right]. \label{eqn:LQNumAbsBound1}
\end{align}
Alternatively, in the case of finite rank $\Prj$, the operator $\LV\PrjC$ may be thought of as a bounded perturbation of $\LV$, i.e. $\LV\PrjC = \LV - \LV\PrjC$, and the numerical abscissa of $\LV\PrjC$ can be bounded using the bounded perturbation theorem \cite[III.1.3]{engel1999one}, obtaining
\begin{align}
	\omega_{\LV\PrjC} \leq \omega + \|\LV\Prj\|_{\sigma}.\label{eqn:LQNumAbsBound2}
\end{align}
Either of these bounds for $\omega_{\LV\PrjC}$ can be used to bound the semigroup norm
\begin{subequations}
\begin{align*}
	\|e^{t\LV\PrjC}\|_{\sigma} & \leq e^{\omega_{\LV\PrjC}t} \leq e^{\frac{1}{2}\left(\sqrt{\omega^{2} + \|\Prj\LV\PrjC \|_{\sigma}^{2}} + \omega\right)t}\\
	\|e^{t\LV\PrjC}\|_{\sigma} & \leq e^{\omega_{\LV\PrjC}t} \leq e^{(\omega + \|\LV\Prj\|_{\sigma})t}.	
\end{align*}
\end{subequations}
Which of these two estimates gives the tighter bound will generally depend on the values of $\|\Prj\LV\PrjC\|_{\sigma}$ and $\|\LV\Prj \|_{\sigma}$.  It may be noted, however, that, when $\sigma$ is invariant, $\omega = 0$ and $\LV$ is skew-adjoint, so that 
\begin{align*}
\|\Prj\LV\PrjC\|_{\sigma} = \|\PrjC\LV^{\dag}\Prj\|_{\sigma} = \|\PrjC\LV\Prj\|_{\sigma} \leq \|\LV\Prj\|_{\sigma}
\end{align*}
and therefore the bound in \eqref{eqn:LQNumAbsBound1} is half that of \eqref{eqn:LQNumAbsBound2}. }
\section{The Mori-Zwanzig Formulation in PDF Space}
\label{app:MZPDF}
It was shown in \cite{Dominy2017} that the Banach dual 
of \eqref{MZKoop} defines an evolution in the probability 
density function space. Specifically, the joint probability 
density function of the state vector $u(t)$ that solves equation 
\eqref{eqn:nonautonODE} is pushed forward by the 
Frobenius-Perron operator $\F(t,0)$ (Banach dual of the Koopman
operator \eqref{Koopman})
\begin{equation}
 p(x,t)= \F(t,s) p(x,s), \qquad \F(t,s)=e^{(t-s)\M},
\end{equation}
where 
\begin{equation}
\M(x) p(x,t) =-\nabla\cdot(F(x) p(x,t)).
\end{equation}
By introducing a projection $\P$ in the space of probability density 
functions\footnote{With some abuse of notation we denote the 
projections $\P$ and $\Q$ in the PDF space with the same letter we 
used for projections in the phase space.} and 
its complement $\Q=\I-\P$, it is easy to show
that the projected density $\P p$ satisfies the MZ 
equation \cite{venturi2014convolutionless}
\begin{align}
\frac{\partial \mathcal{P}p(t)}{\partial t}=\mathcal{PMP}p(t)+
\mathcal{P}e^{t\mathcal{QM}}\mathcal{Q}p(0)+
\int_0^t \mathcal{PM}e^{(t-s)\mathcal{QM}}\mathcal{QMP}p(s) ds.
\label{mzpdf1}
\end{align}
In the next sections we perform an analysis of 
different types of approximations of the 
MZ memory integral
\begin{equation}
\int_0^t \mathcal{PM}e^{(t-s)\mathcal{QM}}\mathcal{QMP}p(s) ds.
\label{MemoryPDFSpace}
\end{equation}
The main objective of such analysis is to establish 
rigorous error bounds for widely used approximation 
methods, and also propose new  provably convergent
approximation schemes.

\subsection{Analysis of the Memory Integral}
\label{sec:MZPDFspace}
In this section, we develop the analysis of the memory integral arising 
in the PDF formulation of the MZ equation. 
The starting point is the definition \eqref{MemoryPDFSpace}. 
As before, we begin with the following estimate of upper 
bound estimation of the integral
\begin{theorem}
\label{PDFMgrowth}
{\bf (Memory growth)}
Let $e^{t\mathcal{MQ}}$ and $e^{t\M}$ be 
strongly continuous semigroups with upper bounds 
$\|e^{t\M}\|\leq Me^{t\omega}$ and $\|e^{t\M\Q}\|\leq M_{{\Q}}e^{t\omega_{\Q}}$, and let $T>0$ be a fixed integration time.  Then for any $0\leq t\leq T$ we have
\begin{align*}
\bigg\|\int_0^t
\mathcal{PM}e^{(t-s)\mathcal{Q\M}}
\mathcal{Q\M P}p(s) ds\bigg\|\leq N_0(t),
\end{align*}
where 
\begin{align*}
N_0(t) = \begin{dcases}
tC_4,\quad &\omega_{\Q}=0;\\
\frac{C_4}{\omega_{\Q}}(e^{t\omega_{\Q}}-1)
,\quad &\omega_{\Q}\neq 0;
\end{dcases}
\end{align*}
and $\displaystyle C_{4}=\max_{0\leq s\leq T}\|\P\|\|\M\P\M p(s)\|$.  Moreover, $N(t)$ satisfies $\displaystyle\lim_{t\rightarrow0}N(t)=0$.
\end{theorem}

\begin{proof}
Consider 
\begin{align*}
\bigg\|\int_0^t
\mathcal{PM}e^{(t-s)\mathcal{Q\M}}
\mathcal{Q\M P}p(s)ds\bigg\| & = \bigg\|\int_0^t
\P e^{(t-s)\M\Q}
\M\mathcal{Q\M P}p(s) ds\bigg\|\\
& 
\leq 
C_4 M_{\Q}\int_0^te^{(t-s)\omega_{\Q}}
ds
\\
&=
\begin{dcases}
tC_4,\quad &\omega_{\Q}=0\\
\frac{C_4}{\omega_{\Q}}(e^{t\omega_{\Q}} - 1)
,\quad &\omega_{\Q}\neq 0
\end{dcases}
\end{align*}
where $\displaystyle C_4=\max_{0\leq s\leq T}\|\P\|\|\M\Q\M\P p(s)\|$. 

\end{proof}

\begin{theorem}\label{PDF_t-model}
{\bf (Memory approximation via the $t$-model)}
Let $e^{t\mathcal{MQ}}$ and $e^{t\M}$ be 
strongly continuous semigroups with bounds 
$\|e^{t\M}\|\leq Me^{t\omega}$ and $\|e^{t\M\Q}\|\leq M_{{\Q}}e^{t\omega_{\Q}}$, and let $T>0$ be a fixed integration time. If the function $k(s,t)=\mathcal{P\M  }e^{(t-s)\mathcal{Q\M  }}\mathcal{Q\M  P}p(s)$ (integrand of the memory term) is at least twice differentiable respect to $s$ for all $t\geq 0$, then 
\begin{align*}
\bigg\|\int_0^t
\mathcal{P\M  }e^{(t-s)\mathcal{Q\M  }}
\mathcal{Q\M  P}p(s)ds-t\mathcal{P\M  Q\M  P}p(t)\bigg\|\leq N_{1}(t)
\end{align*}
where $N_1(t)$ is defined as 
\begin{align*}
 N_1(t)=\begin{dcases}
t(M_{\Q}+1)C_4,\quad &\omega_{\Q}=0\\
\frac{C_2M_{\Q}}{\omega_{\Q}}(e^{t\omega_{\Q}}-1)+
tC_4,\quad &\omega_{\Q}\neq 0
\end{dcases}
\end{align*}
and $C_4$ is as in Theorem \ref{PDFMgrowth}.
\end{theorem}

\begin{proof}
\begin{align*}
\bigg\|\int_0^t
\mathcal{PM}e^{(t-s)\mathcal{QM}}
\mathcal{QMP}p(s)ds-t\mathcal{PMQMP}p(t)\bigg\|
&\leq C_4 M_{\Q}\int_0^t e^{(t-s)\omega_{\Q}}ds+C_4t\\
& =\begin{dcases}
t(M_{\Q}+1)C_4,\quad &\omega_{\Q}=0\\
\frac{C_4M_{\Q}}{\omega_{\Q}}(e^{t\omega_{\Q}}-1)+
tC_4,\quad &\omega_{\Q}\neq 0
\end{dcases}
\end{align*}
\end{proof}

\subsection{Hierarchical Memory Approximation in PDF Space}
The hierarchical memory approximation methods we discussed in 
section \ref{sec:hierarchical} can be also developed in the PDF space. 
To this end, let us first define 
\begin{align}
v_0(t)=\int_0^t
\mathcal{P\M  }e^{(t-s)\mathcal{Q\M  }}
\mathcal{Q\M  P}p(s)ds.
\label{memoryPDF}
\end{align}
By repeatedly differentiating $v_0(t)$ with respect to time (assuming 
$v_0(t)$ smooth enough) we obtain the hierarchy of equations
\begin{align*}
\frac{\partial}{\partial t}v_{i-1}(t)=\mathcal{P\M  }(\mathcal{Q\M  })^i\mathcal{P}p(t)
+ v_{i}(t)\qquad i=1,\dots ,n
\end{align*}
where,
\begin{align*}
v_i(t)=\int_{0}^t\mathcal{P\M  }
e^{(t-s)\mathcal{Q\M  }}(\mathcal{Q\M  })^{i+1}\mathcal{P}p(s)ds.
\end{align*}
By following closely the discussion in section \ref{sec:hierarchical} we introduce the hierarchy of memory equations 
\begin{equation}
\begin{cases}
\displaystyle \frac{dv_0^n(t)}{dt} =
\mathcal{P\M  }\mathcal{Q\M  }\mathcal{P}p(t)+v_1^n(t)\vspace{0.2cm}\\
\displaystyle\frac{dv_1^n(t)}{dt} =
\mathcal{P\M  }(\mathcal{Q\M  })^2\mathcal{P}p(t)+v_2^n(t)\\
\hspace{0.5cm}\vdots\\
\displaystyle\frac{dv_{n-1}^n(t)}{dt} =
\mathcal{P\M  }(\mathcal{Q\M  })^n\mathcal{P}p(t)+v_n^{e_n}(t)
\end{cases}
\label{hier_equation_PDF}
\end{equation}
and approximate the last term in such hierarchy in different ways. 
Specifically, we consider  
\begin{align*}
v_n^{e_n}(t)& =\int_{t}^t\mathcal{P\M  }e^{(t-s)\mathcal{Q\M}}
(\mathcal{Q\M})^{n+1}\mathcal{P}p(s)ds=0&&\qquad(\text{$H$-model}),\\
v_n^{e_n}(t)& =\int_{\max(0,t-\Delta t)}^t\mathcal{P\M  }e^{(t-s)\mathcal{Q\M}}
(\mathcal{Q\M})^{n+1}\mathcal{P}p(s)ds&&\qquad(\text{Type \rom{1} Finite Memory Approximation}),\\
v_n^{e_n}(t)& =\int_{\min(t,t_n)}^t\mathcal{P\M  }e^{(t-s)\mathcal{Q\M}}
(\mathcal{Q\M})^{n+1}\mathcal{P}p(s)ds&&\qquad(\text{Type \rom{2} Finite Memory Approximation}),\\
v_n^{en}(t)& =t\P\M(\Q\M)^{n+1}\P p(t)&&\qquad(\text{$H_t$-model}).
\end{align*}
Hereafter we establish the accuracy of the approximation schemes
resulting from the substitution of each $v_n^{e_n}(t)$ above into \eqref{hier_equation_PDF}. 

\begin{theorem}\label{Thm_TA_PDF}
{\bf(Accuracy of the $H$-model)}
Let $e^{t\M}$ and $e^{t\M\Q}$ be strongly continuous semigroups, $T>0$ a fixed integration time, and  
\begin{align}
\beta_i=\frac{\displaystyle \sup_{s\in[0,T]}\|(\M\Q)^{i+1}\M\P p(s)\|}{\displaystyle \sup_{s\in[0,T]}\|(\M\Q)^{i}\M\P p(s)\|}, \qquad 1\leq i\leq n.
\label{condPDF}
\end{align}
Then for $1\leq q\leq n$ we have 
\begin{align*}
\|v_0(t)-v_0^q(t)\|\leq N_2^{q}(t),
\end{align*}
where $$N_2^{q}(t)= A_{2}C_{4}\left(\prod_{i=1}^{q}\beta_{i}\right)\frac{t^{q+1}}{(q+1)!}, $$
$\displaystyle A_2=\max_{s\in[0,T]}e^{s\omega_{\PrjC}} $, and 
$C_4$ is as in Theorem \ref{PDFMgrowth}.
\end{theorem}

\begin{proof}
The error at the $n$-th level can be bounded as 
\begin{align*}
\|v_0(t)-v_0^q(t)\|&\leq
\int_0^t\int_0^{\tau_q}\dots \int_0^{\tau_2}\|\P\M e^{(\tau_1-s)\Q\M}(\Q\M)^{p+1}\P p(s) \|ds d \tau_1\dots  d\tau_q\\
&\leq 
	A_{2}C_{4}\left(\prod_{i=1}^{q}\beta_{i}\right)\frac{t^{q+1}}{(q+1)!},
\end{align*}
where
\begin{align}
	A_{2} & = \max_{s\in[0,T]}e^{s\omega_{\PrjC}} = \begin{cases}1 &  \omega_{\PrjC} \leq 0\\e^{T\omega_{\PrjC}} & \omega_{\PrjC}\geq 0\end{cases}
.
\end{align}
Let 
\begin{align}
\beta_{i} & = \frac{\displaystyle \sup_{s\in[0,T]}\|(\M\Q)^{i+1}\M\P p(s)\|}{\displaystyle\sup_{s\in[0,T]}\|(\M\Q)^{i}\M\P p(s)\|},
\end{align}
under the assumption that these quantities are finite. Then we have 
\begin{align*}
\|v_0(t)-v_0^q(t)\|\leq A_{2}C_{4}\left(\prod_{i=1}^{q}\beta_{i}\right)\frac{t^{q+1}}{(q+1)!}.
\end{align*}

\end{proof}

\begin{corollary}\label{coro_PDF_sequence}
{\bf (Uniform convergence of the $H$-model)}
If $\beta_i$ in Theorem \ref{Thm_TA_PDF}  satisfy
\begin{align*}
\beta_i<\frac{i+1}{T},\quad 1\leq i\leq n
\end{align*}
for any fixed time $T$, then there exits 
a sequence $\delta_1>\delta_2>\dots >\delta_n$ such that 
\begin{align*}
\|v_0(T)-v_0^q(T)\|\leq \delta_q,
\end{align*}
where $1\leq q\leq n $.
\end{corollary}

\begin{corollary}\label{coro_PDF_sequence2}
{\bf (Asymptotic convergence of the $H$-model)}
If $\beta_i$ in Theorem \ref{Thm_TA_PDF} satisfy
\begin{align*}
\beta_i<C, \quad 1\leq i<+\infty
\end{align*}
for some constant $C$, then for any fixed time $T$ and arbitrary $\delta>0$, there exits an integer $q$ such that for all $n>q$,
\begin{align*}
\|v_0(T)-v_0^n(T)\|\leq \delta.
\end{align*}
\end{corollary}

\noindent
The proofs of the Corollary \ref{coro_PDF_sequence} and \ref{coro_PDF_sequence2} closely follow the proofs of Corollary \ref{coro_state_sequence} and \ref{cor_decaying_1}. Therefore 
we omit details here.

\begin{theorem}\label{Thm_type1_PDF}
{\bf (Accuracy of Type-I FMA)}
Let $e^{t\M}$ and $e^{t\M\Q}$ be strongly continuous semigroups, $T>0$ a fixed integration time, and let 
\begin{align}
\beta_{i}=\frac{\displaystyle \sup_{s\in[0,T]}\|(\M  \Q)^{i+1}\M\P p(s)
\|}{\displaystyle \sup_{s\in[0,T]}\|(\M  \Q)^{i}\M\P p(s)\|}, \qquad 1\leq i\leq n.
\label{condPDF1}
\end{align}
Then for $1\leq q\leq n$
\begin{align*}
\|v_0(t)-v_0^q(t)\|\leq N_3^{q}(t),
\end{align*}
where $$\displaystyle N_3^{q}(t)=A_{2}C_{4}\left(\prod_{i=1}^{q}\beta_{i}\right)\frac{(t-\Delta t_q)^{q+1}}{(q+1)!},$$ and 
$C_4$ is as in Theorem \ref{PDFMgrowth}.
\end{theorem}

\begin{proof}
The proof is very similar with the proof of 
Theorem \ref{Thm_type_1}. We begin with the estimate of 
$v_{0}(t) - v_{0}^{q}(t)$
\begin{subequations}
\begin{align*}
 & \int_{0}^{\max(0,t-\Delta t_{q})}\int_{0}^{\tilde{\tau}_{q}}\cdots\int_{0}^{\tilde{\tau}_{2}}\int_{0}^{\tilde{\tau}_{1}}\P e^{(\tilde{\tau}_{1}+\Delta t_{q}-s)\M\Q}(\M\Q)^{q+1}\M\Q p(s)dsd\tilde{\tau}_{1}\cdots d\tilde{\tau}_{q},\\
	& = \begin{dcases}
		0 & 0\leq t\leq \Delta t_{q}\\
		\int_{0}^{t-\Delta t_{q}}\int_{0}^{\sigma}\frac{(t-\Delta t_{q}-\sigma)^{q-1}}{(q-1)!}\P e^{(\sigma+\Delta t_{q}-s)\M\Q}(\M\Q)^{q+1}\M\Q p(s)dsd\sigma & t\geq \Delta t_{q}.
		\end{dcases}
\end{align*}
\end{subequations}
This can  be bounded by following the technique in the proof of Theorem \ref{Thm_TA_PDF}. This yields  
\begin{subequations}
\begin{align}
	\|v_{0}(t) - v_{0}^{q}(t)\|	& \leq \begin{dcases}
	0 & 0\leq t\leq \Delta t_{q}\\
	A_{2}C_{4}\left(\prod_{i=1}^{q}\beta_{i}\right)\frac{(t-\Delta t_{q})^{q+1}}{(q+1)!}& t\geq \Delta t_{q}
	\end{dcases}.
\end{align}
\end{subequations}
\end{proof}

\begin{corollary}
{\bf (Uniform convergence of Type-I FMA)}
If $\beta_i$ in Theorem \ref{Thm_type1_PDF} satisfy 
\begin{align}\label{convergence_condition1pdf}
\beta_i<(i+1)\left[\frac{\delta i!}{ C_4A_2\left(\prod_{k=1}^i\beta_k\right)}\right]^{-\frac{1}{i}}
\end{align} 
then for any $\delta>0$, 
there exists an ordered time 
sequence $\Delta t_n<\Delta t_{n-1}<\dots <\Delta t_1<T$ such that 
\begin{align*}
\|w_0(T)-w_0^{q}(T)\|\leq \delta,\quad 1\leq q\leq n
\end{align*}
and which satisfies
\begin{align*}
\Delta t_q\leq T-\left[\frac{\delta(q+1)!}{ C_4A_2\left(\prod_{i=1}^{q}\beta_i\right)}\right]^{\frac{1}{q+1}}.
\end{align*}
\end{corollary}
\noindent
The proof is very similar with the proof of Corollary \ref{coro_Type1} and therefore we omit it.

\begin{theorem}\label{PDF_finite_memory_theorem}
{\bf (Accuracy of Type-II FMA) }
Let $e^{t\mathcal{\M  }}$ and $e^{t\mathcal{\M  Q}}$ be 
strongly continuous semigroups and $T>0$ a fixed integration time. Set 
\begin{align}
\beta_{i}=\frac{\displaystyle \sup_{s\in [0,T]}\|(\M  \Q)^{i+1}\M\P p(s)
\|}{\displaystyle \sup_{s\in [0,T]}\|(\M  \Q)^{i}\M\P p(s)\|}, \qquad 1\leq i\leq n.
\label{condPDF2}
\end{align}
Then for $1\leq q\leq n$
\begin{align*}
\|v_0(t)-v_0^q(t)\|\leq N_4^{q}(t),
\end{align*}
where $$\displaystyle N_4^{q}(t)=\frac{C_4}{\omega_{\Q}}\big[1-e^{-t_q\omega_{\Q}}\big]\left(\prod_{i=1}^{q}\beta_{i}\right)f_q(\omega_{\Q},t),$$
$f_q(\omega_\Q,t)$ is defined in \eqref{f_p} and $C_4$ is as in Theorem \ref{PDFMgrowth}.
\end{theorem}
\begin{proof}
The proof is very similar with the proof of Theorem \ref{finite_memory_approximation}. Hereafter we provide 
the proof for the case when $\omega_{\Q}>0$. Other 
cases can be easily obtained by using the same 
method. First of all, we have error estimation 
\begin{align*}
	\|v_{0}(t) - v_{0}^{q}(t)\| & \leq \int_{0}^{t}\int_{0}^{\tau_{q}}\cdots \int_{0}^{\tau_{2}}\int_{0}^{t_{q}}\|\Prj \M e^{(\tau_{1}-s)\Q\M }(\Q\M)^{q+1}\P p(s)\|dsd\tau_{1}\cdots d\tau_{q}\\
	& \leq C_{4}\left(\prod_{i=1}^{q}\beta_{i}\right)\int_{0}^{t}\int_{0}^{\tau_{q}}\cdots \int_{0}^{\tau_{2}}\int_{0}^{t_{q}} e^{(\tau_{1}-s)\omega_{\PrjC}}dsd\tau_{1}\cdots d\tau_{q}	\\
	&=
	\frac{C_4}{\omega_{\Q}}\big[1-e^{-t_q\omega_{\Q}}\big]
\left(\prod_{i=1}^{q}\beta_{i}\right)f_q(\omega_{\Q},t),
\end{align*}
where $f_q(\omega_{\Q},t)$ is as in \eqref{f_p}. 

\end{proof}

\begin{corollary}\label{coro_type2_PDF}
{\bf (Uniform convergence of Type-II FMA) }
If $\beta_i$ in Theorem \ref{PDF_finite_memory_theorem} 
satisfy 
\begin{align*}
\begin{dcases}
\beta_i<\frac{i}{T}, \quad &\omega_{\Q}=0 \\
\beta_i< \omega_{\Q}\frac{\displaystyle e^{T\omega_{\Q}}- \sum_{k=0}^{i-2}\frac{(T\omega_{\Q})^k}{k!}}
{\displaystyle e^{T\omega_{\Q}}-\sum_{k=0}^{i-1}\frac{(T\omega_{\Q})^k}{k!}}, \quad &\omega_{\Q}\neq 0
\end{dcases}
\end{align*} 
for all $t\in[0,T]$,  then for arbitrarily 
small $\delta>0$ there exists an ordered time sequence 
$0<t_0<t_1<\dots t_n<T $ such that
\begin{align*}
\|v_0(T )-v_0^{q}(T )\|\leq \delta,\quad 1\leq q\leq n
\end{align*}
which satisfies
\begin{align*}
	t_{q} & \geq 
		\frac{1}{\omega_{\PrjC}}\ln\left[1-\frac{\delta\omega_{\PrjC}^{q+1}}{ \displaystyle C_{4} \left(\prod_{i=1}^{q}\beta_{i}\right)\left[e^{T \omega_{\PrjC}} - \sum_{k=0}^{q-1}\frac{(T \omega_{\PrjC})^{k}}{k!}\right]}\right].
\end{align*}
\end{corollary}

\begin{proof}
To ensure that $\|v_{0}(t) - v_{0}^{q}(t)\| \leq \delta$ for all $0\leq t\leq T$, we can take (for $\omega_\Q>0$) 
\begin{align*}
		\frac{C_4}{\omega_{\Q}}\big[1-e^{-t_q\omega_{\Q}}\big]
	\left(\prod_{i=1}^{q}\beta_{i}\right)f_q(\omega_{\Q},T) =\max_{t\in[0,T]}	\frac{C_2}{\omega_{\Q}}\big[1-e^{-t_q\omega_{\Q}}\big]
	\left(\prod_{i=1}^{q}\beta_{i}\right)f_q(\omega_{\Q},t)  \leq \delta.
\end{align*}
Therefore  
\begin{align*}
	e^{-t_q\omega_{\Q}} & \geq  1-\frac{\delta\omega_{\PrjC}}{\displaystyle C_{4}\left(\prod_{i=1}^{q}\beta_{i}\right)f_{q}(\omega_{\PrjC},T)}\qquad \textrm{i.e., }\\
	t_{q} & \leq 
		\frac{1}{\omega_{\PrjC}}\ln\left[1-\frac{\delta\omega_{\PrjC}^{q+1}}{ \displaystyle C_{4} \left(\prod_{i=1}^{q}\beta_{i}\right)\left[e^{T \omega_{\PrjC}} - \sum_{k=0}^{q-1}\frac{(T \omega_{\PrjC})^{k}}{k!}\right]}\right].
\end{align*}
Since for $\omega_{\Q}>0$, we have have condition
\begin{align*}
\beta_i(t)< \omega_{\Q}\frac{ \displaystyle e^{T\omega_{\Q}}-\sum_{k=0}^{i-2}\frac{(T\omega_{\Q})^k}{k!}}
{\displaystyle e^{T\omega_{\Q}}-\sum_{k=0}^{i-1}\frac{(T\omega_{\Q})^k}{k!}}.
\end{align*}
Thus, there exists an ordered time sequence $0<t_1<\dots <t_n$ such that $\|v_0(T )-v_0^n(T )\|\leq \delta$. As in Theorem \ref{finite_memory_approximation}, this $\delta$-bound on the 
error holds for all $t_{n}$ (with upper bound as above), which implies the existence of such an increasing time sequence $0<t_1<\dots <t_n$ with $t_{n}$ bounded from below by the same quantities.
\end{proof}


\end{document}